\documentclass[a4paper]{article}
\usepackage{amsfonts}
\usepackage{amsmath}
\usepackage{amssymb}
\usepackage{amsthm}
\usepackage{graphicx}
\usepackage{epstopdf}
\usepackage{algorithm}
\usepackage{algpseudocode}
\usepackage{mathtools}
\usepackage{indentfirst}
\usepackage{placeins}
\usepackage{listings}
\usepackage{xcolor}
\usepackage{comment}
\usepackage{url}
\usepackage{soul}
\usepackage{bm}
\usepackage[version=3]{mhchem}
\usepackage{multirow}
\usepackage{enumitem}

\ifpdf
  \DeclareGraphicsExtensions{.eps,.pdf,.png,.jpg}
\else
  \DeclareGraphicsExtensions{.eps}
\fi

\usepackage[pdffitwindow=false,
            plainpages=false,
            pdfpagelabels=true,
            pdfpagemode=UseOutlines,
            pdfpagelayout=SinglePage,
            bookmarks=false,
            colorlinks=true,
            hyperfootnotes=false,
            linkcolor=blue,
            urlcolor=blue!30!black,
            citecolor=green!50!black]{hyperref}

\newtheorem{remark}{Remark}
\newtheorem{assump}{Assumption}
\newtheorem{proposition}{Proposition}
\newtheorem{lemma}{Lemma}
\newtheorem{theorem}{Theorem}

\DeclareMathOperator{\sgn}{sgn}
\DeclareMathOperator*{\argmin}{argmin}

\def\longrightharpoonup{\relbar\joinrel\rightharpoonup}
\def\longleftharpoondown{\leftharpoondown\joinrel\relbar}

\def\longrightleftharpoons{
  \mathop{
    \vcenter{
       \hbox{
       \ooalign{
          \raise1pt\hbox{$\longrightharpoonup\joinrel$}\crcr
  	  \lower1pt\hbox{$\longleftharpoondown\joinrel$}
	}
      }
    }
  }
}

\definecolor{cream}{RGB}{222,217,201}

\allowdisplaybreaks

\makeatletter
\let\@fnsymbol\@arabic
\makeatother

\title{Learning chemical reaction networks from trajectory data}

\author{Wei Zhang\thanks{Zuse Institute Berlin, D-14195 Berlin, Germany. }
\and Stefan Klus\thanks{Department of Mathematics and Computer Science, Freie
Universit\"at Berlin, D-14195 Berlin, Germany. }
\and Tim Conrad\footnotemark[1] \footnotemark[2]
\and Christof Sch\"{u}tte\footnotemark[1] \footnotemark[2]
}

\usepackage{amsopn}

\begin{document}

\maketitle

\begin{abstract}
    We develop a data-driven method to learn chemical reaction networks from
    trajectory data. Modeling the reaction system as a continuous-time Markov
    chain and assuming the system is fully observed, our
    method learns the propensity functions of the system with predetermined
    basis functions by maximizing the likelihood function of the trajectory
    data under $l^1$ sparse regularization. We demonstrate our method with
    numerical examples using synthetic data and carry out an asymptotic analysis of the proposed learning procedure in the infinite-data limit.
\end{abstract}


\begin{ams}
  92C42, 62M86
\end{ams}

\section{Introduction}
\label{sec-intro}

Chemical reaction networks~\cite{ssck_gillespie,Anderson2011} have been shown to be very useful in studying dynamical
processes in chemistry and biology, where systems under investigation
typically contain many different reactants that interact with each
other.
In in-silico biology, for instance, the cellular processes are often modeled as chemical reaction networks,   
which take the relevant biological/chemical components as well as their
interactions into
account~\cite{Gunawardena_chemicalreaction,nature-network-biology-barabasi,SRIVASTAVA2002,formalism-in-system-biology-machado2011}.
Modeling cellular processes, or finding the kinetic structure of the
underlying reaction networks~\cite{wilkinson2011stochastic,compare-reverse-methods-silico-network,efficient-infer-ssystems-alternating-regression,how-to-infer-gene-network-revisited,
reverse-engineering-in-system-biology-2014,sparse-regression-from-single-cell-2016}, is one of the most prominent fields of in-silico biology due to the important
role of such models in understanding the cellular behavior. This task is particularly
challenging for realistic reaction networks that are characterized by a large number of elements
and interactions (reactions). At the same time, more and more trajectory data of
cellular processes is becoming available due to state-of-the-art single-cell
based laboratory techniques~\cite{scRNAsq-2019,single-cell-will-revolutionize}.

The aim of this work is to develop data-driven methods~\cite{kutz2013data-book} that allow us to
learn chemical reaction networks from trajectory data and to apply the
new methods to the modeling of cellular processes. Given trajectory data
of a stochastic chemical reaction process, 
we propose a numerical approach to reconstruct the underlying reaction network
by maximizing the likelihood function of the trajectory with sparsity regularization.
Roughly speaking, our approach consists of three steps. In the first step, preliminary information of the reaction network such as the number of different elements
(reactant, products) and the total number of reaction channels is extracted from trajectory data by counting and enumerating.
Based on this information, the second step is to define several basis functions which will be used in learning the propensity functions of the reaction network. 
The theory of chemical reactions suggests that we can choose each basis
function as the product of copy-numbers of at most two different
reactants~\cite{ANGELI2009,erdi1989mathematical}, i.e., polynomial functions of degree up to $2$.
In the third step, the propensity function of each reaction channel is represented using linear combinations of the basis functions involving unknown
coefficients, which are then determined by maximizing the log-likelihood
function of the trajectory data along with sparse regularization techniques using the $l^1$-norm~\cite{hastie-Tibshirani-book-sparsity}. 

In contrast to Lasso~\cite{lasso1996,lasso-retrospective2011}, the optimization problem that needs to be solved in our
learning approach is a nonlinear sparse optimization problem, due to the nonlinearity of the log-likelihood function of the reaction network.
In our study, we find that FISTA (Fast Iterative Shrinkage-Thresholding
Algorithm) proposed in~\cite{fista2009} is a suitable algorithm for solving our problem.
We also propose a simple preconditioning technique which can significantly
improve the performance of the numerical algorithm by allowing larger step-sizes in FISTA.
This preconditioning technique turns out to be particularly useful when the basis functions
take values at different orders of magnitudes for the given trajectory data.
Furthermore, we provide an asymptotic analysis of our learning approach in the infinite-data limit.
Under certain technical assumptions, by applying large sample theory~\cite{ferguson1996course,lehmann2004elements,van2000asymptotic} and limit theorems for stochastic processes~\cite{ethier1986markov}, 
we establish the asymptotic consistency and the asymptotic normality of the
estimators in our learning procedure, which therefore provides a solid
theoretical basis for the data-driven method proposed in this paper.

Let us first review related work and summarize the contributions of this paper. 
The reconstruction of the governing equations from data using sparsity constraints is getting more and more attention, see~\cite{predict-catastrophes-compress-sensing,Brunton-sindy,inferring-bio-network-by-sparse-identy-2016,efficient-infer-ssystems-alternating-regression}
for methods pertaining to ordinary differential equations (ODEs) and \cite{sindy-sde} for learning stochastic differential equations (SDEs).
For chemical and biological reaction systems, the problem of estimating unknown parameters
has been well studied when the systems are modeled both as
ODEs~\cite{HUG2013-high-dim-param-estimation,parameter-estimation-sys-biology}
and as continuous-time Markov chain processes~\cite{Anderson2011,reinker-param-estimation-stochastic,Boys2008,wilkinson2011stochastic},
while the reconstruction of the entire chemical reaction networks, i.e., finding parsimonious models, has only
been considered when the systems are modeled as ODE systems~\cite{reverse-engineering-in-system-biology-2014,inferring-bio-network-by-sparse-identy-2016,efficient-infer-ssystems-alternating-regression}.
We refer the readers to the nice review~\cite{reverse-engineering-in-system-biology-2014}
for recent developments on the reverse engineering in systems biology.
Compared to the aforementioned existing results, our work is new in the
following three aspects.
Firstly, we study sparse reconstruction of chemical reaction networks as continuous-time Markov chains, which, to the best of our knowledge, has not been considered in the literature.
In contrast to ODE models, a continuous-time Markov chain as a stochastic model
has the ability to provide more details of the reaction systems by capturing
stochastic effects, which are known to be important for cellular processes~\cite{survey-methods-for-modeling-and-analyzing,Swain-stochasticity,Kar-explore-the-noise}.
Secondly, we have developed numerical codes in which we implemented the FISTA
method~\cite{fista2009} to solve a nonlinear sparse optimization problem in order to learn the reaction networks from trajectory data. 
Our numerical approaches, in particular the preconditioning technique, may be useful in other sparse optimization problems as well.
Thirdly, we provide a theoretical justification of the proposed data-driven method.
Note that, although different data-driven methods using sparsity~\cite{predict-catastrophes-compress-sensing,Brunton-sindy,sindy-sde} have been developed 
in the literature for different types of dynamical systems, the theoretical
analysis of these methods is largely incomplete (see~\cite{exact-recovery-chaotic-2017}).
We expect the theoretical analysis presented in the current work to shed light on the characteristic properties of other data-driven methods as well.

Before concluding this introduction, we discuss several issues that will not be studied in detail in the current paper.
Most importantly, we assume that the dynamics of the system is fully observed.
In applications, it may be the case that either only certain ``important'' species
in the system are observed or the full dynamics is only discretely observed at a fixed observation frequency~\cite{Boys2008}.
In the former case, one can still apply the sparse learning approach proposed here and the output will be an ``effective'' model for the observed ``important'' species. 
However, the theoretical asymptotic analysis does not carry over directly, and it is therefore important to assess the quality of the effective model provided by the learning approach.
In the latter case, where the full dynamics is observed discretely, learning
the parsimonious model becomes more challenging. First of all, since not all
reactions are observed, the reaction channels of the system need to be identified by other means.
Supposing that this can be done, the likelihood function of the given
trajectory data can be obtained by summing up the likelihood of all possible underlying trajectories that are consistent with the observation data.
One can formulate the learning approach again as a sparse minimization
problem, but it will be necessary to sample the underlying trajectories of the
system in order to evaluate both the likelihood function and its derivatives.
This results in some difficulties when solving the sparse minimization problem. 
We will address these issues in future work.

The remainder of the paper is organized as follows. In Section~\ref{sec-forward}, we introduce chemical reaction networks and the required notation. Learning chemical reaction networks from trajectory data and its formulation as an optimization problem will be considered in Section~\ref{sec-inverse}. 
In Section~\ref{sec-example}, we demonstrate the efficiency of
the numerical algorithm for solving the (sparse) optimization problems with three concrete numerical examples.
In Section~\ref{sec-asymptotics}, we analyze the learning tasks when the
length of the trajectory data goes to infinity and study the asymptotic behavior of the solutions of the optimization problems.
Appendix~\ref{app-0} summarizes the main steps of the algorithm FISTA. 
Appendix~\ref{app-1} contains properties of an elementary function used in the
current work. 
Two useful limit lemmas of counting processes are summarized in Appendix~\ref{app-2}.
Finally, the proofs of results in Section~\ref{sec-asymptotics} are collected in Appendix~\ref{app-3}.

The code used for producing the numerical results in Section~\ref{sec-example} is available at: \url{https://github.com/zwpku/sparse-learning-CRN}.

\section{Chemical reaction networks as continuous-time Markov chains: forward problem}
\label{sec-forward}
Chemical reaction networks consist of different chemical species that can interact with each other through independent chemical reactions. 
Suppose the system has $n$ different chemical species, denoted by $S_1,
S_2, \dots, S_n$. Each species $S_i$, $1 \le i \le n$, has $x^{(i)}$ copies, where the copy-number $x^{(i)} \ge 0$ may change whenever a reaction involving the species $S_i$ has occurred.
The state of the system can be represented as the vector 
\begingroup
\setlength\abovedisplayskip{6pt}
\setlength\belowdisplayskip{6pt}
\begin{align*}
x = (x^{(1)}, x^{(2)}, \dots, x^{(n)})^\top \in \mathbb{X}\subseteq \mathbb{N}^n\,,
\end{align*}
\endgroup
where $\mathbb{N}=\{0, 1, 2,\dots\}$ and $\mathbb{X}$ is the set of all possible states of the system. 

The evolution of the system's state $x$ can be modeled as a state-dependent continuous-time Markov chain~\cite{Anderson2011,ssck_gillespie}. 
Let $\mathcal{R}$ denote a reaction in the system. The state change vector $v$
of $\mathcal{R}$, $v \in \mathbb{N}^n$, is defined such that, starting in state $x$, the state of the system will change to $x+v$ when the reaction $\mathcal{R}$ occurs. 
The waiting time $\tau_{\mathcal{R}}$ of the system before the reaction $\mathcal{R}$ occurs satisfies an exponential distribution with the rate
parameter $a^*_{\mathcal{R}}(x)$ (propensity function), which in turn depends on both the state $x$ and the structure of $\mathcal{R}$. 
Specifically, the probability density function of $\tau_{\mathcal{R}}$ is given by
\begingroup
\setlength\abovedisplayskip{6pt}
\setlength\belowdisplayskip{6pt}
\begin{align*}
  \psi^*_\mathcal{R}(t\,|\,x) = a^*_{\mathcal{R}}(x)\,\exp(- a^*_{\mathcal{R}}(x) t)\,,
  \qquad t \ge 0\,.
\end{align*}
\endgroup
In Table~\ref{propensity-table}, we list the propensity functions of reactions
which consume at most two molecules (see~\cite{ball06,kang2013} for further details). In particular, note that the propensity functions for the reactions in
Table~\ref{propensity-table} are polynomial functions whose degrees are less or equal to $2$. 

\begin{table}
    \caption{
For different types of chemical reactions, propensity function $a_{\mathcal{R}}^*(x)$ as a function of system's state
      $x = (x^{(1)}, \dots, x^{(n)})^\top$ is given by law of mass-action. $V$ is a constant related to either the volume or the total
      number of molecules in the system and $\kappa$ denotes the rate constants of chemical reactions. 
      \label{propensity-table}}
    \centering
    \scalebox{0.9}{
    \begin{tabular}{cll}
      No. & Reaction $\mathcal{R}$ & $a_{\mathcal{R}}^*(x)$ \\
      \hline
      $1$ & \ce{$\emptyset$ ->[\kappa] products} &   $\kappa V$ \\
      $2$ & \ce{$S_i$ ->[\kappa] products} &   $\kappa x^{(i)}$ \\
      $3$ & \ce{2$S_i$ ->[\kappa] products} &  $\frac{\kappa}{V} x^{(i)}(x^{(i)} - 1)$ \\
      $4$ & \ce{$S_i$ + $S_j$ ->[\kappa] products} & $\frac{\kappa}{V} x^{(i)}x^{(j)}$ 
    \end{tabular}}
\end{table}

In many reaction systems, different chemical reactions may have the same state change vector $v$ (see the second example in Remark~\ref{rmk-2-simple-examples}).
Assume that $N$ chemical reactions $\mathcal{R}_1$, $\mathcal{R}_2$, $\dots$,
$\mathcal{R}_N$ are involved in the evolution of the system
and these $N$ reactions have in total $K$ different state change vectors $v_1, v_2, \dots, v_K$, where $K\le N$. For each $v_i$, $1 \le i \le K$, we introduce the terminology \textit{chemical
channel} $\mathcal{C}_i$. We say the reaction $\mathcal{R}$ belongs to
the channel $\mathcal{C}_i$, or $\mathcal{C}_i$ contains the
reaction $\mathcal{R}$, if the state change vector of $\mathcal{R}$ is $v_i$.
For each $\mathcal{C}_i$, we also define the index set 
\begin{align*}
  \mathcal{I}_i = \Big\{j\, \Big|\, 1 \le j \le N,~\mathcal{R}_j~\mbox{belongs
  to the channel}~ \mathcal{C}_i\Big\}\,,
\end{align*}
and let $N_i$ be the number of chemical reactions belonging to $\mathcal{C}_i$, i.e., $N_i = |\mathcal{I}_i|$.
Clearly, these index sets satisfy
  $\bigcup_{i=1}^K \mathcal{I}_i = \big\{1, 2, \dots, N\big\}$, $\mathcal{I}_i \bigcap \mathcal{I}_{i'} = \emptyset$, if $i\neq i'$\,, 
and therefore $\sum\limits_{i=1}^{K} N_i = N$.

A reaction channel $\mathcal{C}_i$ is said to be activated when a certain reaction $\mathcal{R}$ belonging to $\mathcal{C}_i$ occurs. 
For each $1 \le i \le K$,  $\tau_{i} = \min\limits_{j \in \mathcal{I}_i}
\tau_{\mathcal{R}_j}$ is the waiting time at a state $x$ before the activation of the
channel $\mathcal{C}_i$,  while $\tau = \min\limits_{1 \le j \le N} \tau_{\mathcal{R}_j}$
is the waiting time before any of the chemical reactions in the system occurs.
Assuming the chemical reactions are independent of each
other and the waiting times $\tau_{\mathcal{R}_j}$ follow exponential
distributions, we know that the waiting times $\tau_i$ and $\tau$ also follow exponential
distributions, with the propensity functions 
{ \small
\setlength\abovedisplayskip{5pt}
\setlength\belowdisplayskip{5pt}
\begin{align}
  a^*_i(x) = \sum_{j \in \mathcal{I}_i} a^*_{\mathcal{R}_j}(x)\,,\quad 
  a^*(x) = \sum_{i=1}^K a^*_i(x) = \sum_{j=1}^N a^*_{\mathcal{R}_j}(x)\,,
  \label{a-sum-ai-true}
\end{align}}
respectively. In particular, let $\psi^*(t\,;\,x)$ be the probability density function of $\tau$
and $p^*(i\,;\,x)$ the probability that $\mathcal{C}_i$ is the first
channel which becomes activated at state $x$, then 
\begin{equation}
  \begin{split}
    \psi^*(t\,;\,x) &= a^*(x) \exp\big(-a^*(x)\,t\big)\,, \quad t \ge 0\,,\\
  p^*(i\,;\,x) &= \frac{a^*_i(x)}{a^*(x)}\,, \quad 1 \le i \le K\,.
  \end{split}
  \label{phi-p-density-true}
\end{equation}

We point out that the evolution equation of the dynamics described above
(continuous-time Markov chains) can be expressed in a simple form.
In fact, denoting $X(t) \in \mathbb{N}^n$ the state of the system at time
$t \ge 0$, from~\cite{Anderson2011} we know that $X(t)$ satisfies the dynamical equation
{\small
\setlength\abovedisplayskip{5pt}
\setlength\belowdisplayskip{5pt}
\begin{align}
  X(t) = X(0) + \sum_{i=1}^K \mathcal{P}_i\Big(\int_0^t
  a^*_i(X(s))\,ds\Big)v_i\,,\quad t \ge 0\,,
  \label{eqn-of-xt}
\end{align}
}
where $\mathcal{P}_i$, $ i = 1, \dots, K $, are independent unit Poisson processes. 

\begin{remark}
As concrete examples, let us consider two simple reaction networks. 
  \begin{enumerate}[wide]
    \item Reactions $\ce{$A+B$ ->[\kappa_1] 2B}\,, \, \ce{$B$ ->[\kappa_2] A}$,
      with rate constants $\kappa_1$, $\kappa_2$. In this case, we
      have two different reactions ($N=2$) and two different reaction channels ($K=2$), with state change vectors
$v_1 = (-1, 1)^\top$ and $v_2 = (1,-1)^\top$ respectively. 
According to Table~\ref{propensity-table}, the propensity functions of these two channels (assuming $V=1$)
are $a_1^*(x) = \kappa_1\,x^{(1)} x^{(2)}$,\, $a_2^*(x) = \kappa_2\, x^{(2)}$.
\item 
    Reactions $\ce{$A$ + $B$ ->[\kappa_1] $B$}\,,\, \ce{$A$ ->[\kappa_2] $\emptyset$}$, with rate constants $\kappa_1$, $\kappa_2$.
In this case, we have $N=2$, $K=1$, since the state change vector of both reactions is $v=(-1,0)^\top$. 
      The propensity functions of the two reactions $\mathcal{R}_1$, $\mathcal{R}_2$ (assuming $V=1$)
are $a_{\mathcal{R}_1}^*(x) = \kappa_1\,x^{(1)} x^{(2)}$ and
      $a_{\mathcal{R}_2}^*(x) = \kappa_2\, x^{(1)}$, while the 
      propensity function of the channel $v$ is $a_1^*(x) = a_{\mathcal{R}_1}^*(x) + a_{\mathcal{R}_2}^*(x) = \kappa_1\,x^{(1)} x^{(2)} + \kappa_2\, x^{(1)}$.
  \end{enumerate}
  \label{rmk-2-simple-examples}
\end{remark}

\section{Learning chemical reaction networks: inverse problem}
\label{sec-inverse}
In this section, we study the problem of learning chemical reaction networks from trajectory data. 
Depending on the information known about the chemical reaction networks, we consider two different learning
tasks in Subsection~\ref{subsec-learn-rate} and Subsection~\ref{subsec-unkown-structure}, where the second task is the main focus of this paper. 
In both tasks, the propensity functions in \eqref{a-sum-ai-true} are
determined by maximizing the log-likelihood function among the parameterized propensity
functions which depend on both a set of basis functions and several parameters.
To emphasize the dependence on parameters, let the parameterized propensity functions be denoted by $a_i(x\,;\,\bm{\omega})$ and $a(x\,;\,\bm{\omega})$,
respectively, where $x\in \mathbb{X}$ and $\bm{\omega}$ is the vector consisting of all parameters. 
Similar to \eqref{phi-p-density-true}, we define the probability (density)
functions corresponding to $\bm{\omega}$
\begingroup
\small
\setlength\abovedisplayskip{6pt}
\setlength\belowdisplayskip{6pt}
\begin{align}
  \begin{split}
    \psi(t\,;x, \bm{\omega}) &= a(x\,;\bm{\omega}) \exp\big(-
    a(x\,;\,\bm{\omega}) t\big)\,, \quad t \ge 0\,,\\
  p(i\,;x,\bm{\omega}) &=
  \frac{a_i(x\,;\bm{\omega})}{a(x\,;\bm{\omega})}\,, \quad 1 \le i \le K\,.
  \end{split}
  \label{phi-p-density}
\end{align}
\endgroup

In the first learning task (Subsection~\ref{subsec-learn-rate}), we assume
that the structure of the chemical reactions is known and the goal is to determine the reaction rate constant of each reaction,
i.e., the constants $\kappa$ in Table~\ref{propensity-table}. In this case,
each basis function in the parameterized propensity functions
corresponds to an actual chemical reaction that is indeed involved in the
evolution of the system (no redundancy), while the task is to determine the
value of each parameter (parameter estimation) by maximizing the log-likelihood function. 
This is indeed a standard problem and has been widely studied in the
literature. We include it in this section due to its connections to the sparse learning task considered in Subsection~\ref{subsec-unkown-structure}.

In the second learning task (Subsection~\ref{subsec-unkown-structure}), on the other hand, we assume that the structure of the chemical reactions in the system is also
unknown.  In this case, candidate basis functions are chosen to parameterize the
propensity functions, and $l^1$ sparsity regularization is used to remove the redundancy in the basis functions.

Before introducing the two learning tasks, we briefly discuss the trajectory of the system and derive the likelihood function of a given
trajectory.

\subsection{Space of trajectories and the likelihood function}
\label{subsec-space-trajectory}
Given $T>0$, there are two different ways to represent the trajectories of the system in the interval $[0, T]$. 
The first representation relies on the total number $M$ of reactions occurred within $[0,
T]$, the waiting time $\tau$ of each reaction, and the new state of the system after each of the $M$ reactions. Specifically, starting
from a state $y_0 \in \mathbb{X}$ at time $s=0$, each trajectory $X(s)$ in the time $[0, T]$ can be represented as a sequence 
\begingroup
\setlength\abovedisplayskip{6pt}
\setlength\belowdisplayskip{6pt}
\begin{align}
    (y_0, t_0)\,,\,(y_1, t_1)\,,\, (y_2, t_2)\,, \dots, \,(y_M, t_M)\,,
    \label{y-t-sequence}
\end{align}
\endgroup
which means that, starting from $y_0$, the state of the system changes from 
$y_l$ to $y_{l+1}$ after waiting for a period of time of length $t_l$, where $0 \le l < M$.
The final time $t_M$ in \eqref{y-t-sequence} is the amount of time that the system spends at the final state $y_M$
before time $s = T$. Clearly, we have $\sum\limits_{l=0}^{M} t_l = T$.
In the second representation, the indices of the reaction channels are used instead of the new state after each reaction. That is, we represent the same
trajectory $X(s)$, $s \in [0, T]$, as 
\begin{align}
  (i_0, t_0)\,,\,(i_1, t_1)\,,\, (i_2, t_2)\,, \dots, \,(i_{M-1}, t_{M-1})\,,
  \label{i-t-sequence}
\end{align}
where, for each $0 \le l < M$, $i_l \in \{1, 2, \dots, K\}$
denotes the index of the reaction channel and 
$t_l>0$ is the waiting time before the $(l+1)$-th reaction occurs, respectively. 
The two representations \eqref{y-t-sequence} and \eqref{i-t-sequence}
can be converted from one to the other, using the relation $v_{i_l} = y_{l+1} - y_l$,
which holds for $0 \le l < M$. 

In this work, we assume that a trajectory $X(s)$ of the system, represented
either as described in \eqref{y-t-sequence} or \eqref{i-t-sequence}, is available up to time $T$.
In other words, we assume that both the change of the state and the length
of the waiting time are known for each occurrence of the $M$ chemical reactions. 
From the trajectory data, we can deduce the total number of different reaction channels $K$, as well as the state change vector $v_i \in \mathbb{N}^n$
for each channel $\mathcal{C}_i$, $1 \le i \le K$. 
(Note, however, that when a certain channel $\mathcal{C}$ contains more than one reaction, from the data alone we will not be able to tell which reaction $\mathcal{R}$
  belonging to $\mathcal{C}$ has actually occurred when $\mathcal{C}$ is activated.)
For each $1 \le i \le K$, we denote by 
\begin{align}
0 \le l^{(i)}_1 <  l^{(i)}_2< \cdots < l^{(i)}_{M_i} < M\,,
  \label{l-sequqnece-for-channel-i}
\end{align}  
the indices $l$ such that $i_l = i$ in \eqref{i-t-sequence}, where $M_i \ge 0$ is the total number of times that the channel $\mathcal{C}_i$ has been activated within time $[0,T]$, and
  therefore the relation 
  {\small
  \begin{align}
    \sum_{i=1}^K M_i = M
    \label{sum-of-mi}
  \end{align}
  }
is satisfied. For brevity, let us introduce the notation
{\small
\begin{align}
    \mathbf{X}=\Big(M, (y_l, t_l)_{l=0,1,\dots,M}\Big)
    \label{traj-x}
\end{align}
}
to describe the trajectory of the system within the time interval $[0, T]$. 
The space consisting of all trajectories of the system on $[0,T]$ will be denoted by $\mathcal{D}_T$.
Note that, as a random variable, $\mathbf{X}$ contains both continuous and discrete components. 
Given a parameter vector $\bm{\omega}$, we consider the
chemical reaction system determined by the (parameterized) probability density
functions $\psi$, $p$ in \eqref{phi-p-density}, and define 
{\small
\begin{align}
\rho^{(T)}(\mathbf{X}\,|\,\bm{\omega}) = 
  \bigg[\prod_{l=0}^{M-1} \psi(t_l\,;\,y_l,\bm{\omega})\,p(i_l\,;\,y_l,
  \bm{\omega})\bigg] \exp\Big(-a(y_M\,;\, \bm{\omega})\,t_M\Big) \,,
    \label{f-t-pdf}
\end{align}}
  for the trajectory $\mathbf{X}$ in \eqref{traj-x}.
Let $\mathbf{E}$ denote the mathematical expectation with respect to the trajectories of the system. Then, for any bounded measurable function $g \colon \mathcal{D}_T\rightarrow \mathbb{R}$, we have
{\small
\begin{align}
  \mathbf{E} \, g(\mathbf{X})   = 
  \sum_{M=0}^{+\infty}
  \sum_{i_0=1}^{K}\sum_{i_1=1}^{K}\cdots\sum_{i_{M-1}=1}^K
  \int_{\big\{t_0+t_1+\cdots +t_M=T\big\}}\, g(\mathbf{X})\,
    \rho^{(T)}(\mathbf{X}\,|\,\bm{\omega})~ dt_0\,\cdots\,dt_{M-1}\,,
  \label{expectation-path-int}
\end{align}}
from which we can view the function $\rho^{(T)}(\mathbf{X}\,|\,\bm{\omega})$
as the probability density (distribution) of $\mathbf{X}$ on the space
$\mathcal{D}_T$ (we can indeed verify that $\mathbf{E} 1 = 1$). To simplify the notation, we will formally write 
\begingroup
\small
\setlength\abovedisplayskip{6pt}
\setlength\belowdisplayskip{6pt}
\begin{align}
\mathbf{E} \, g(\mathbf{X}) =
  \int_{\mathcal{D}_T} g(\mathbf{X})\, \rho^{(T)}(\mathbf{X}\,|\,\bm{\omega})\,
  d\mathbf{X} 
  \label{e-g-rho}
\end{align}
\endgroup
as the integration on the right-hand side of \eqref{expectation-path-int}.  Using \eqref{f-t-pdf} and \eqref{e-g-rho}, we can write down the likelihood
  function of the trajectory data as 
  {\small
\begin{align}
  \begin{split}
    \mathcal{L}^{(T)}(\bm{\omega}) &= 
    \mathcal{L}^{(T)}\big(\bm{\omega}\,\big|\,\mathbf{X}\big)\\
    &= \rho^{(T)}(\mathbf{X}\,|\,\bm{\omega}) \\
    &=\bigg[\prod_{l=0}^{M-1} \psi(t_l\,;\,y_l,\bm{\omega})\,p(i_l\,;\,y_l,
    \bm{\omega})\bigg] \exp\Big(-a(y_M\,;\, \bm{\omega})\,t_M\Big) \\
    &= \bigg[\prod_{l=0}^{M} \exp\Big(-a(y_l\,;\,\bm{\omega}) t_l\Big)\bigg] 
    \prod_{l=0}^{M-1} a_{i_l}(y_l\,;\,\bm{\omega}) \\
    &= \prod_{i=1}^K \mathcal{L}_i^{(T)}(\bm{\omega})\,,
  \end{split}
  \label{likelihood-l-y-t}
\end{align}
}
where 
{\small
\begin{align}
    \mathcal{L}_i^{(T)}(\bm{\omega}) = 
\bigg[\prod_{l=0}^{M} \exp\Big(-a_i(y_l\,;\,\bm{\omega}) t_l\Big)\bigg] 
\prod_{k=1}^{M_i} a_{i}(y_{l_k^{(i)}}\,;\,\bm{\omega})\,,  \qquad  1 \le i \le K\,,
\label{likelihood-ith}
\end{align}
}
can be considered as the likelihood function along the reaction channel $\mathcal{C}_i$.
\subsection{Learning task $\textbf{1}$: determine rate constants by maximizing the log-likelihood}
\label{subsec-learn-rate}
Assuming that the structure of the chemical reactions of the system is
known, we now consider the problem of determining the reaction rate
constant of each reaction. Note that the propensity function of each reaction $\mathcal{R}$ in Table~\ref{propensity-table} can be written as $\omega
\varphi(x)$, where $\varphi(x)$ is a polynomial of the system's state whose
specific form depends on the structure of $\mathcal{R}$, and $\omega$ is the rate constant.
Therefore, in the current learning task we assume that the propensity function of the $j$th chemical
reaction $\mathcal{R}_j$ in the system is given by 
\begin{align}
  a^*_{\mathcal{R}_j}(x) = \omega_j \varphi_j(x)\,, \quad 1 \le j \le N\,,
  \label{ar-expression}
\end{align}
where the nonnegative function $\varphi_j$ is known from the structure of
$\mathcal{R}_j$, and $\omega_j$ is the unknown rate
constant which we want to determine from trajectory data.

Let $\bm{\omega}$ be the vector 
\begin{align}
  \bm{\omega} = (\omega_{1}, \omega_{2}\,,\dots\,, \omega_{N})^\top
  \in \mathbb{R}^N\,,
  \label{omega-vector}
\end{align}
consisting of all the unknown rate constants, where $\omega_j \ge 0 $ for all $ 1 \le j \le N$.
For each channel~$\mathcal{C}_i$, $1 \le i \le K$, we also define the vector 
  \begingroup
\setlength\abovedisplayskip{6pt}
\setlength\belowdisplayskip{6pt}
\begin{align*}
  \bm{\omega}^{(i)} = (\omega_{j_1}, \omega_{j_2}, \dots,
  \omega_{j_{N_i}})^\top\,,\qquad \mbox{where}\quad \mathcal{I}_i = \{j_1, j_2, \dots,
  j_{N_i}\}\,,
\end{align*}
\endgroup
which consists of the rate constants of reactions belonging to $\mathcal{C}_i$.
Corresponding to \eqref{ar-expression}, the parameterized propensity functions in \eqref{a-sum-ai-true} are
\begingroup
\small
\setlength\abovedisplayskip{6pt}
\setlength\belowdisplayskip{6pt}
\begin{align}
  \begin{split}
    a_i\big(x\,;\bm{\omega}\big) &= a_i\big(x\,;\bm{\omega}^{(i)}\big) = \sum_{j\in \mathcal{I}_i} \omega_{j}
  \varphi_{j}(x)\,, 
    \quad 1 \le i \le K\,,\\[-2pt]
    \mbox{and}\quad a\big(x\,;\bm{\omega}\big) &= \sum_{j=1}^N \omega_{j} \varphi_{j}(x)\,,
  \end{split}
  \label{ai-omega}
\end{align}
\endgroup
while the optimal value of $\bm{\omega}$ is determined by maximizing the (logarithmic) likelihood functions in
\eqref{likelihood-l-y-t}, or equivalently, by solving the minimization problem 
\begin{align}
  \min_{\bm{\omega}} 
  \Big[-\ln\mathcal{L}^{(T)}(\bm{\omega})\Big]\,.
    \label{opt-problem-y-t}
\end{align}
With the trajectory data as defined in \eqref{y-t-sequence} and using the propensity functions in
\eqref{ai-omega}, the objective function above can be computed explicitly and we have 
\begingroup
\small
\begin{align}
  \begin{split}
    \ln\mathcal{L}^{(T)}(\bm{\omega})
    &=  -\sum_{l=0}^{M-1} \ln\bigg[\sum_{j\in \mathcal{I}_{i_l}}
    \omega_{j}\,\varphi_{j}(y_l)\bigg]
    + \sum_{l=0}^{M} t_l\bigg[ \sum_{j=1}^{N} \omega_{j}\,\varphi_{j}(y_l)\bigg]\,\\
    &=-\sum_{i=1}^{K}\sum_{k=1}^{M_i} \ln\bigg[\sum_{j\in\mathcal{I}_i}
    \omega_{j}\,\varphi_{j}(y_{l^{(i)}_k})\bigg]
    + \sum_{l=0}^{M} t_l\bigg[ \sum_{j=1}^N
    \omega_{j}\,\varphi_{j}(y_l)\bigg] \\
    &= -\sum_{i=1}^K \ln \mathcal{L}_i^{(T)}(\bm{\omega}^{(i)})\,.
  \end{split}
  \label{opt-problem-y-t-f}
\end{align}
\endgroup
In the above, we recall that the indices $l^{(i)}_k$ are defined in
\eqref{l-sequqnece-for-channel-i}, the logarithmic likelihood function
\begingroup
\small
\begin{align}
    \ln \mathcal{L}_i^{(T)}(\bm{\omega}^{(i)}) 
    =  \sum_{k=1}^{M_i} \ln\bigg[\sum_{j\in\mathcal{I}_i}
    \omega_{j}\,\varphi_{j}(y_{l^{(i)}_k})\bigg]
    - \sum_{l=0}^{M} t_l\bigg[ \sum_{j\in \mathcal{I}_i}
    \omega_{j}\,\varphi_{j}(y_l)\bigg] 
\label{log-likelihood-ith-omega}
\end{align}
\endgroup
only depends on $\bm{\omega}^{(i)}$ and should be compared to \eqref{likelihood-ith}. 
Note that the expressions above also imply that the minimization problem
\eqref{opt-problem-y-t} can be decomposed into $K$ minimization problems
\begin{align*}
  \min_{\bm{\omega}^{(i)}} 
  \Big[-\ln\mathcal{L}_i^{(T)}(\bm{\omega}^{(i)})\Big]\,, \qquad 1 \le i \le K\,,
\end{align*}
which can be solved separately.

For each index $j$, $1 \le j \le N$, such that $j \in \mathcal{I}_i$ for some $1 \le i \le K$, 
    the corresponding Euler--Lagrange equation of \eqref{opt-problem-y-t} is 
    \begingroup
    \small
\begin{align}
  \mathcal{M}^{(T)}_{j}(\bm{\omega}) = \frac{\partial
  \big(-\ln\mathcal{L}^{(T)}\big)}{\partial \omega_{j}}(\bm{\omega}) =  -\sum_{k=1}^{M_i}
  \frac{\varphi_{j}(y_{l^{(i)}_k})}{\sum\limits_{j'\in \mathcal{I}_i}
  \omega_{j'}\,\varphi_{j'}(y_{l^{(i)}_k})}
    + \sum_{l=0}^{M} t_l \, \varphi_{j}(y_l) = 0\,.
    \label{opt-problem-euler-lagrange}
\end{align} 
\endgroup
Differentiating one more time, we get the Hessian matrix of the objective function
in \eqref{opt-problem-y-t}
\begingroup
\small
\begin{align}
	\frac{\partial^2 \big(-\ln\mathcal{L}^{(T)}\big)}{\partial \omega_{j}\partial \omega_{j'}}(\bm{\omega}) = 
	\frac{\partial \mathcal{M}^{(T)}_{j}}{\partial \omega_{j'}}(\bm{\omega}) =
	\begin{cases}
\sum\limits_{k=1}^{M_i}
\frac{\varphi_{j}(y_{l^{(i)}_k})\,\varphi_{j'}(y_{l^{(i)}_k})}{\big(\sum\limits_{r\in\mathcal{I}_i}
	  \omega_{r} \varphi_{r}(y_{l^{(i)}_k})\big)^2}\,, &\quad  \mbox{if}~j,\, j' \in \mathcal{I}_i\,,\\ 
	  \,0\,, &\quad\mbox{otherwise}\,, 
	  \end{cases}
	  \label{hessian}
\end{align}
\endgroup
where $1 \le j, j' \le N$.

In order to study the optimization problem \eqref{opt-problem-y-t}--\eqref{opt-problem-y-t-f}, 
 let us introduce the matrix
 \begingroup
 \small
\begin{align}
      \Phi_i  = 
      \begin{bmatrix}
	\varphi_{j_1}(y_{l^{(i)}_1}) & \varphi_{j_2}(y_{l^{(i)}_1}) & \cdots &
	\varphi_{j_{N_i}}(y_{l^{(i)}_1}) \\
	\varphi_{j_1}(y_{l^{(i)}_2}) & \varphi_{j_2}(y_{l^{(i)}_2}) & \cdots &
	\varphi_{j_{N_i}}(y_{l^{(i)}_2}) \\
	\varphi_{j_1}(y_{l^{(i)}_3}) & \varphi_{j_2}(y_{l^{(i)}_3}) & \cdots &
	\varphi_{j_{N_i}}(y_{l^{(i)}_3}) \\
	\vdots & \vdots & \ddots & \vdots\\
	\varphi_{j_1}(y_{l^{(i)}_{M_i}}) &
	\varphi_{j_2}(y_{l^{(i)}_{M_i}}) & \cdots &
	\varphi_{j_{N_i}}(y_{l^{(i)}_{M_i}}) \\
      \end{bmatrix}
      \in \mathbb{R}^{M_i \times N_i}\,,
      \label{mat-phi-i}
    \end{align}
   \endgroup 
for each $1 \le i \le K$, where we have assumed that the index set $\mathcal{I}_i = \big\{j_1, j_2, \dots, j_{N_i} \big\}$.
We define $\Phi_{i,k} \in \mathbb{R}^{M_i}$ to be the $k$th column vector of $\Phi_i$ for $1 \le k \le N_i$ and thus obtain the following result concerning the solution of the optimization problem \eqref{opt-problem-y-t}--\eqref{opt-problem-y-t-f}.
    \begin{proposition}
      The following three conditions are equivalent.
      \begin{enumerate}[wide]
	\item
For each $1 \le i \le K$, the vectors $\Phi_{i,1}, \Phi_{i, 2}, \dots, \Phi_{i, N_i}$ are linearly independent.
	 \item
	   The function $-\ln\mathcal{L}^{(T)}(\bm{\omega})$ in \eqref{opt-problem-y-t-f} is strictly convex.
	\item
	  The optimization problem
	  \eqref{opt-problem-y-t}--\eqref{opt-problem-y-t-f} has a unique solution.
      \end{enumerate}
      \label{prop-uniqueness-conditions}
    \end{proposition}
    \begin{proof}
      (2) $\Rightarrow$ (3) is obvious.
      To show that (1) implies (2), it is sufficient to verify that the
      Hessian matrix of $-\ln\mathcal{L}^{(T)}$ is
      positive definite. Using \eqref{hessian}, for any vector $\bm{\eta} =
      (\eta_1, \eta_2, \dots, \eta_N)^\top \in \mathbb{R}^N$, we have
      \begingroup
      \small
\setlength\abovedisplayskip{6pt}
\setlength\belowdisplayskip{6pt}
      \begin{align}
	\sum_{j=1}^{N}\sum_{j'=1}^{N}
	\frac{\partial^2 \big(-\ln\mathcal{L}^{(T)}\big)}{\partial \omega_{j}\partial \omega_{j'}}
	\eta_{j}\eta_{j'} = \sum_{i=1}^K \sum\limits_{k=1}^{M_i} \frac{
	\Big(\sum\limits_{j\in \mathcal{I}_i}
      \eta_{j}\,\varphi_{j}(y_{l_k^{(i)}})\Big)^2}{\Big(\sum\limits_{j\in\mathcal{I}_i}
	\omega_{j}\,\varphi_{j}(y_{l^{(i)}_k})\Big)^2} \ge 0\,.
	  \label{hessian-f}
      \end{align}
      \endgroup
      Since the columns of $\Phi_i$ are linearly independent for
      each $i$, we conclude that \eqref{hessian-f} is zero if and only if $\bm{\eta}$ is a zero vector. 
      This implies that $-\ln\mathcal{L}^{(T)}$ is strictly convex.

      Finally, let us prove that (3) implies (1) by contradiction. Define       $\bm{\omega}$ to be the unique solution of the optimization problem \eqref{opt-problem-y-t}. 
      Assume that there is $i$, $1 \le i \le K$, such that the vectors $\Phi_{i,1},
      \Phi_{i,2}, \dots, \Phi_{i, N_i}$ are
      linearly dependent. As a result, we can find a vector
      $\widetilde{\bm{\omega}} = (\widetilde{\omega}_{1},
      \widetilde{\omega}_{2}, \dots, \widetilde{\omega}_{N})^\top \neq
      \bm{\omega}$, such that 
      \begingroup
\setlength\abovedisplayskip{6pt}
\setlength\belowdisplayskip{6pt}
      \begin{gather}
	\begin{aligned}
	   \widetilde{\omega}_{j} &=
	    \omega_{j},&& ~\forall~j \in \mathcal{I}_{i'},~ i'\neq i\,,\\
	    \mbox{and}\quad {\sum\limits_{j\in \mathcal{I}_i}
	    \widetilde{\omega}_{j}\,\varphi_{j}(y_{l^{(i)}_k})}
	    &= {\sum\limits_{j\in \mathcal{I}_i} \omega_{j}\, \varphi_{j}(y_{l^{(i)}_k})}\,,&& ~\forall~1 \le k \le M_i~.
	\end{aligned}
	\label{omega-and-wide-omega}
      \end{gather}
      \endgroup
      Since $\bm{\omega}$ satisfies \eqref{opt-problem-euler-lagrange}, the property \eqref{omega-and-wide-omega} implies
      that $\widetilde{\bm{\omega}}$ satisfies \eqref{opt-problem-euler-lagrange} as well. 
      Multiplying by $\omega_{j}$ (or $\widetilde{\omega}_{j}$) on both sides of
      \eqref{opt-problem-euler-lagrange} and summing up the indices, we get
      \begingroup\small
\setlength\abovedisplayskip{6pt}
\setlength\belowdisplayskip{6pt}
      \begin{align}
	\sum_{l=0}^{M} t_l\bigg[ \sum_{j=1}^{N}
	\widetilde{\omega}_{j}\,\varphi_{j}(y_l)\bigg] = 
	\sum_{l=0}^{M} t_l\bigg[ \sum_{j=1}^N
	\omega_{j}\,\varphi_{j}(y_l)\bigg] = M\,.
      \label{omega-linear-summation}
      \end{align}
     \endgroup 
      Combining \eqref{omega-and-wide-omega}, \eqref{omega-linear-summation},
      as well as the expressions in \eqref{opt-problem-y-t-f}, we obtain
      $-\ln\mathcal{L}^{(T)}(\bm{\omega}) = -\ln\mathcal{L}^{(T)}(\widetilde{\bm{\omega}})$, which contradicts the uniqueness of 
      $\bm{\omega}$. 
    \end{proof}

To distinguish the parameters obtained from solving the optimization problem
\eqref{opt-problem-y-t} and the true parameters of the system, we will 
define $\bm{\omega}^{(T)}$ to be the maximizer of $-\ln\mathcal{L}^{(T)}$ for
fixed time $T>0$ in what follows, and $\bm{\omega}^{*}$ to be the vector consisting of the true
parameters such that \eqref{ar-expression} holds.
In particular, when $N_i = 1$ and $\mathcal{I}_i = \{j\}$ , i.e., the channel $\mathcal{C}_i$ only contains one
reaction $\mathcal{R}_j$, the Euler--Lagrange equation
\eqref{opt-problem-euler-lagrange} can be solved analytically and we have 
\begingroup
\small
\begin{align}
  \omega_{j}^{(T)} = \frac{M_i}{\sum\limits_{l=0}^{M} t_l \,
  \varphi_{j}(y_l)}\,. 
  \label{opt-omega-ni1}
\end{align}
\endgroup
\subsection{Learning task $\textbf{2}$: determine the rate constants and the structure of chemical reactions using sparsity}
\label{subsec-unkown-structure}
In this subsection, we study the problem of learning the propensity functions of the chemical reaction networks from 
trajectory data when neither the structure of the chemical reactions nor their rate constants is known. 

First of all, we can figure out the total number $K$ of the reaction channels from
the trajectory data, as discussed in Subsection~\ref{subsec-space-trajectory}.
Now suppose that we are given $N$ candidate basis functions
\begingroup
\setlength\abovedisplayskip{6pt}
\setlength\belowdisplayskip{6pt}
\begin{align}
\varphi_j : \mathbb{N}^n \rightarrow \mathbb{R}\,, \qquad 1 \le j \le N\,,
  \label{library-functions}
\end{align}
\endgroup
together with $K$ index sets $\mathcal{I}_i = \{j_1, j_2, \dots, j_{N_i}\}$, $1 \le i \le K$, such that $N_i = |\mathcal{I}_i|$, 
\begingroup
\small
\setlength\abovedisplayskip{6pt}
\setlength\belowdisplayskip{6pt}
\begin{align}
  \bigcup_{i=1}^K \mathcal{I}_i = \Big\{1, 2, \dots, N\Big\}\,,  \qquad
  \mbox{and}\quad \mathcal{I}_i \bigcap \mathcal{I}_{i'} = \emptyset\,,~\mbox{if}~ i\neq i'\,.
  \label{set-i-repeat}
\end{align}
\endgroup
Accordingly, we introduce the vectors
   \begin{align}
       \bm{\omega} = (\omega_{1}, \omega_{2}\,,\dots\,, \omega_{N})^\top
     \in \mathbb{R}^N\,, \quad 
     \mbox{and}\quad \bm{\omega}^{(i)} = (\omega_{j_1}, \omega_{j_2}\,,\dots\,, \omega_{j_N})^\top
     \in \mathbb{R}^{N_i}\,. 
  \label{omega-vector-lambda}
\end{align}
For each channel $\mathcal{C}_i$, the propensity function $a_i^*$ in \eqref{a-sum-ai-true} will be approximated
using the basis functions $\varphi_j$, $j \in \mathcal{I}_i$, and the coefficients
in $\bm{\omega}^{(i)}$. More precisely, we define 
\begingroup
\setlength\abovedisplayskip{6pt}
\setlength\belowdisplayskip{6pt}
  \begin{align}
    a_i^{(\epsilon)}\big(x\,;\,\bm{\omega}\big) =
    a_i^{(\epsilon)}\big(x\,;\,\bm{\omega}^{(i)}\big) = G_\epsilon\Big(\sum_{j\in \mathcal{I}_i}
    \omega_{j} \varphi_{j}(x)\Big)\,,
  \label{ai-omega-lambda}
\end{align}
\endgroup
where $\epsilon>0$, and the function 
\begin{align}
  G_\epsilon(x) = \epsilon\ln(1 + e^{x/\epsilon})\,, \quad \epsilon>0\,,
  \label{g-eps}
\end{align}
is introduced (see Figure~\ref{fig-g-eps}), in order to guarantee the
non-negativity of $a^{(\epsilon)}_i$ for all vectors $\bm{\omega} \in \mathbb{R}^N$.
Corresponding to \eqref{ai-omega-lambda}, the total propensity function is given by
\begingroup
      \small
\setlength\abovedisplayskip{6pt}
\setlength\belowdisplayskip{6pt}
\begin{align}
  a^{(\epsilon)}\big(x\,;\,\bm{\omega}\big) = \sum_{i=1}^K G_\epsilon\Big(\sum_{j\in \mathcal{I}_i} \omega_{j} \varphi_{j}(x)\Big)\,.
  \label{prop-a-lambda}
\end{align}
\endgroup
\begin{figure}[htpb]
  \includegraphics[width=0.42\textwidth]{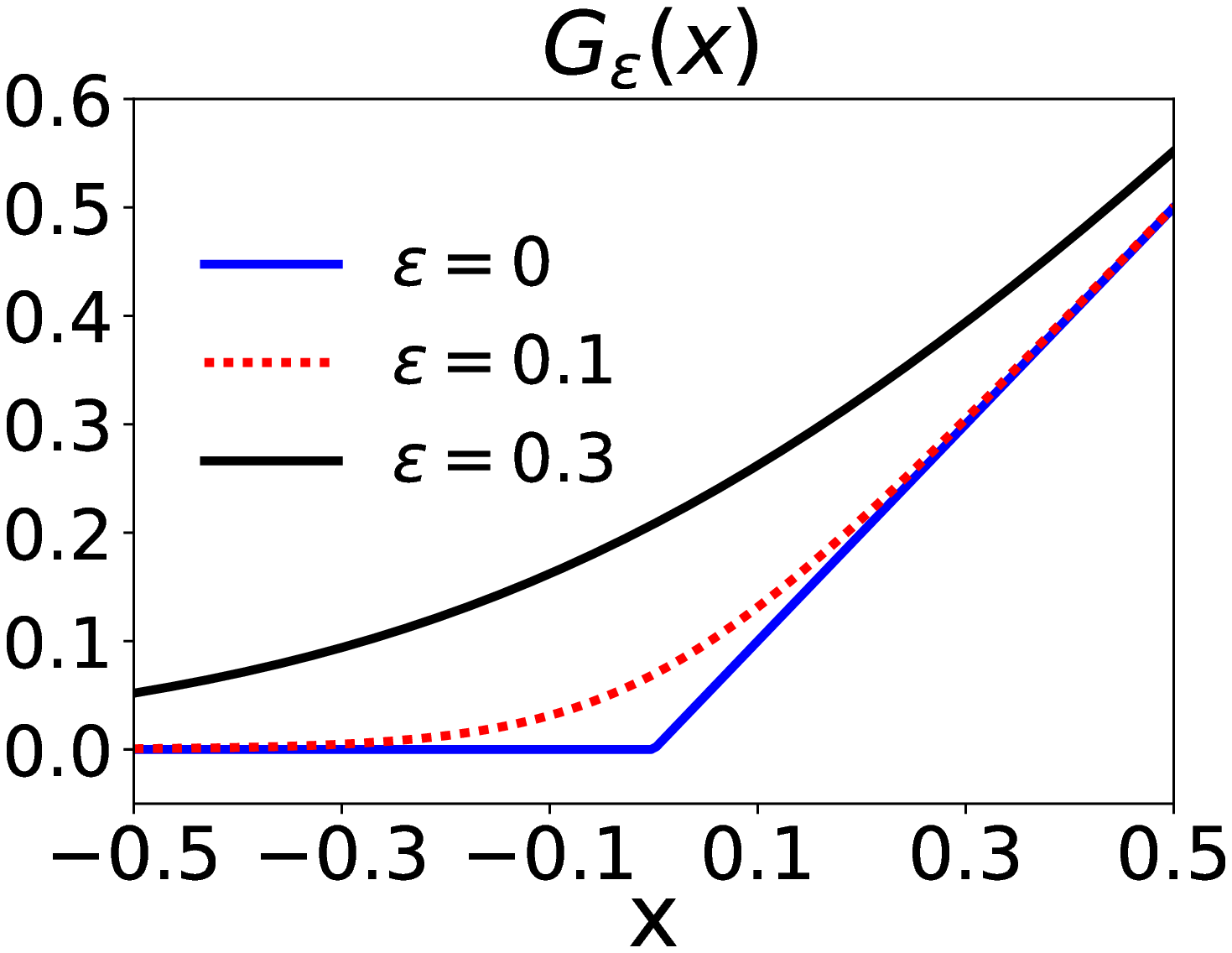}
  \includegraphics[width=0.42\textwidth]{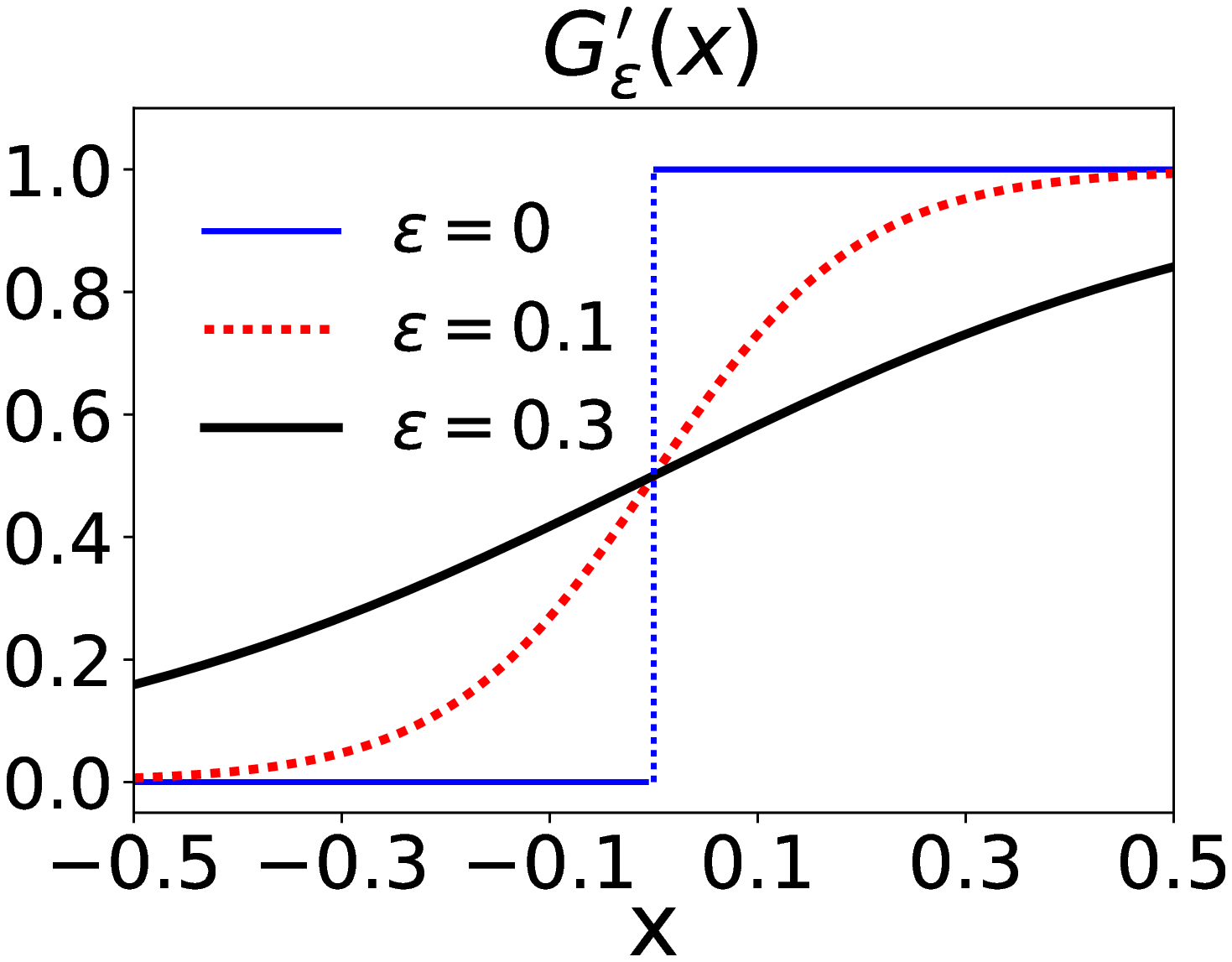}
  \centering
  \caption{Profiles of $G_\epsilon$ in \eqref{g-eps} and its derivative $G'_\epsilon$. 
  For $\epsilon = 0$, we define $G_0(x) = \lim\limits_{\epsilon \rightarrow 0+}
  G_\epsilon(x) = \max(x,0)$. See Remark~\ref{rmk-g-eps} and Appendix~\ref{app-1}
  for the properties of $G_\epsilon$. \label{fig-g-eps}}
\end{figure}

Since the propensity functions of reactions in many applications typically
have a simple form (Table~\ref{propensity-table}), there is likely redundancy
in the basis functions and therefore we can assume that
the unknown vector $\bm{\omega}$ only has a few nonzero entries (and is thus sparse). 
With this observation in mind, we propose to determine $\bm{\omega}$ by maximizing the (logarithmic) likelihood function under the
sparsity assumption, or, equivalently, by solving the nonlinear sparse minimization problem 
\begingroup
\setlength\abovedisplayskip{6pt}
\setlength\belowdisplayskip{6pt}
\begin{align}
  \min_{\bm{\omega}} \Big[-\ln\mathcal{L}^{(T,\epsilon)}(\bm{\omega})\Big]\,,\quad \bm{\omega}~\mbox{is sparse}\,, 
  \label{opt-problem-l0}
\end{align}
\endgroup
where $\mathcal{L}^{(T, \epsilon)}(\bm{\omega})$ is the likelihood function
\eqref{likelihood-l-y-t} with the propensity functions $a_i=a_i^{(\epsilon)},
a=a^{(\epsilon)}$ in \eqref{ai-omega-lambda} and \eqref{prop-a-lambda}. Explicitly, we have 
\begingroup
\small
\setlength\abovedisplayskip{6pt}
\setlength\belowdisplayskip{6pt}
\begin{align}
    -\ln\mathcal{L}^{(T, \epsilon)}(\bm{\omega})
    =  -\sum_{l=0}^{M-1} \ln G_\epsilon\bigg(\sum_{j\in \mathcal{I}_{i_l}} \omega_{j}\,\varphi_{j}(y_l)\bigg)
    + \sum_{l=0}^{M} t_l 
    \bigg[\sum_{i=1}^K G_\epsilon\Big(\sum_{j\in \mathcal{I}_i} \omega_{j} \varphi_{j}(y_l)\Big)\bigg]\,.
    \label{likelihood-t-eps}
\end{align}
\endgroup
If we quantify the sparsity of $\bm{\omega}$ using the $l^1$ norm (denoted by $\|\cdot\|_1$), then \eqref{opt-problem-l0} results in
\begingroup
\setlength\abovedisplayskip{6pt}
\setlength\belowdisplayskip{6pt}
\begin{align}
  \min_{\bm{\omega}} \Big(-\frac{1}{T}\ln\mathcal{L}^{(T,\epsilon)}(\bm{\omega}) +
  \lambda \|\bm{\omega}\|_1\Big)\,.
  \label{opt-problem-1}
\end{align}
\endgroup
In \eqref{opt-problem-1}, the log-likelihood function is rescaled by $1/T$
(this scaling is suggested by the analysis in Section~\ref{sec-asymptotics}), and the constant
$\lambda = \lambda(T) > 0$, which measures the strength of the sparsity
regularization, can be chosen depending on $T$. 

Similar to the problem \eqref{opt-problem-y-t} in the previous subsection, the minimizer of \eqref{opt-problem-1} 
can be computed by solving $K$ sparse minimization problems
\begin{equation}
  \begin{split}
    &\min_{\bm{\omega}^{(i)}} \Big(-\frac{1}{T}\ln\mathcal{L}_i^{(T,
  \epsilon)}(\bm{\omega}^{(i)}) + \lambda
  \|\bm{\omega}^{(i)}\|_1\Big)\,,\qquad 1 \le i \le K\,,
  \end{split}
    \label{opt-problem-1-sub}
\end{equation}
separately, where 
\begingroup
\small
\begin{equation}
\ln\mathcal{L}_i^{(T, \epsilon)}(\bm{\omega}^{(i)})
    =  \sum_{k=1}^{M_i} \ln G_\epsilon\bigg(\sum_{j\in \mathcal{I}_{i}}
\omega_{j}\,\varphi_{j}(y_{l_k^{(i)}})\bigg)
    - \sum_{l=0}^{M} t_l G_\epsilon\Big(\sum_{j\in \mathcal{I}_i} \omega_{j} \varphi_{j}(y_l)\Big)\,.
    \label{likelihood-t-eps-sub}
  \end{equation}
  \endgroup
In practice, we find that \eqref{opt-problem-1}, or equivalently \eqref{opt-problem-1-sub}, 
can be efficiently solved by FISTA proposed in~\cite{fista2009}, especially
when preconditioning is applied (see Remark~\ref{rmk-preconditioning} below
and examples in Section~\ref{sec-example}). 
The main algorithmic steps of FISTA are provided in Algorithm~\ref{fista-algo} in Appendix~\ref{app-0}.

We obtain the following result concerning the minimization problems \eqref{opt-problem-1} and \eqref{opt-problem-1-sub}.
\begin{proposition}
  Suppose $\epsilon, \lambda > 0$. The objective functions of the optimization problems
  \eqref{opt-problem-1} and \eqref{opt-problem-1-sub} are strictly convex. 
\end{proposition}
\begin{proof}
  It is sufficient to consider the objective function in \eqref{opt-problem-1-sub}.
  By straightforward calculations (for instance, see \eqref{g-1st-2rd} and
  \eqref{log-g-eps-derivatives} in Appendix~\ref{app-1}), we can verify that both 
  $-\ln G_\epsilon$ and $G_\epsilon$ are strictly convex functions. 
  Therefore, the function $-\ln\mathcal{L}^{(T,\epsilon)}_i$ in
  \eqref{likelihood-t-eps-sub} is strictly convex.
 Since the norm $\|\cdot\|_1$ is convex as well, we conclude that the
 objective function in \eqref{opt-problem-1-sub} is strictly convex. 
\end{proof}

Let $\bm{\omega}^{(T, \epsilon, \lambda)}$ denote the unique minimizer of the problem \eqref{opt-problem-1}.  Similar to the Euler--Lagrange equation \eqref{opt-problem-euler-lagrange},
 in the current case $\bm{\omega}^{(T, \epsilon, \lambda)}$ satisfies the
 inclusion relation~\cite{Bagirov-nonsmooth,clarke1990optimization}
 \begingroup
\small
\setlength\abovedisplayskip{6pt}
\setlength\belowdisplayskip{6pt}
\begin{align}
  \frac{1}{T}\mathcal{M}_j^{T,\epsilon}(\bm{\omega}) \in -\lambda\,\partial |\omega_j|
  \,,\quad  \forall\, 1 \le j \le N\,,
  \label{euler-lagrange-lambda}
\end{align}
\endgroup
where  
\begingroup
\small
\begin{align}
  \mathcal{M}^{(T,\epsilon)}_{j}(\bm{\omega}) 
  &= \frac{\partial \big(-\ln\mathcal{L}^{(T,\epsilon)}\big)}{\partial
  \omega_{j}}(\bm{\omega})\nonumber \\
  &=  -\sum_{k=1}^{M_i}
  \frac{
    G_\epsilon'\Big(\sum\limits_{j'\in \mathcal{I}_i} \omega_{j'}\,\varphi_{j'}(y_{l^{(i)}_k})\Big)
    \varphi_{j}(y_{l^{(i)}_k})}{G_\epsilon\Big(\sum\limits_{j'\in \mathcal{I}_i}
  \omega_{j'}\,\varphi_{j'}(y_{l^{(i)}_k})\Big)} + \sum_{l=0}^{M} t_l \,
  \bigg[G_\epsilon'\Big(\sum\limits_{j'\in \mathcal{I}_i}
  \omega_{j'}\,\varphi_{j'}(y_{l})\Big) \varphi_{j}(y_l)\bigg]\,, \label{log-likelihood-derivative-lambda}
\end{align}
\endgroup
for $j \in \mathcal{I}_i$, and $\partial |\omega_j|$ is the subdifferential of the absolute value function
$|\omega_j|$, defined by  
\begingroup
\small
\setlength\abovedisplayskip{6pt}
\setlength\belowdisplayskip{6pt}
\begin{align*}
\partial |\omega_j| = 
	  \begin{cases}
	    ~\{1\}\,, & \omega_j > 0 \,, \\
	    ~[-1, 1]\,, & \omega_j = 0\,,\\
	    ~\{-1\}\,, & \omega_j < 0 \,. 
	  \end{cases}
\end{align*}
\endgroup
Finally, let $\mathcal{M}^{(T,\epsilon)}$ be the vector in $ \mathbb{R}^N$
whose components are defined in \eqref{log-likelihood-derivative-lambda} and define the set
  $\partial |\bm{\omega}| = \big\{\bm{v} \in \mathbb{R}^N~\big|~ \bm{v} = (v_1, v_2,
  \dots, v_N)^\top\,,~v_j \in \partial |\omega_j|\,,~1 \le j \le N\big\}$\,.
We can express the condition \eqref{euler-lagrange-lambda} in vector form as 
\begin{align}
  \frac{1}{T}\mathcal{M}^{T,\epsilon}(\bm{\omega}) \in -\lambda\,\partial |\bm{\omega}| \,.
  \label{euler-lagrange-lambda-vector}
\end{align}
The characterization above of the minimizers will be used in the analysis in Section~\ref{sec-asymptotics}. 

We conclude this section with the following remarks. 
\begin{remark}[Role of the function $G_\epsilon$]
  In principle, we would like to allow both the basis functions $\varphi_j$
  and the unknown coefficients $\omega_j$ to be either positive or negative. 
  By introducing the function $G_\epsilon$ in \eqref{likelihood-t-eps},
  we avoid imposing many inequality constraints which would be otherwise
  needed in order to guarantee that the log-likelihood function in
  \eqref{likelihood-t-eps} is well-defined. 
    The properties of $G_\epsilon$ in \eqref{g-eps} are discussed in
      Appendix~\ref{app-1}. In particular, we have $\lim\limits_{\epsilon
      \rightarrow 0+} G_\epsilon(x) = \max(x,0)$, uniformly $\forall~x \in
      \mathbb{R}$. For this reason, we define $G_0(x) =
      \max(x,0)$.
      \label{rmk-g-eps}
\end{remark}
\begin{remark}[Choice of basis functions]
  \begin{enumerate}[wide]
    \item
  In the sparse minimization problem \eqref{opt-problem-1}, 
  the vector $\bm{\omega}$ contains all the $N$ coefficients $\omega_j$, and 
  the corresponding $N$ basis functions $\varphi_j$ in \eqref{library-functions} are involved. 
This formulation makes the notations simpler and is also convenient for analysis, particularly in Section~\ref{sec-asymptotics}.
Numerically, on the other hand, the coefficient vectors $\bm{\omega}^{(i)}$ in
      \eqref{omega-vector-lambda} can be computed separately by solving the
      minimization problems \eqref{opt-problem-1-sub}, $1 \le i \le K$, with the same set of basis functions
  $\phi_1,\,\phi_2, \dots, \phi_L$, $L>0$, for all the $K$ channels.
      In this case, corresponding to the formulation adopted at the beginning
      of this subsection where all $N$ coefficients are put together, we define the index sets 
   $\mathcal{I}_i = \big\{(i-1)L+1\,,~(i-1)L+2\,,\, \dots\,,~iL\big\}$,\, $1
      \le i \le K$, and for each $j \in \mathcal{I}_i$, we define the function 
      \begingroup
\setlength\abovedisplayskip{6pt}
\setlength\belowdisplayskip{6pt}
  \begin{align}
    \varphi_j = \phi_k\,, \quad \mbox{when}~~j = (i-1)L + k\,, \quad 1 \le k \le L\,.
    \label{varphi-j-phi-lambda}
  \end{align}
  \endgroup
      Accordingly, we have 
    $\bm{\omega}^{(i)} = \big(\omega_{(i-1)L + 1},\, \omega_{(i-1)L + 2},\, \dots, \omega_{(i-1)L + L}\big)^\top$, 
  and the propensity function in \eqref{ai-omega-lambda} can be written more transparently as
      \begingroup
\setlength\abovedisplayskip{6pt}
\setlength\belowdisplayskip{6pt}
  \begin{align*}
    a_i^{(\epsilon)}\big(x\,;\,\bm{\omega}\big) =
    a_i^{(\epsilon)}\big(x\,;\,\bm{\omega}^{(i)}\big) =
    G_\epsilon\Big(\sum_{k=1}^{L} \omega_{(i-1)L + k}\,\phi_{k}(x)\Big)\,.
\end{align*} 
\endgroup
\item
  While we are mainly interested in chemical reaction systems, the same
      learning approach can be applied to other types of continuous-time
      Markov chains whose jump distributions are state-dependent. In
      particular, for chemical reaction systems that obey law of mass-action, 
      according to Table~\ref{propensity-table} we may choose $\varphi_j$ from
      the polynomials 
\begin{align}
  \begin{split}
  & 1\,, \quad x^{(1)}\,,\quad x^{(2)}\,,\quad \dots\,, \quad x^{(n)}\,,\quad
   x^{(1)}x^{(2)}\,,\quad x^{(1)}x^{(3)}\,,\quad\dots\,, \quad
  x^{(1)}x^{(n)}\,,\quad \\
    & x^{(2)}x^{(3)}\,,\quad \dots\,, \quad x^{(n-1)}x^{(n)}\,,   \dots\,,
  \end{split}
  \label{poly-basis}
\end{align}
where $x^{(k)}$ denotes the $k$th component of the state $x = (x^{(1)}, x^{(2)}, \dots, x^{(n)})^\top$,
based on the knowledge about the potential chemical reactions that are possibly involved in the system. 
  \end{enumerate}
  \label{rmk-choice-of-basis}
\end{remark}
  \begin{remark}[Preconditioning]
In concrete applications, due to the complexity of the trajectory data, 
      different basis functions may take values that are of different orders
      of magnitude. As a result, the objective functions in \eqref{opt-problem-1-sub}, or equivalently in \eqref{opt-problem-1}, 
      may become inhomogeneous along different components $\omega_j$. This leads to numerical difficulties in solving \eqref{opt-problem-1-sub} 
      since a small step-size has to be used as a result of the strong dependence of the objective function on the change of $\bm{\omega}$
      along certain directions (i.e., large gradient, ill-conditioned).
A simple way to alleviate this numerical issue is to precondition the problems
\eqref{opt-problem-1-sub} by rescaling the basis functions. Equivalently, let $c_j$ denote
      the rescaling constants, where $c_j>0$, $1 \le j \le N$.
      Instead of \eqref{opt-problem-1-sub}, we can compute the minimizer
      $\overline{\bm{\omega}}^{(i)}$ of the rescaled sparse minimization problem 
      \begingroup
      \small
\begin{align}
  \min_{\overline{\bm{\omega}}^{(i)}} \bigg\{\!-\frac{1}{T}\sum_{k=1}^{M_i} \ln G_\epsilon\bigg(\sum_{j\in
    \mathcal{I}_{i}}
    \frac{\overline{\omega}_{j}}{c_j}\,\varphi_{j}\big(y_{l_k^{(i)}}\big)\bigg)
    + \frac{1}{T}\sum_{l=0}^{M} t_l 
    G_\epsilon\Big(\sum_{j\in \mathcal{I}_i}
    \frac{\overline{\omega}_{j}}{c_j} \varphi_{j}(y_l)\Big)
  + \lambda \sum_{j\in \mathcal{I}_i} \frac{|\overline{\omega}_j|}{c_j}\bigg\}\,,
  \label{opt-problem-1-rescaled}
    \end{align}
    \endgroup
where the vector $\overline{\bm{\omega}}^{(i)}$ consists of $\overline{\omega}_j$, $j \in \mathcal{I}_i$. 
    Then it is easy to verify that the minimizer $\bm{\omega}^{(i)}$ of \eqref{opt-problem-1-sub}
    can be recovered from $\omega_j = \frac{\overline{\omega}_j}{c_j}$, for $j \in \mathcal{I}_i$.
    By properly choosing the constants $c_j$ based on analyzing the trajectory data, 
    we can expect that minimizing \eqref{opt-problem-1-rescaled} will be
    easier compared to \eqref{opt-problem-1-sub}. Readers are referred to
    Section~\ref{sec-example} for further discussions on this issue and concrete examples.
  \label{rmk-preconditioning}
\end{remark}
\begin{remark}[Possible extensions]
  Below we discuss several possible generalizations.
  \begin{enumerate}[wide]
    \item
      So far, we have assumed that the evolution of the system is fully observed. In concrete applications, sometimes a small subset of species in the
      system is supposed to be able to describe the system's
      dynamics~\cite{nssa-weinan-jcp2005}. Correspondingly, it may happen that the
      trajectory data is only partially observed for these ``important'' species. 
      In this case, one can still apply the second
      learning approach in this subsection to learn the system and the outcome of the optimization 
      problem \eqref{opt-problem-1} will be an effective dynamics for these selected ``important'' species.
      However, we point out that, since the effective reactions among these ``important'' species do not necessarily obey the law of mass-action any more, 
      it may be important to include other types of basis functions (e.g.,
      rational functions for Michaelis--Menten type kinetics~\cite{murray2007mathematical}) together with the polynomial basis in \eqref{poly-basis} in order to obtain a good
      approximation of the effective dynamics. 
    \item
      It is straightforward to generalize the analysis to the case where
      multiple trajectories of the system are available. We refer the readers to the
      numerical examples in Section~\ref{sec-example} for details.
    \item
      In this work, in particular in Section~\ref{sec-asymptotics}, we are mainly interested
      in the theoretical justification of the two learning approaches in the infinite-data limit, i.e., $T\rightarrow +\infty$. 
      The numerical examples in Section~\ref{sec-example} also mainly serve this purpose.
	Regarding the choice of the sparsity parameter $\lambda$, one can expect that a large $\lambda$ will increase the sparsity of the solution, but at the same time will also introduce bias in the prediction. Therefore, in the numerical experiments in Section~\ref{sec-example}, we empirically choose $\lambda$ in such a way that the sparsity and accuracy of the solution are balanced.
      In practice, instead of choosing a fixed $\lambda>0$ in
      \eqref{opt-problem-1} for coefficients in front of all basis functions, it is helpful to
      consider different values of $\lambda$ for different coefficients and
      to tune the parameter(s) $\lambda$ carefully using the cross-validation technique~\cite{cross-validation-geisser,hastie2009elements}.
      See~\cite{sindy-sde} for more details.
  \end{enumerate}
  \label{rmk-possible-extension}
\end{remark}

\section{Examples}
\label{sec-example}
In this section, we study the learning tasks discussed in Section~\ref{sec-inverse} with three concrete numerical examples.
\subsection{Example 1}
\label{subsec-ex1}
In the first example, we study the chemical reaction system given by
Table~\ref{ex1-cr-table}, where two different species $A, B$ are involved in
$4$ chemical reactions. The propensity functions of these $4$ reactions depend on both the state $x=(x^{(1)}, x^{(2)})^\top$ of the system, i.e., the copy-numbers of the species $A$ and $B$, and the rate constants $\kappa_i$, $i=1,2,3,4$.

To study the two learning tasks discussed in Section~\ref{sec-inverse}, we fix the
parameters 
      \begingroup
\setlength\abovedisplayskip{6pt}
\setlength\belowdisplayskip{6pt}
\begin{align}
(\kappa_1, \kappa_2, \kappa_3, \kappa_4) = (1.0,~0.1,~1.0,~0.9)\,,
\label{ex1-true-kappa}
\end{align}
\endgroup
and $Q=100$ trajectories of the system are generated using the stochastic simulation
algorithm (SSA)~\cite{Gillespie1976_ssa, Gillespie1977_ssa, ssck_gillespie}. Each trajectory starts from the same initial state
$x=(20, 10)^\top$ at time $t=0$ and is simulated until time $T=10$ ($5$ of the $100$ trajectories are shown in Figure~\ref{fig-traj-ex1} for illustration). 
From Table~\ref{ex1-cr-table}, it is clear that different reactions belong to
different reaction channels and therefore there are in total $4$ reaction channels in the reaction
network. For the quantities introduced in Section~\ref{sec-forward}, we obtain $N_i=1$ and $K=N=4$.
After processing the trajectory data, we find that the activation numbers of
the $4$ reaction channels
within these $100$ trajectories are $2296$, $1778$, $2777$, and $2135$,
respectively, as shown in Table~\ref{ex1-info-data-table}.

\begin{table}
    \caption{\label{ex1-cr-table} Example 1. Chemical reaction system
    consists of two species $A$ and $B$ and $4$ chemical reactions. The
    copy-numbers of these two species are denoted by $x=(x^{(1)},
    x^{(2)})^\top$. Here, $\kappa_i$, $v$, and $a^*_{\mathcal{R}}(x)$ are the rate constant, the state change vector, and the propensity function of the reactions, respectively.} 
    \centering
    \scalebox{0.9}{
    \begin{tabular}{clccl}
      No. & Reaction & $v^\top$ & Channel & $a_{\mathcal{R}}^*(x)$ \\
      \hline
      $1$ & \ce{$A$ ->[\kappa_1] $\emptyset$}  & $(-1, 0)$ & $1$  & $\kappa_1 x^{(1)}$ \\
      $2$ & \ce{$A$ + $B$ ->[\kappa_2] $2B$}  & $(-1,1)$ & $2$ & $\kappa_2 x^{(1)}x^{(2)}$ \\
      $3$ & \ce{$B$ ->[\kappa_3] $\emptyset$}  & $(0,-1)$ & $3$ & $\kappa_3 x^{(2)}$ \\
      $4$ & \ce{$A$ ->[\kappa_4] 2$A$}  & $(1,0)$ & $4$& $\kappa_4 x^{(1)}$ \\
      \hline
    \end{tabular}}
\end{table}

\begin{figure}[htpb]
  \includegraphics[width=0.42\textwidth]{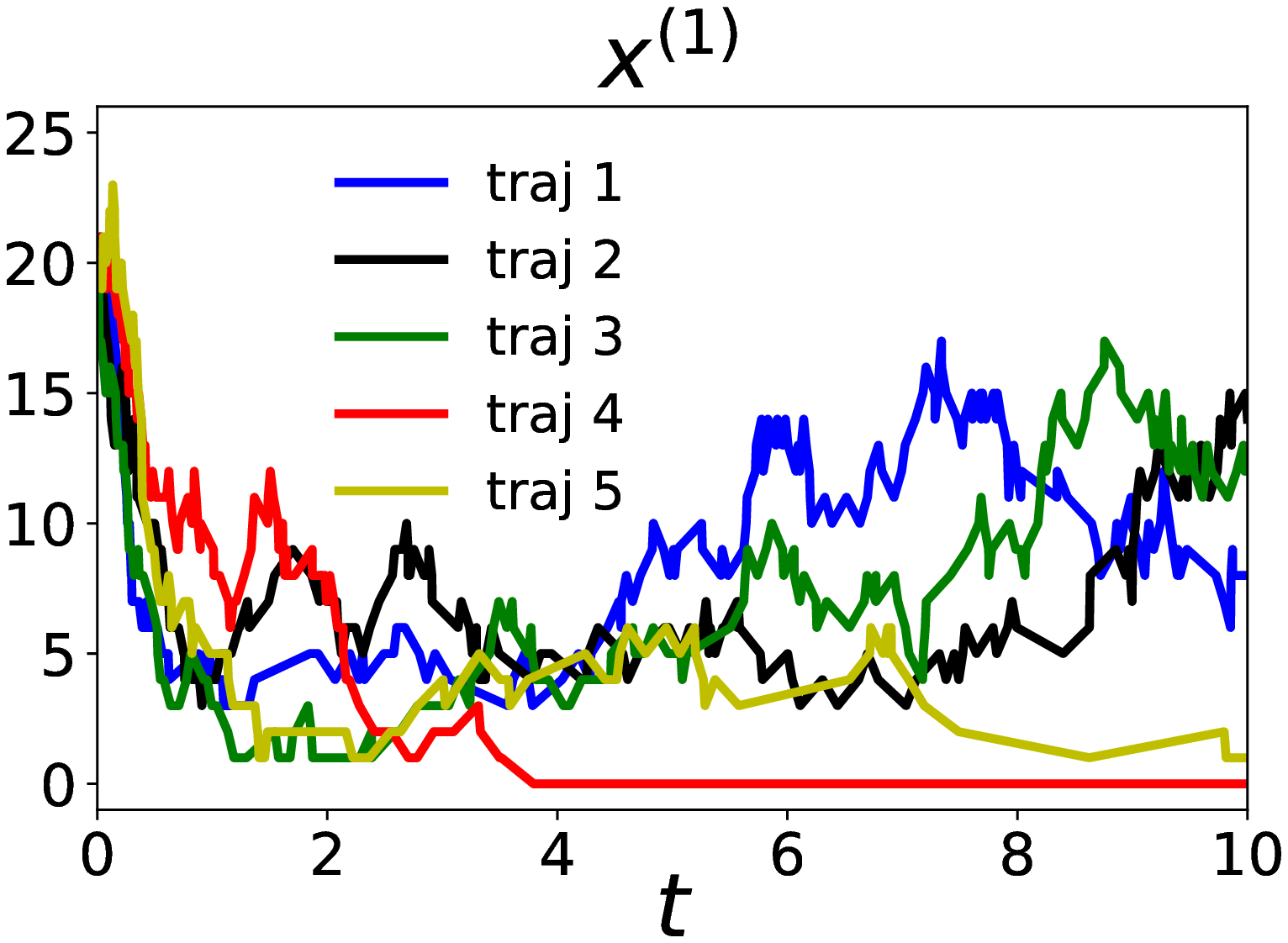}
  \includegraphics[width=0.42\textwidth]{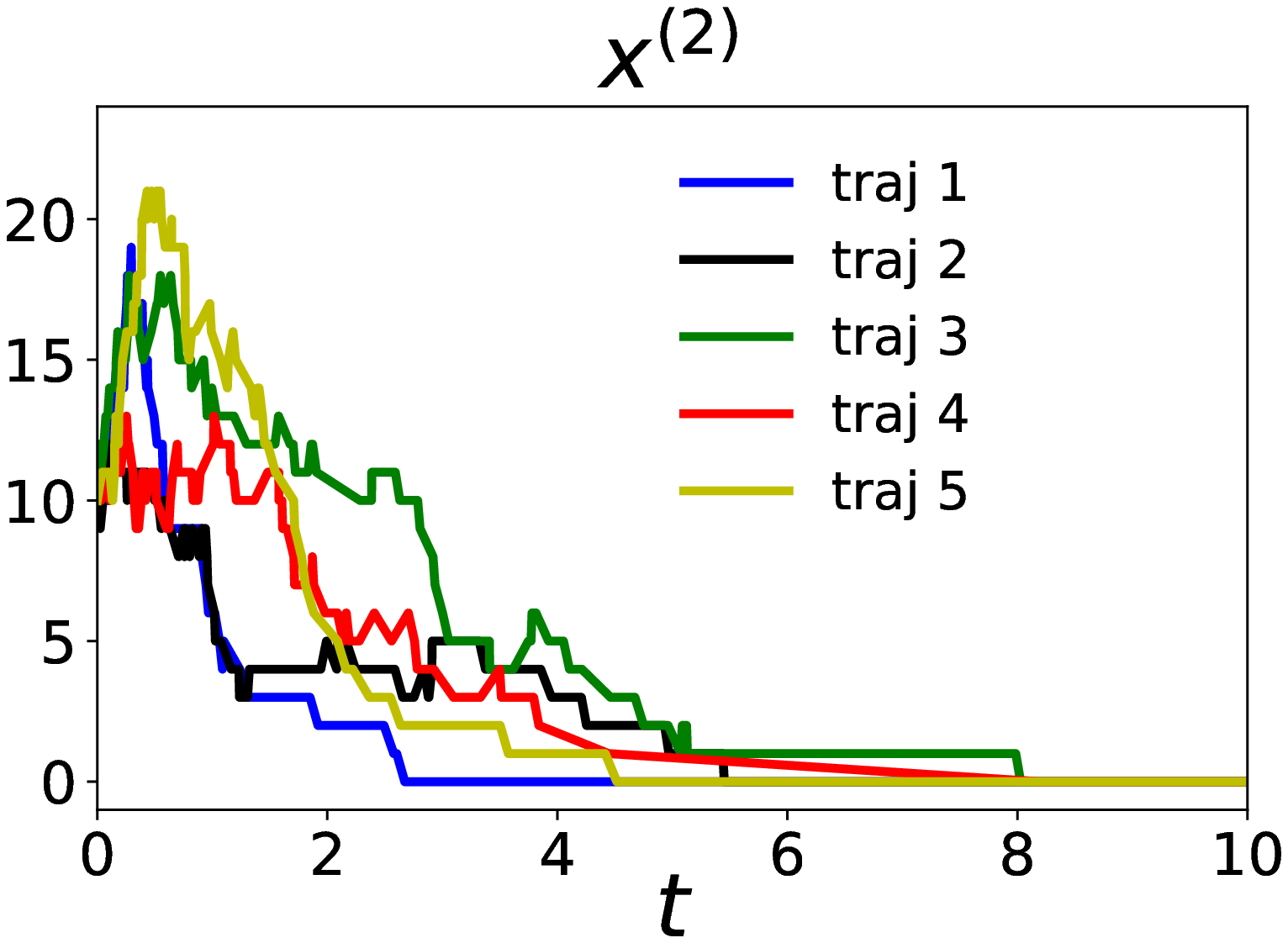}
  \centering
  \caption{
    Example 1. The evolution of the system's state $x=(x^{(1)}, x^{(2)})^\top$. Displayed are $5$ sample trajectories (of overall $100$ trajectories).
    \label{fig-traj-ex1}}
\end{figure}

With the prepared trajectory data, let us first consider the problem of learning the 
rate constants $\kappa_i$, $1 \le i \le 4$, assuming that the types of these
$4$ reactions are known. 
For this purpose, we consider the negative log-likelihood function
\begingroup
\small
\begin{align}
  \begin{split}
    -\frac{1}{QT}\ln\mathcal{L}^{(T)}(\bm{\omega}) 
    &=  -\frac{1}{QT}\sum_{q=1}^{Q}\bigg\{ \sum_{l=0}^{M^{(q)}-1}
    \ln\bigg[\sum_{j\in \mathcal{I}_{i_{l}^{(q)}}}
    \omega_{j}\,\varphi_{j}(y^{(q)}_{l})\bigg]
    + \sum_{l=0}^{M^{(q)}} t^{(q)}_{l}\bigg[ \sum_{j=1}^{N}
  \omega_{j}\,\varphi_{j}(y^{(q)}_{l})\bigg]\,\bigg\} \\
    &=  -\frac{1}{QT}\sum_{q=1}^{Q}\sum_{i=1}^K\bigg\{\sum_{k=1}^{M_i^{(q)}}
    \ln\bigg[\sum_{j\in \mathcal{I}_{i}}
    \omega_{j}\,\varphi_{j}(y^{(q)}_{l_k^{(q,i)}})\bigg]
    + \sum_{l=0}^{M^{(q)}} t^{(q)}_{l}\bigg[\sum_{j\in \mathcal{I}_i}
  \omega_{j}\,\varphi_{j}(y^{(q)}_{l})\bigg]\,\bigg\}\,,
\end{split}
    \label{opt-problem-ex1-task1}
    \end{align}
   \endgroup 
which is similar to \eqref{opt-problem-y-t-f}, except that in
\eqref{opt-problem-ex1-task1} we have taken all the $100$ trajectories into account. Specifically,  $q$ in
\eqref{opt-problem-ex1-task1} denotes the index of the trajectory, while 
the notation $M^{(q)}$, $M^{(q)}_i$, $i^{(q)}_{l}$, $y^{(q)}_{l}$, $t^{(q)}_{l}$,
$l^{(q,i)}_k$ has the same meaning (for the $q$th trajectory) as the corresponding notations $M$, $M_i$, $i_l$, $y_{l}$, $t_{l}$, and $l^{(i)}_k$ in
\eqref{opt-problem-y-t-f}, respectively.
Following the setting in Subsection~\ref{subsec-learn-rate}, in this example we have the parameter set
$\bm{\omega}=(\kappa_1, \kappa_2, \kappa_3, \kappa_4)^\top$, the index set
$\mathcal{I}_i=\{i\}$, $1 \le i \le 4$, as well as the functions given by
\begin{align*}
  \varphi_1(x) = x^{(1)}, \quad \varphi_2(x) = x^{(1)}x^{(2)}, \quad \varphi_3(x) = x^{(2)}, \quad \varphi_4(x) = x^{(1)}\,.
\end{align*}
Since each reaction channel only contains one single reaction, the minimizer of the objective
function \eqref{opt-problem-ex1-task1} can be computed explicitly using an
expression similar to \eqref{opt-omega-ni1}, and we get $\bm{\omega}^{(T)}=(0.98,~0.10,~0.97, ~0.91)^\top$,
  which is indeed close to the true parameters (see Table~\ref{ex1-task1-table}).

Let us now study the second learning task in
Subsection~\ref{subsec-unkown-structure} with the same trajectory data, where we assume that
the structure of the chemical reactions involved in the system is unknown as well. 
Notice that, by analyzing the trajectory data, in this case we can still figure out that there are in
total $2$ species and $4$ different reaction channels in the reaction network (see Table~\ref{ex1-info-data-table}).
In order to determine the propensity function of each reaction channel,
  based on Table~\ref{propensity-table} and the discussions in Remark~\ref{rmk-choice-of-basis}, we 
choose polynomials of degree at most $2$ for $x = (x^{(1)}, x^{(2)})^\top$, i.e.,  
\begin{align}
  \begin{split}
  &  \phi_1(x) = 1,~~ \phi_2(x) = x^{(1)},~~\phi_3(x) = x^{(2)}, \\
  &   \phi_4(x) = (x^{(1)})^2, ~~\phi_5(x) = x^{(1)}x^{(2)}, ~~
    \phi_6(x) = (x^{(2)})^2\,,
  \end{split}
  \label{ex1-phi-basis}
\end{align}
as basis functions.
The propensity functions of the reaction channels are approximated by 
\begingroup
\small
\setlength\abovedisplayskip{6pt}
\setlength\belowdisplayskip{6pt}
\begin{align}
  a_i^{(\epsilon)}(x\,;\bm{\omega}) =a_i^{(\epsilon)}(x\,;\bm{\omega}^{(i)}) = G_\epsilon\Big(\sum_{k=1}^6 \omega_{6(i-1)+k} \phi_k(x)\Big)\,,\quad 1 \le i\le 4\,,
  \label{ex1-ai}
\end{align}
\endgroup
where $G_\epsilon$ is defined in \eqref{g-eps} and we set $\epsilon=0.1$. In \eqref{ex1-ai}, the function $a_i^{(\epsilon)}$ depends on the
$6$ parameters $\bm{\omega}^{(i)} = \big(\omega_{6(i-1)+1}$, $\omega_{6(i-1)+2}$, $\dots$,
$\omega_{6(i-1)+6}\big)^\top$, and the same set of basis functions in
\eqref{ex1-phi-basis} is used for each of the $4$ channels. 

To determine the value of $\bm{\omega}=(\omega_1, \omega_2, \dots,
\omega_{24})^\top$, which consists of all the unknown parameters,
we follow the discussions in Remark~\ref{rmk-choice-of-basis} of
Subsection~\ref{subsec-unkown-structure} and solve the sparse minimization problems 
\begingroup
\small
\begin{equation}
  \begin{split}
    \min_{\bm{\omega}^{(i)} \in
    \mathbb{R}^{N_i}}\bigg\{-\frac{1}{QT}\sum_{q=1}^{Q}\bigg[
	\sum_{k=1}^{M^{(q)}_{i}}
    \ln G_\epsilon\bigg(\sum_{j\in \mathcal{I}_{i}}
    \omega_{j}\,\varphi_{j}(y^{(q)}_{l^{(q,i)}_{k}})\bigg)
     + \sum_{l=0}^{M^{(q)}} t^{(q)}_{l} G_\epsilon\bigg(\sum_{j \in \mathcal{I}_i}
  \omega_{j}\,\varphi_{j}(y^{(q)}_{l})\bigg)\,\bigg] + \lambda
  \|\bm{\omega}^{(i)}\|_1\bigg\}
\end{split}
    \label{opt-problem-ex1-task2}
    \end{equation}
    \endgroup
for each channel $\mathcal{C}_i$ separately, by applying Algorithm~\ref{fista-algo} in Appendix~\ref{app-0}. 
We choose the parameter $\lambda=0.2,~0.1,~0.01$ empirically, such that the sparsity and accuracy of the solution are balanced.
In each iteration step, evaluating the objective function in
\eqref{opt-problem-ex1-task2} as well as its derivative requires traversing every reaction along the $100$ trajectories. This part of the calculation is
performed in parallel using MPI in our code.
The iteration procedure continues until the relative difference between the
minimal and the maximal values of the objective function in the last $20$ iteration steps is smaller than $5 \cdot 10^{-8}$.
In this example, we run the code using $20$ processors in parallel and it 
takes only a few seconds to meet the convergence criterion.  

The final results are summarized in Table~\ref{ex1-task2-omega-table}. To make a comparison with the true parameters in \eqref{ex1-true-kappa}, we notice that, with the basis functions
in \eqref{ex1-phi-basis}, the true propensity functions of the $4$ reaction
channels in the system (see Table~\ref{ex1-cr-table} and Table~\ref{ex1-task1-table}) can be expressed as 
\begin{align}
  \begin{split}
    & a_1^{*}(x) = 1.0\, x^{(1)} = G_0(1.0\,\phi_2(x))\,,\quad 
     a_2^{*}(x) = 0.1\, x^{(1)}x^{(2)} = G_0(0.1\phi_5(x))\,,\\ 
    &a_3^{*}(x) = 1.0\, x^{(2)} = G_0(1.0\,\phi_3(x))\,, \quad
     a_4^{*}(x) = 0.9\,x^{(1)} = G_0(0.9\,\phi_2(x))\,, 
  \end{split}
  \label{ex1-true-ai-by-g0}
\end{align}
where $G_0(x)=\max(x,0)$. From the expressions above, we see that the propensity
functions in \eqref{ex1-ai}, with the estimated parameters in Table~\ref{ex1-task2-omega-table} (for $\lambda=0.1$ or $0.01$), indeed
provide reasonable approximations of the true propensity functions in \eqref{ex1-true-ai-by-g0}. 
  Comparing the results for different $\lambda$, we can observe that 
while the solution is sparser for $\lambda=0.2$ (e.g., coefficients corresponding to the basis $\phi_1\equiv 1$
in Table~\ref{ex1-task2-omega-table}),
the approximation of the true coefficients is better when $\lambda$ is smaller
(i.e., $\lambda=0.01$, underlined coefficients in Table~\ref{ex1-task2-omega-table}).

\begin{table}
    \caption{Example 1. The state change vectors $v$ of the $4$ reaction channels in the system and the numbers of occurrences of their activations within the $100$ trajectories are obtained by analyzing the trajectory data. \label{ex1-info-data-table}}
    \centering
    \scalebox{0.9}{
    \begin{tabular}{c|cccc}
      \hline
      Channel & $1$ & $2$ & $3$ & $4$ \\
      \hline
      Vector $v^\top$ & $(-1,0)$ & $(-1,1)$ & $(0, -1)$ & $(1,0)$ \\
      \hline
      No. of occurrences & $2296$ & $1778$ & $2777$ & $2135$ \\
      \hline
    \end{tabular}}
\end{table}

\begin{table}
    \caption{The first learning task in Example 1. The row with label ``True'' shows the parameters in \eqref{ex1-true-kappa} used to generate the $100$ trajectories of the reaction system. The row with label ``Estimated'' shows the parameters obtained by minimizing the negative log-likelihood function~\eqref{opt-problem-ex1-task1}.\label{ex1-task1-table}}
    \centering
    \scalebox{0.9}{
    \begin{tabular}{c|cccc}
      \hline
      & $\kappa_1$ & $\kappa_2$ & $\kappa_3$ & $\kappa_4$ \\
      \hline
      True & $1.0$ & $0.1$ & $1.0$ & $0.9$\\
      Estimated & $0.98$ & $0.10$ & $0.97$ & $0.91$\\
      \hline
    \end{tabular}}
\end{table}
  
\begin{table}
    \caption{The second learning task in Example 1. The parameters in the
    propensity functions \eqref{ex1-ai} of the $4$ channels are
    estimated by solving the sparse minimization problems~\eqref{opt-problem-ex1-task2}, with $\epsilon=0.1$ and
    $\lambda=0.2,\, 0.1,\, 0.01$, respectively. For each channel $\mathcal{C}_i$, $1 \le
    i \le 4$, the same set of basis functions in \eqref{ex1-phi-basis} is used in the estimation. 
    In each row, the estimated parameters $\bm{\omega}^{(i)} =
    \big(\omega_{6(i-1)+1},\, \omega_{6(i-1)+2},\, \dots, \omega_{6(i-1)+6})^\top$,
    which are involved in \eqref{ex1-ai} in front of the basis functions $1$, $x^{(1)}$, $x^{(2)}$, $(x^{(1)})^2$,
    $x^{(1)}x^{(2)}$, and $(x^{(2)})^2$, are shown. The parameter
    that has the largest absolute value within the same row is underlined.\label{ex1-task2-omega-table}}
    \centering
    \scalebox{0.9}{
    \begin{tabular}{c|c|cccccc}
      \hline
      Channel & $\lambda$ & $1$ & $x^{(1)}$ & $x^{(2)}$ & $(x^{(1)})^2$ & $x^{(1)}x^{(2)}$ & $(x^{(2)})^2$ \\
      \hline
      \multirow{3}{*}{$1$} 
      & $0.2$ & $-1.7\cdot 10^{-2}$ & \underline{$0.66$} & $0$ & $1.7 \cdot 10^{-2}$& $1.1 \cdot 10^{-2}$ & $1.7\cdot 10^{-4}$  \\
      & $0.1$ & $-1.2\cdot 10^{-1}$ & \underline{$0.84$} & $0$ & $6.6 \cdot 10^{-3}$& $6.7 \cdot 10^{-3}$ & $1.6 \cdot 10^{-4}$  \\
      & $0.01$ & $-0.24$ & \underline{$1.02$} & $2.6 \cdot 10^{-3}$ & $-2.4 \cdot 10^{-3}$& $1.4 \cdot 10^{-3}$ & $1.0 \cdot 10^{-4}$  \\
      \hline
      \multirow{3}{*}{$2$} 
      & $0.2$ & $-7.6\cdot 10^{-2}$ & $0$& $0$& $-3.8\cdot 10^{-4}$ & \underline{$0.10$} & $-2.4 \cdot 10^{-4}$  \\
      & $0.1$ & $-0.14$ & $0$& $0$& $-1.5\cdot 10^{-4}$ & \underline{$0.10$} & $0$  \\
      & $0.01$ & $-0.24$ & $1.8\cdot 10^{-2}$& $2.1\cdot 10^{-2}$& $-1.2\cdot 10^{-3}$ & \underline{$0.10$} & $-1.1 \cdot 10^{-3}$  \\
      \hline
      \multirow{3}{*}{$3$} 
      & $0.2$ & $0$ & $0$ & \underline{$0.73$} & $-2.0 \cdot 10^{-3}$ & $0$ & $2.0 \cdot 10^{-2}$  \\
      & $0.1$ & $-0.11$ & $-8.4\cdot 10^{-6}$ & \underline{$0.90$} & $-2.6 \cdot 10^{-3}$ & $0$ & $7.5 \cdot 10^{-3}$  \\
      & $0.01$ &$-0.25$ & $3.5\cdot 10^{-5}$ & \underline{$1.12$} & $-3.3 \cdot 10^{-3}$ & $-1.5\cdot 10^{-3}$ & $-6.7 \cdot 10^{-3}$  \\
      \hline
      \multirow{3}{*}{$4$} 
      & $0.2$ & $-1.7 \cdot 10^{-2}$ & \underline{$0.62$} & $0$ & $1.6\cdot 10^{-2}$ & $8.0\cdot 10^{-3}$ & $4.8 \cdot 10^{-4}$ \\
      & $0.1$ & $-0.11$ & \underline{$0.79$} & $9.9 \cdot 10^{-6}$ & $6.0\cdot 10^{-3}$ & $4.5\cdot 10^{-3}$ & $4.4\cdot 10^{-4}$ \\
      & $0.01$ & $-0.25$ & \underline{$0.96$} & $1.7 \cdot 10^{-6}$ & $-2.3\cdot 10^{-3}$ & $3.5\cdot 10^{-4}$ & $6.7 \cdot 10^{-4}$ \\
      \hline
    \end{tabular}}
\end{table}


\subsection{Example 2: predator-prey system}
\label{subsec-ex2}
  In the second example, we consider the predator-prey type reaction system in
  Table~\ref{ex2-cr-table}, which has two different species and $5$ chemical reactions~\cite{hkz2017}. 
  The system models the birth and death of two different species and is widely used as building block of more complicate chemical or biological systems.
  In contrast to the previous example where different reactions have different
  state change vectors, in the current case both the reaction \ce{$A$
  ->[\kappa_2] $\emptyset$} and the reaction \ce{$A$ + $B$ ->[\kappa_5] $B$}
  have the same state change vector $v=(-1,0)^\top$.

  In the first step, we generate the trajectory data of the system with the parameters 
\begingroup
\setlength\abovedisplayskip{6pt}
\setlength\belowdisplayskip{6pt}
  \begin{align}
  (\kappa_1,\,\kappa_2,\,\kappa_3,\,\kappa_4,\,\kappa_5) = (1.2,\, 0.3,\,
  0.8,\, 0.75,\,0.1)\,. 
    \label{ex2-true-kappa}
  \end{align}
  \endgroup
  Starting from the state $x=(25, 15)^\top$ at time $t=0$, $Q=100$ trajectories
  are simulated using SSA until the final time $T=10$, and $5$ of these $100$ trajectories are shown in Figure~\ref{fig-traj-ex2}. 
  After analyzing the trajectory data, we can identify the $4$ different
  reaction channels in the system as well as the numbers of occurrences of activations for each channel within the $100$ trajectories (see Table~\ref{ex2-info-data-table}).

\begin{table}
    \caption{\label{ex2-cr-table} Example 2. Chemical reaction system of
    predator-prey type. Two species $A$ (prey) and $B$ (predator) are involved
    in $5$ chemical reactions.
    The copy-numbers of $A, B$ are denoted by $x=(x^{(1)},
    x^{(2)})^\top$. 
    The $1$st and the $3$rd reactions model the replication (birth) of $A$ and $B$, respectively. 
  The $2$nd and the $4$th reactions model the depopulation (death) of $A$ and $B$, respectively.
  The $5$th reaction models the preying process of $B$ on $A$.
    Here, $\kappa_i$, $v$, and $a^*_{\mathcal{R}}(x)$ are the
    rate constant, the state change vector, and the propensity function of the reactions, respectively. The $2$nd and the $5$th reactions have the same
    state change vector $v=(-1,0)^\top$ and belong to the same reaction channel $\mathcal{C}_1$.}
    \centering
    \scalebox{0.9}{
    \begin{tabular}{clccl}
      No. & Reaction & $v^\top$ & Channel & $a_{\mathcal{R}}^*(x)$ \\
      \hline
      $1$ & \ce{$A$ ->[\kappa_1] $2A$}  & $(1,0)$ & $4$ & $\kappa_1 x^{(1)}$ \\
      $2$ & \ce{$A$ ->[\kappa_2] $\emptyset$}  & $(-1,0)$ & $1$ & $\kappa_2 x^{(1)}$ \\
      $3$ & \ce{$B$ ->[\kappa_3] $2B$}  & $(0,1)$ & $3$ & $\kappa_3 x^{(2)}$ \\
      $4$ & \ce{$B$ ->[\kappa_4] $\emptyset$} & $(0,-1)$ & $2$ & $\kappa_4 x^{(2)}$ \\
      $5$ & \ce{$A$ + $B$ ->[\kappa_5] $B$} & $(-1,0)$ & $1$ & $\kappa_5 x^{(1)}x^{(2)}$ \\
      \hline
    \end{tabular}}
\end{table}

\begin{figure}
  \includegraphics[width=0.42\textwidth]{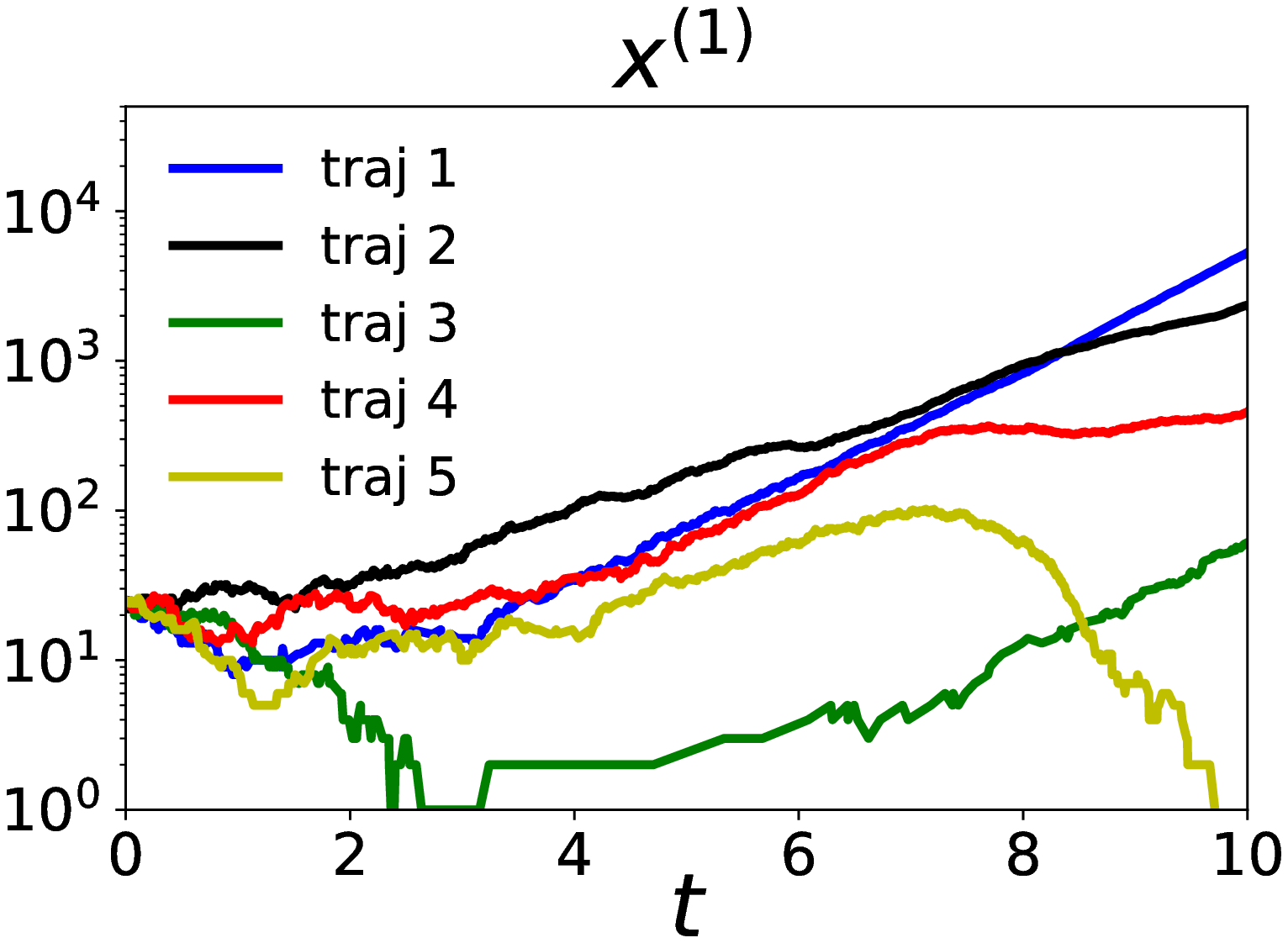}
  \includegraphics[width=0.42\textwidth]{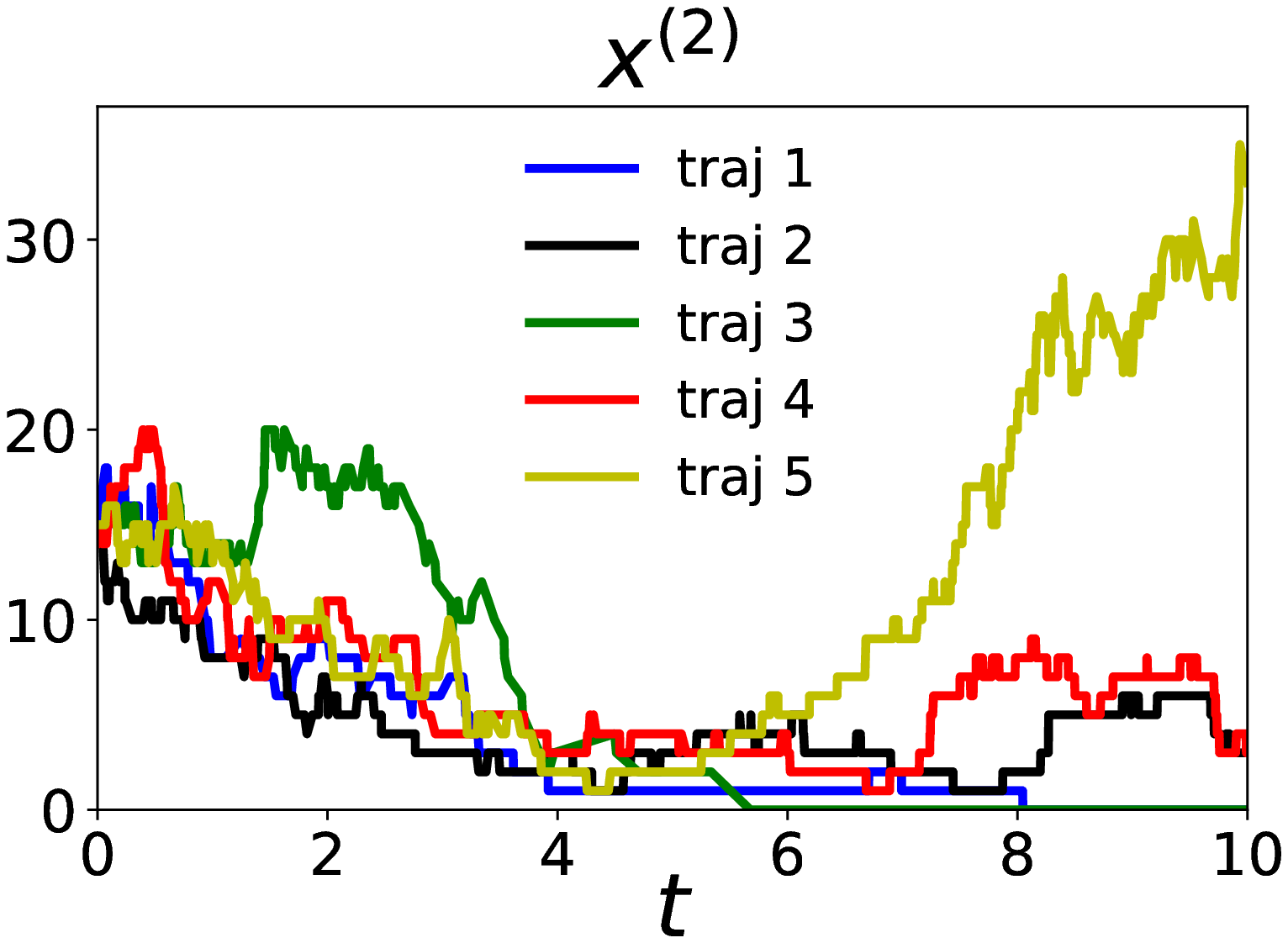}
  \centering
  \caption{Example 2. The evolution of the system's state $x=(x^{(1)}$,
  $x^{(2)})^\top$, shown are $5$ of the overall $100$ trajectories. Note that, unlike the trajectory data in Example 1 (Figure~\ref{fig-traj-ex1}), where the copy-numbers of both species $A$ and $B$ stay below $30$, in some trajectories of this example the copy-number of the species $A$ ($x^{(1)}$) may grow from $25$ to nearly $10^4$ within the time interval $[0,10]$.
  \label{fig-traj-ex2}}
\end{figure}

\begin{table}
    \caption{Example 2. Both the state change vectors $v$ of the $4$ reaction channels in the system and the numbers of occurrences of their activations within the $100$ trajectories can be obtained by analyzing the trajectory data. \label{ex2-info-data-table}}
    \centering
    \scalebox{0.9}{
    \begin{tabular}{c|cccc}
      \hline
      Channel & $1$ & $2$ & $3$ & $4$ \\
      \hline
      Vector $v^\top$ & $(-1,0)$ & $(0,-1)$ & $(0, 1)$ & $(1,0)$ \\
      \hline
      No. of occurrences & $22828$ & $14065$ & $14840$ & $42837$ \\
      \hline
    \end{tabular}}
\end{table}

  With the prepared trajectory data, we study the estimation of the parameters $\kappa_i$, $1
  \le i \le 5$, assuming that the structure of the $5$ reactions in Table~\ref{ex2-cr-table} is
  known (learning task 1). In the same way as in the previous
  example, we consider the minimization of the same negative log-likelihood
  function \eqref{opt-problem-ex1-task1}. The parameters $\kappa_1, \kappa_3,
  \kappa_4$ can be computed explicitly from the expression which is similar to
  \eqref{opt-omega-ni1} since the corresponding reaction channel contains only one single reaction, while 
  the parameters $\kappa_2, \kappa_5$, both of which are involved in the same
  channel $v^\top=(-1,0)$, can be found using a standard gradient descent method~\cite{snyman2018practical}.
  In the latter case, we choose the time step-size $\Delta t=10^{-3}$ and the initial
  values are set to $1.0$. In both cases, it only takes several seconds to run the code and the estimated parameters $\kappa_i$ are indeed
  very close to the true parameters (see Table~\ref{ex2-task1-table}).

\begin{table}
    \caption{The first learning task in Example 2. The row with label ``True'' shows the parameters in \eqref{ex2-true-kappa} which are used to generate the $100$ trajectories of the system. The row with label ``Estimated'' shows the parameters obtained by minimizing the negative log-likelihood function~\eqref{opt-problem-ex1-task1}.\label{ex2-task1-table}}
    \centering
    \scalebox{0.9}{
    \begin{tabular}{c|ccccc}
      \hline
      & $\kappa_1$ & $\kappa_2$ & $\kappa_3$ & $\kappa_4$& $\kappa_5$ \\
      \hline
      True & $1.2$ & $0.3$ & $0.8$ & $0.75$ & $0.1$ \\
      Estimated & $1.20$ & $0.30$ & $0.80$ & $0.76$ & $0.10$\\
      \hline
    \end{tabular}}
\end{table}

Next, we study the second learning task described in Subsection~\ref{subsec-unkown-structure},
  where our aim is to learn the propensity functions of the $4$ identified reaction channels without knowing
  the structure of the chemical reactions.  The propensity functions are approximated in the same way as in \eqref{ex1-ai}, 
  with the same set of basis functions in \eqref{ex1-phi-basis} and $\epsilon=0.1$. 
For each channel $\mathcal{C}_i$ and each $\lambda=0.2,\, 0.1,\,0.01$, the sparse minimization problem
\eqref{opt-problem-ex1-task2} is solved separately by ``FISTA with backtracking'' (Algorithm~\ref{fista-algo} in Appendix~\ref{app-0}), using the same number of processors (i.e., $20$) and the same convergence criterion as in the previous example.

   However, as shown in Figure~\ref{fig-traj-ex2}, the trajectory data in the
   current example exhibits further complexities, as the copy-number $x^{(1)}$
   of the species $A$ in the system varies significantly (from $25$ to
   nearly $10^4$) within the time interval $[0,10]$, unlike the trajectory data in the previous example, where the
   copy-numbers of the both species stay below $30$ (Figure~\ref{fig-traj-ex1}).
   As a result, in Table~\ref{ex2-task2-range-cj-table} we see that the different basis
   functions in \eqref{ex1-phi-basis} are of vastly different orders of magnitude when they are evaluated at the states contained in the $100$ trajectories.
   At the same time, in the numerical experiment we find that direct minimization of
   \eqref{opt-problem-ex1-task2} using FISTA does not converge at all for any
   of the $4$ reaction channels, due
   to the extremely small step-size between $10^{-11}$ and $10^{-8}$ (the step-size is determined by the algorithm itself; see Algorithm~\ref{fista-algo} in Appendix~\ref{app-0} and~\cite{fista2009}). 
   
To overcome this difficulty, we apply the preconditioning  idea discussed in Remark~\ref{rmk-preconditioning}.
Let $\varphi_j$ denote the basis functions, where $\varphi_j = \phi_k$, for $j=6(i-1) + k$, $1 \le k \le 6$. For each index $j$ belonging to the $i$th channel $\mathcal{C}_i$, we record the maximal values of $\varphi_j$ among all the states in the trajectory data at
which $\mathcal{C}_i$ has been activated. These maximal values are then
used to (empirically) determine the rescaling constants $c_j$, shown in
Table~\ref{ex2-task2-range-cj-table} such that  the functions
$\varphi_j/c_j$ after rescaling are roughly of the same order of magnitude. 
As discussed in Remark~\ref{rmk-preconditioning},  we solve the rescaled sparse minimization
problem, which is similar to \eqref{opt-problem-1-rescaled}, for each channel separately, and restore the
parameters $\bm{\omega}$ in the propensity functions. It turns out that the problems after rescaling become much easier to solve, because in this case 
the step-size is increased to $10^{-5}$ on average, which is $3$ to $6$ orders of
magnitude larger than the step-size in the unrescaled problem. It takes
less than $10$ minutes in total to meet the convergence criterion for all
$4$ reaction channels and the results are summarized in Table~\ref{ex2-task2-omega-table}.

To compare with the true parameters in \eqref{ex2-true-kappa}, notice that the
true propensity functions of the $4$ channels in Table~\ref{ex2-info-data-table} can be expressed as 
\begin{align}
  \begin{split}
    & a_1^{*}(x) = 0.3\,x^{(1)} + 0.1\,x^{(1)}x^{(2)} = G_0\big(0.3\,\phi_2(x)+0.1\,\phi_5(x)\big)\,,\\
    & a_2^{*}(x) = 0.75\, x^{(2)} = G_0\big(0.75\,\phi_3(x)\big)\,,\\ 
    &a_3^{*}(x) = 0.8\,x^{(2)} = G_0\big(0.8\,\phi_3(x)\big)\,, \quad
     a_4^{*}(x) = 1.2\,x^{(1)} = G_0\big(1.2\,\phi_2(x)\big)\,, 
  \end{split}
  \label{ex2-true-ai-by-g0}
\end{align}
where $G_0(x)=\max(x,0)$. From the expressions above, we can conclude that the propensity
functions in \eqref{ex1-ai}, together with the parameters given in Table~\ref{ex2-task2-omega-table}, indeed
approximate the true propensity functions in \eqref{ex2-true-ai-by-g0} quite
well. Comparing the results for $\lambda=0.01$,
we observe that the solutions are slightly sparser for $\lambda=0.2$ and
$\lambda=0.1$ (e.g., coefficients corresponding to the basis $\phi_1\equiv 1$
in Table~\ref{ex2-task2-omega-table}), 
while the approximation of the true coefficients is better when $\lambda$ is
smaller (i.e., $\lambda=0.01$, underlined coefficients in
Table~\ref{ex2-task2-omega-table}). Finally, we point out that the solution could be further improved if necessary, 
by using thresholding techniques (i.e., removing unimportant
basis functions)~\cite{Brunton-sindy} or cross-validation techniques (i.e.,
tuning $\lambda$)~\cite{sindy-sde}. 

\begin{table}[htbp]
    \caption{Example 2. As discussed in Remark~\ref{rmk-choice-of-basis},
    index $k$, $1\le k \le 6$, counts the different basis functions $\phi_k$,
    while index $j$, $1 \le j \le 24$, counts the basis functions $\varphi_j$
    for all $4$ channels. The same set of basis functions $\phi_k$ in
    \eqref{ex1-phi-basis} is used for each of the $4$ channels. For each $j$
    belonging to channel $\mathcal{C}_i$, i.e., $6(i-1)< j \le 6i$, we have
    the correspondence $\varphi_j = \phi_k$, if $j=6(i-1) + k$. See
    \eqref{varphi-j-phi-lambda}. For each channel $\mathcal{C}_i$, the column
    with label~``$\max \varphi_j$'' shows the maximal values of the $6$ basis functions $\phi_k$ (in different rows) evaluated on the trajectory data.
     The maximal values are computed among all the states in the trajectory data at which $\mathcal{C}_i$ has been activated. The rescaling constants $c_j$ are determined empirically depending on these maximal values such that the functions $\varphi_j/c_j$ are roughly of the same order of magnitude. 
    \label{ex2-task2-range-cj-table}} 
    \centering
    \scalebox{0.9}{
    \begin{tabular}{cc||cc|cc|cc|cc}
      \hline
      & & \multicolumn{2}{c|}{Channel $1$} & \multicolumn{2}{c|}{Channel $2$} & \multicolumn{2}{c|}{Channel $3$} & \multicolumn{2}{c}{Channel $4$} \\
      \hline
      $k$ & $\phi_k$ & $\max \varphi_j$ & $c_j$ & $\max \varphi_j$ & $c_j$ & $\max \varphi_j$ & $c_j$ & $\max \varphi_j$ & $c_j$ \\
      \hline
      $1$ & $1$ & $1$ & $1$  & $1$ & $1$ & $1$ & $1$ & $1$ & $1$  \\
      $2$ & $x^{(1)}$ & $5.3\cdot 10^{3}$ & $10$  & $2.2\cdot 10^{3}$ & $10$ & $2.1\cdot 10^{3}$ & $10$ & $5.3 \cdot 10^{3}$ & $50$  \\
      $3$ & $x^{(2)}$ & $41$ & $1$ & $104$ & $1$ & $103$ & $1$ & $38$ & $1$ \\
      $4$ & $(x^{(1)})^2$ & $2.8 \cdot 10^{7}$ & $50000$ & $4.8\cdot 10^{6}$ & $10000$ & $4.4 \cdot 10^{6}$ & $20000$ &  $2.8\cdot 10^{7}$ & $100000$ \\
      $5$ & $x^{(1)}x^{(2)}$ & $1.1\cdot 10^{4}$ & $100$ & $1.2\cdot 10^{4}$ & $100$ & $8.3\cdot 10^{3}$ & $20$ & $1.2\cdot 10^{4}$ & $100$ \\
      $6$ & $(x^{(2)})^2$ & $1.7\cdot 10^{3}$ & $5$ & $1.1\cdot 10^{4}$ & $100$ &
      $1.1\cdot 10^{4}$ & $50$ & $1.4 \cdot 10^{3}$ & $10$ \\
      \hline
    \end{tabular}}
\end{table}

\begin{table}[htbp]
    \caption{The second learning task in Example 2. The
    parameters in the propensity functions \eqref{ex1-ai} of the $4$ channels in Table~\ref{ex2-info-data-table} are estimated, with
    $\epsilon=0.1$ and $\lambda=0.2,\, 0.1,\, 0.01$, respectively.
As discussed in Remark~\ref{rmk-choice-of-basis}, for each channel $i$, $1 \le i \le 4$, the
same set of basis functions in \eqref{ex1-phi-basis} is used and the rescaled
  version of the sparse minimization problem~\eqref{opt-problem-ex1-task2} is
  solved, by
    rescaling the basis functions using the constants $c_j$ in Table~\ref{ex2-task2-range-cj-table}. 
    In each row, the estimated parameters $\bm{\omega}^{(i)} =
    \big(\omega_{6(i-1)+1},\, \omega_{6(i-1)+2},\, \dots, \omega_{6(i-1)+6})^\top$,
    which are involved in \eqref{ex1-ai} in front of the basis functions $1$, $x^{(1)}$, $x^{(2)}$, $(x^{(1)})^2$,
    $x^{(1)}x^{(2)}$, and $(x^{(2)})^2$, are shown for different $\lambda$. 
    The parameters that have relatively significant absolute values within the same row are underlined. \label{ex2-task2-omega-table}}
    \centering
    \scalebox{0.9}{
    \begin{tabular}{c|c|cccccc}
      \hline
      Channel & 
      $\lambda$ & $1$ & $x^{(1)}$ & $x^{(2)}$ & $(x^{(1)})^2$ & $x^{(1)}x^{(2)}$ & $(x^{(2)})^2$ \\
      \hline
      \multirow{3}{*}{$1$} 
      & $0.2$ & $0$ & \underline{$0.30$} & $-1.1 \cdot 10^{-2}$ & $8.4 \cdot 10^{-7}$& \underline{$0.10$} & $-2.8\cdot 10^{-4}$  \\
      & $0.1$ & $0$ & \underline{$0.30$} & $-2.2 \cdot 10^{-2}$ & $-6.2 \cdot 10^{-7}$& \underline{$0.10$} & $1.9\cdot 10^{-4}$  \\
      & $0.01$ & $-1.7\cdot 10^{-2}$ & \underline{$0.30$} & $-2.4 \cdot 10^{-2}$ & $-2.3 \cdot 10^{-6}$& \underline{$0.10$} & $2.4\cdot 10^{-4}$  \\
      \hline
      \multirow{3}{*}{$2$} 
      & $0.2$ & $0$ & $-1.1\cdot 10^{-3}$& \underline{$0.71$} & $2.8\cdot 10^{-7}$ & $4.1 \cdot 10^{-4}$ & $1.3 \cdot 10^{-3}$  \\
      & $0.1$ & $0$ & $-1.1\cdot 10^{-3}$& \underline{$0.73$}& $2.8\cdot 10^{-7}$ & $3.6 \cdot 10^{-4}$ & $8.4\cdot 10^{-4}$  \\
      & $0.01$ & $-1.1\cdot 10^{-3}$ & $-1.1\cdot 10^{-3}$& \underline{$0.75$}& $2.7\cdot 10^{-7}$ & $3.0\cdot 10^{-4}$ & $3.9 \cdot 10^{-4}$  \\
      \hline
      \multirow{3}{*}{$3$} 
      & $0.2$ & $0$ & $-2.3\cdot 10^{-4}$ & \underline{$0.76$} & $-3.3 \cdot 10^{-7}$ & $1.8\cdot 10^{-4}$ & $1.1 \cdot 10^{-3}$  \\
      & $0.1$ & $0$ & $-3.3\cdot 10^{-4}$ & \underline{$0.78$} & $-1.6 \cdot 10^{-7}$ & $9.6\cdot 10^{-5}$ & $5.8 \cdot 10^{-4}$  \\
      & $0.01$ & $-5.1\cdot 10^{-2}$ & $-1.5\cdot 10^{-4}$ & \underline{$0.80$} & $-1.6 \cdot 10^{-7}$ & $9.4\cdot 10^{-6}$ & $4.1 \cdot 10^{-5}$  \\
      \hline
      \multirow{3}{*}{$4$} 
      & $0.2$ & $0$ & \underline{$1.16$} & $-1.2 \cdot 10^{-2}$ & $1.9\cdot 10^{-5}$ & $4.8\cdot 10^{-3}$ & $-4.0 \cdot 10^{-5}$ \\
      & $0.1$ & $0$ & \underline{$1.17$} & $-1.7 \cdot 10^{-2}$ & $1.4\cdot 10^{-5}$ & $4.0\cdot 10^{-3}$ & $1.4 \cdot 10^{-4}$ \\
      & $0.01$ & $-0.13$ & \underline{$1.18$} & $-1.0 \cdot 10^{-2}$ & $1.0\cdot 10^{-5}$ & $3.5\cdot 10^{-3}$ & $7.6 \cdot 10^{-5}$ \\
      \hline
    \end{tabular}}
\end{table}


\subsection{Example 3: reaction network modeling intracellular viral infection}
\label{subsec-ex3}

In the third example, we consider the reaction network
in~\cite{SRIVASTAVA2002}, which models intracellular viral infection. We refer
the readers to~\cite{SRIVASTAVA2002} for the biological background 
and to~\cite{haseltine-rawlings-2002,ball06} for further details.
As shown in Table~\ref{ex3-cr-table}, the system consists of $4$ different species, i.e., the viral template (T), the viral genome (G), the viral
structure protein (S), and the virus (V). These species are involved in $6$ chemical reactions. 

First of all, starting from the state $x=(1,0,0,0)^\top$ at time $t=0$, $Q=10$ trajectories of
the system are generated using SSA until $T=100$, with the parameters 
\begingroup
\setlength\abovedisplayskip{6pt}
\setlength\belowdisplayskip{6pt}
\begin{align}
    (\kappa_1,\,\kappa_2,\,\kappa_3,\,\kappa_4,\,\kappa_5, \kappa_6) = (0.25,\, 0.001,\, 0.3,\, 100,\,2.0,\,0.1)\,
    \label{ex3-true-kappa}
\end{align}
\endgroup
in Table~\ref{ex3-cr-table}.
For illustration purposes, $5$ of these $10$ trajectories are shown in
Figure~\ref{fig-traj-ex3}. It can be observed that the copy-numbers $x^{(3)}$,
$x^{(4)}$ of $S, V$ may increase to $10^2$--$10^3$, while the
copy-numbers $x^{(1)}$, $x^{(2)}$ of $T$, $G$ remain relatively small (less than $20$)
within the time interval $[0,100]$.  After analyzing the trajectory data, we
can identify the $6$ reaction channels
  of the system. The numbers of occurrences of activations for each channel within the $10$ trajectories can be counted as well
  (see Table~\ref{ex3-info-data-table}). 

  With these trajectory data, we study the estimation of the parameters $\kappa_i$, $1
  \le i \le 6$, assuming that the structure of the $6$ reactions in
  Table~\ref{ex3-cr-table} is known (learning task 1). In the same way as we did in the previous two
  examples, the parameters are estimated by minimizing the same negative log-likelihood function \eqref{opt-problem-ex1-task1}. 
  Since each reaction channel contains only one reaction, the parameters $\kappa_i$
  can be directly computed (see \eqref{opt-omega-ni1}) and are indeed very close to the true parameters in
  \eqref{ex3-true-kappa}, as shown in Table~\ref{ex3-task1-table}.

In what follows, we continue to study the second learning task in Subsection~\ref{subsec-unkown-structure}, where we want to learn the
  propensity functions of the $6$ identified reaction channels in the system without knowing the structure of the chemical reactions. 
As discussed in Table~\ref{propensity-table} and Remark~\ref{rmk-choice-of-basis},
since there are $4$ different species in the system, we construct the following basis functions
  (i.e., polynomials of degree at most $2$)
\begingroup
\small
\begin{align}
  \begin{split}
    & \phi_1(x) = 1, \quad \phi_2(x) = x^{(1)}, \quad \phi_3(x) = x^{(2)}, \quad \phi_4(x) = x^{(3)}, \quad \phi_5(x) = x^{(4)}\,,\\
    &\phi_6(x) = (x^{(1)})^2, \quad \phi_7(x) = x^{(1)}x^{(2)}, \quad \phi_8(x) = x^{(1)}x^{(3)}, \quad \phi_9(x) = x^{(1)}x^{(4)}, \quad \\
    & \phi_{10}(x) = (x^{(2)})^2,  \quad \phi_{11}(x) = x^{(2)}x^{(3)}, \quad \phi_{12}(x) = x^{(2)}x^{(4)}, \quad \phi_{13}(x) = (x^{(3)})^2, \\
& \phi_{14}(x) = x^{(3)}x^{(4)}, \quad \phi_{15}(x) = (x^{(4)})^2\,,
  \end{split}
  \label{ex3-phi-basis}
\end{align}
\endgroup
where $x = (x^{(1)}, x^{(2)},x^{(3)},x^{(4)})^\top$, 
  to learn the propensity function of each reaction channel.
Similar to \eqref{ex1-ai} in the first example, the propensity functions of
the $6$ reaction channels are approximated by
\begingroup
\small
\setlength\abovedisplayskip{6pt}
\setlength\belowdisplayskip{6pt}
\begin{align}
  a_i^{(\epsilon)}(x\,;\bm{\omega}) =a_i^{(\epsilon)}(x\,;\bm{\omega}^{(i)}) =
  G_\epsilon\Big(\sum_{k=1}^{15} \omega_{15(i-1)+k} \phi_k(x)\Big)\,,\quad 1 \le i\le 6\,,
  \label{ex3-ai}
\end{align}
\endgroup
with $\epsilon=0.1$. 
For each $1 \le i \le 6$, the same sparse minimization problem in
\eqref{opt-problem-ex1-task2} is solved in order to determine the coefficients
$\bm{\omega}^{(i)}=\big(\omega_{15(i-1)+1}$, $\omega_{15(i-1)+2}$, $\dots$,
$\omega_{15(i-1)+15}\big)^\top$. From Table~\ref{ex3-task2-range-cj-table}, we can again observe that the maximal
values of the different basis functions in \eqref{ex3-phi-basis}, evaluated on the trajectory data, are of different orders of magnitude.    
 Therefore, the same rescaling strategy discussed in
 Remark~\ref{rmk-preconditioning} and in the previous example is applied to precondition the problem,
 using the rescaling constants $c_j$ in Table~\ref{ex3-task2-range-cj-table} which are determined empirically based on the maximal values of basis functions.
Notice that, since for different channels the basis functions attain similar
maximal values, the same set of rescaling constants is used for all the $6$ channels.
For each reaction channel, the rescaled minimization problem is solved in
parallel using $10$ processors, since the trajectory data only contains $10$ trajectories,
and the iteration procedure continues until the relative difference between 
the minimal and the maximal values of the objective function in the last $20$ iteration steps is smaller than $1.0 \cdot 10^{-7}$.
The estimated coefficients are summarized in Table~\ref{ex3-task2-omega-table}.
For each channel except channel $\mathcal{C}_4$, it takes around $10$
minutes to meet the convergence criterion, while for channel $\mathcal{C}_4$
it takes roughly two hours, because the corresponding solution of channel $\mathcal{C}_4$ has a large coefficient
(i.e., the underlined coefficient $92.1$ in Table~\ref{ex3-task2-omega-table}) which is very different from the zero initial guess.

To compare with the true propensity functions of the $6$ channels in Table~\ref{ex3-info-data-table}
with the true parameters in \eqref{ex3-true-kappa}, let us write the true propensity functions as
\begingroup
\setlength\abovedisplayskip{6pt}
\setlength\belowdisplayskip{6pt}
\begin{align}
  \begin{split}
    & a_1^{*}(x) = 0.25\,x^{(1)} = G_0\big(0.25\,\phi_2(x)\big)\,, \quad
     a_2^{*}(x) = 0.001\, x^{(2)}x^{(3)} = G_0\big(0.001\,\phi_{11}(x)\big)\,,\\ 
    &a_3^{*}(x) = 0.3\,x^{(3)} = G_0\big(0.3\,\phi_4(x)\big)\,, \quad
     a_4^{*}(x) = 100.0\,x^{(1)} = G_0\big(100.0\,\phi_2(x)\big)\,, \\
    & a_5^{*}(x) = 2.0\,x^{(1)} = G_0\big(2.0\,\phi_2(x)\big)\,, \quad
     a_6^{*}(x) = 0.1\,x^{(2)} = G_0\big(0.1\,\phi_3(x)\big)\,, 
  \end{split}
  \label{ex3-true-ai-by-g0}
\end{align}\endgroup 
where $G_0(x)=\max(x,0)$. From the expressions above, we can conclude that the propensity
functions in \eqref{ex3-ai}, together with the estimated parameters in
Table~\ref{ex3-task2-omega-table}, indeed provide good approximation of the true propensity
functions in \eqref{ex3-true-ai-by-g0}. 
  Note that in this numerical experiment we have empirically chosen different values of $\lambda$ for different channels since we only want to demonstrate that the true parameters can indeed be estimated with properly chosen $\lambda$. A more systematic way of choosing $\lambda$ is cross-validation~\cite{sindy-sde,cross-validation-geisser,hastie2009elements}.  See Remark~\ref{rmk-possible-extension}.
Finally, we point out that the solution could be further refined if necessary, by applying thresholding techniques (i.e., removing unimportant basis functions and then solving the minimization problem again)~\cite{Brunton-sindy}. 

\begin{table}
    \caption{\label{ex3-cr-table} Example 3. The reaction network models a
    type of intracellular viral infection~\cite{SRIVASTAVA2002}. There are $4$ different species in the system, i.e., the viral template (T), the viral genome (G), the viral structure protein (S), and the virus (V), which are involved in $6$ chemical reactions. The copy-numbers of $T$, $G$, $S$, and $V$ are denoted by the state vector $x=(x^{(1)}, x^{(2)}, x^{(3)}, x^{(4)})^\top$.}
    \centering
    \scalebox{0.9}{
    \begin{tabular}{clccl}
      No. & Reaction & $v^\top$ & Channel & $a_{\mathcal{R}}^*(x)$ \\
      \hline
      $1$ & \ce{$T$ ->[\kappa_1] $\emptyset$} & $(-1,0,0,0)$ & $1$ & $\kappa_1 x^{(1)}$ \\
      $2$ & \ce{$G$ + $S$ ->[\kappa_2] $V$} & $(0,-1,-1,1)$ & $2$ & $\kappa_2 x^{(2)}x^{(3)}$ \\
      $3$ & \ce{$S$ ->[\kappa_3] $\emptyset$} & $(0,0,-1,0)$ & $3$ & $\kappa_3 x^{(3)}$ \\
      $4$ & \ce{$T$ ->[\kappa_4] $T$ + $S$}  & $(0,0,1,0)$ & $4$ & $\kappa_4 x^{(1)}$ \\
      $5$ & \ce{$T$ ->[\kappa_5] $T$ + $G$}  & $(0,1,0,0)$ & $5$ & $\kappa_5 x^{(1)}$ \\
      $6$ & \ce{$G$ ->[\kappa_6] $T$}  & $(1,-1,0,0)$ & $6$ & $\kappa_6 x^{(2)}$ \\
      \hline
    \end{tabular}}
\end{table}

\begin{figure}
  \begin{tabular}{cc}
    \includegraphics[width=0.42\textwidth]{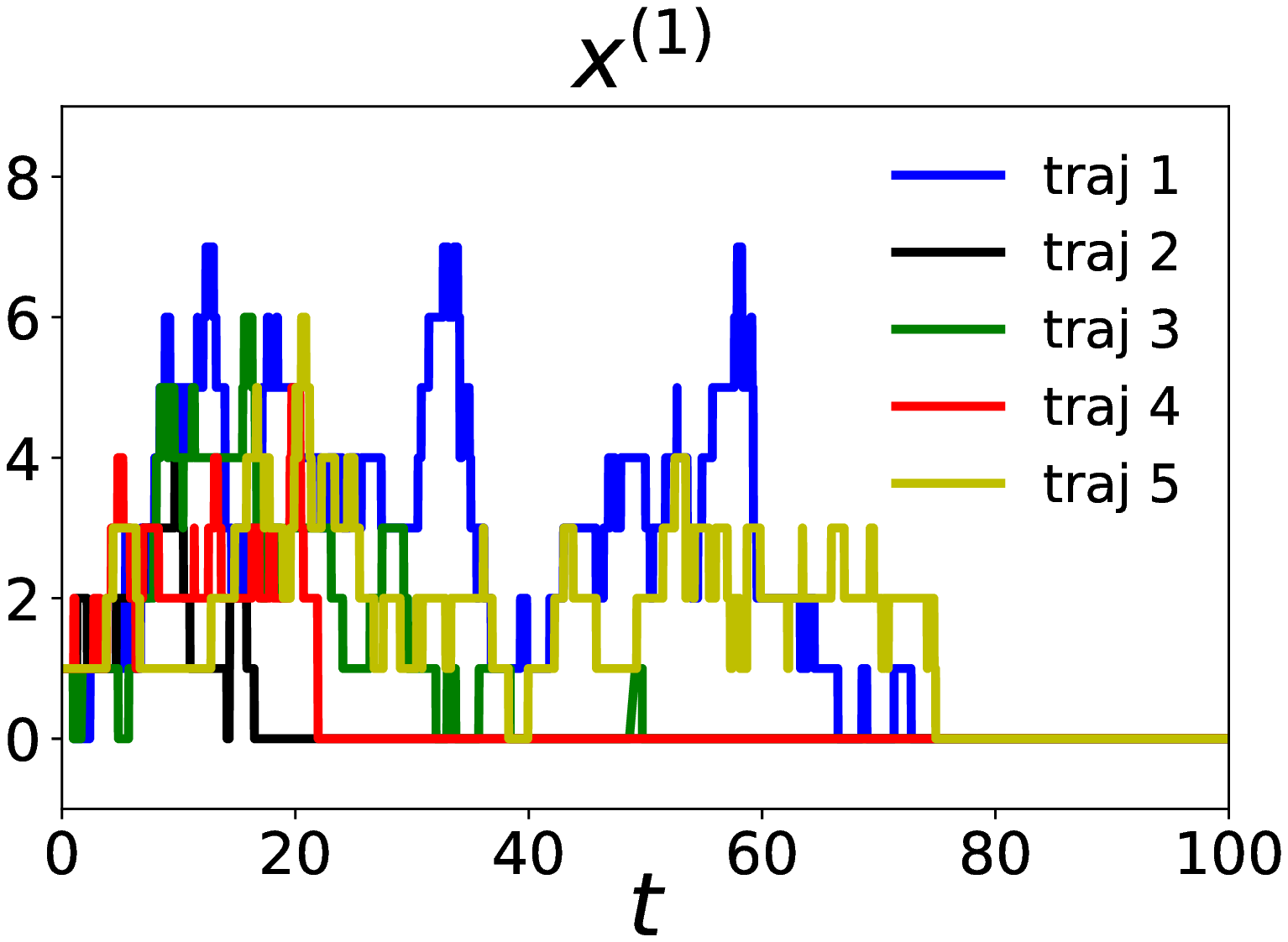}&
  \includegraphics[width=0.42\textwidth]{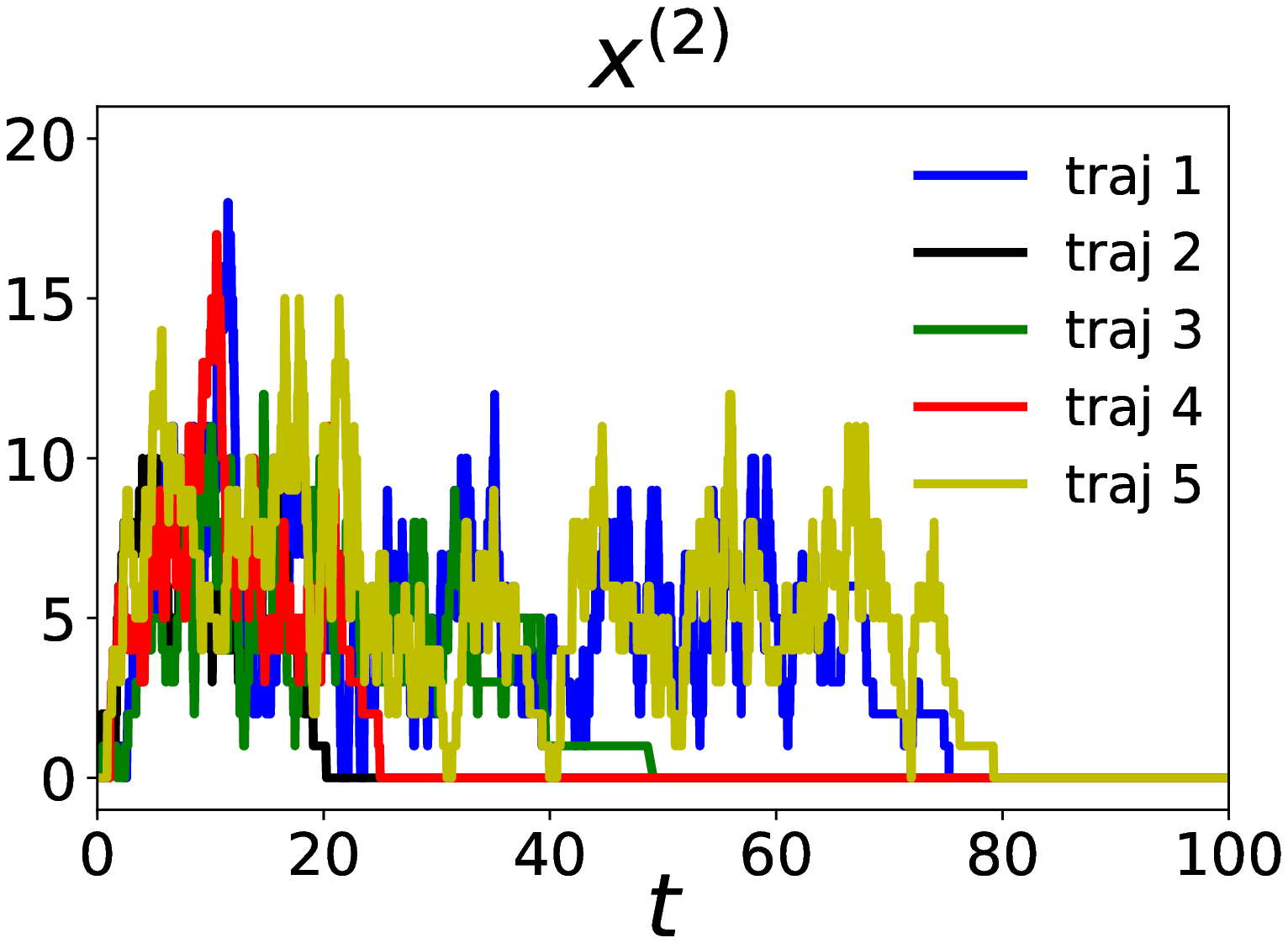}\\
    \includegraphics[width=0.42\textwidth]{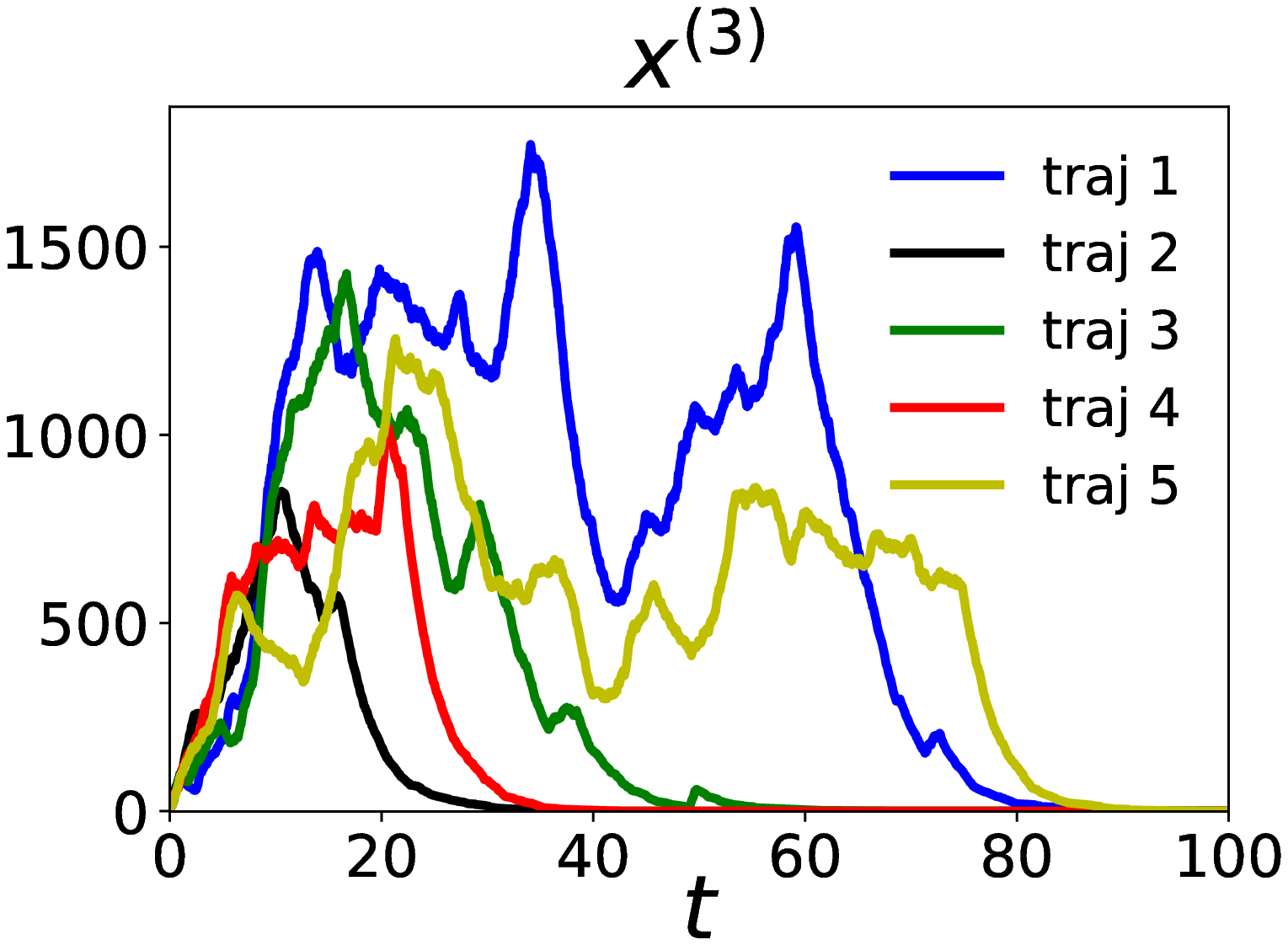}&
  \includegraphics[width=0.42\textwidth]{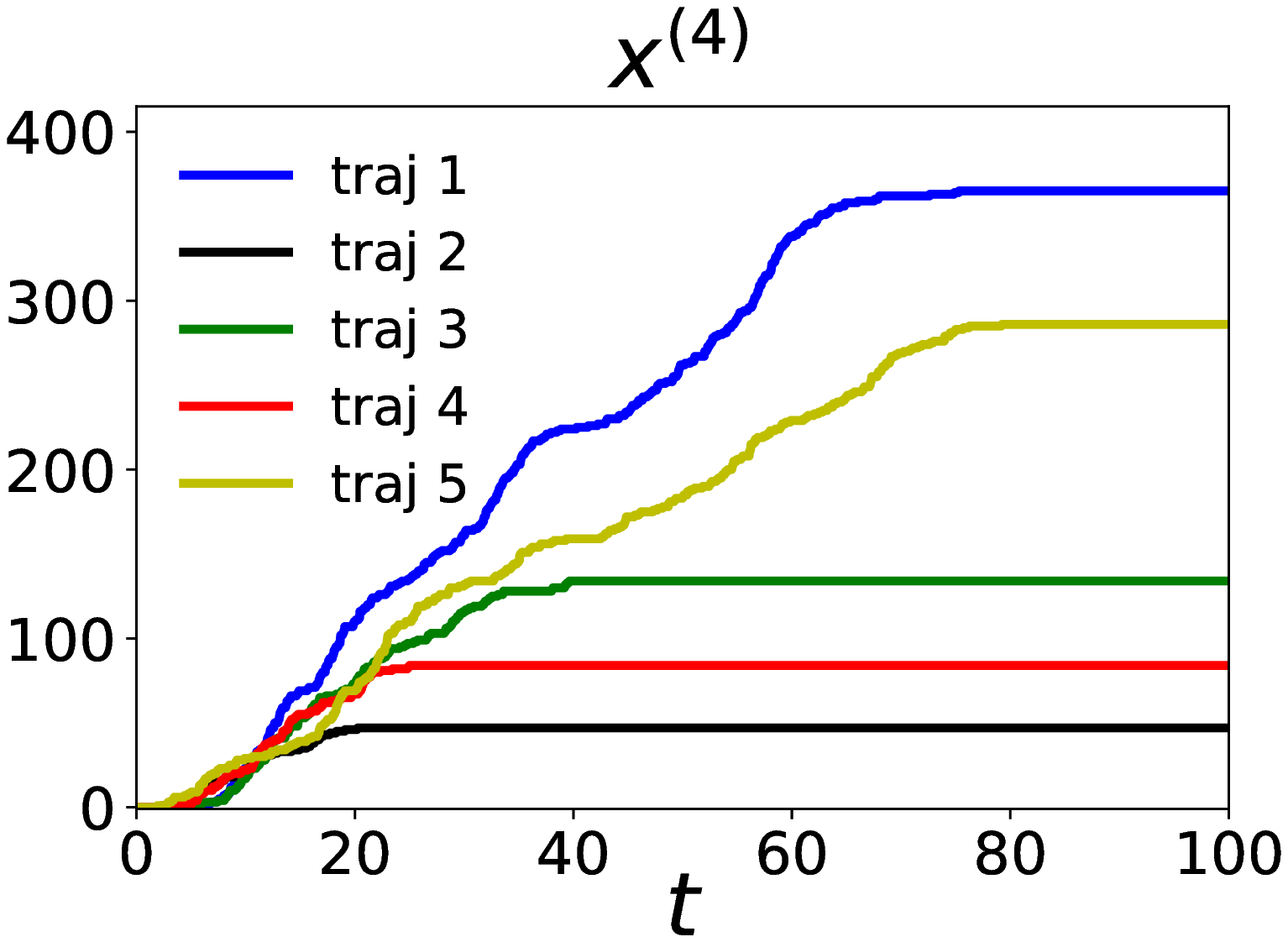}
  \end{tabular}
  \centering
  \caption{
    Example 3. The evolution of the system's state $x=(x^{(1)}, x^{(2)}, x^{(3)}, x^{(4)})^\top$. Shown are 5 of the overall 10 trajectories. The copy-numbers $x^{(3)}$, $x^{(4)}$ of $S, V$ can increase to $10^2$--$10^3$, while the copy-numbers $x^{(1)}$, $x^{(2)}$ of $T$, $G$ remain relatively small (less than $20$) within the time interval $[0,100]$. \label{fig-traj-ex3}}
\end{figure}

\begin{table}
    \caption{Example $3$. The state change vectors $v$ of the $6$ reaction channels in the system and the numbers of occurrences of their activations within the $10$ trajectories can be obtained by analyzing the trajectory data.  \label{ex3-info-data-table}}
    \centering \setlength\tabcolsep{0.6ex}
    \scalebox{0.9}{
    \begin{tabular}{c|cccccc}
      \hline
      Channel & $1$ & $2$ & $3$ & $4$ & $5$ & $6$ \\
      \hline
      Vector $v^\top$ & $(-1,0,0,0)$ & $(0, -1,-1,1)$ & $(0,0,-1,0)$ & $(0,0,1,0)$ & $(0,1,0,0)$ & $(1,-1,0,0)$\\
      \hline
      No. of occurrences & $214$ & $1534$ & $87942$ & $90130$ & $1743$ & $206$\\
      \hline
    \end{tabular}}
\end{table}

\begin{table}
    \caption{The first learning task in Example $3$. The row with label ``True'' shows the parameters in \eqref{ex3-true-kappa} which are used to generate the $10$ trajectories of the system. The row with label ``Estimated'' shows the parameters obtained by minimizing the negative log-likelihood function~\eqref{opt-problem-ex1-task1}. \label{ex3-task1-table}}
    \centering
    \scalebox{0.9}{
    \begin{tabular}{c|cccccc}
      \hline
      & $\kappa_1$ & $\kappa_2$ & $\kappa_3$ & $\kappa_4$& $\kappa_5$ & $\kappa_6$\\
      \hline
      True & $0.25$ & $0.001$ & $0.3$ & $100.0$ & $2.0$ & $0.1$\\
      Estimated & $0.24$ & $0.001$ & $0.30$ & $99.3$ & $1.92$ & $0.10$\\
      \hline
    \end{tabular}}
\end{table}

\begin{table}
    \caption{Example 3. For the reaction channels $\mathcal{C}_1,
    \mathcal{C}_2, \dots, \mathcal{C}_6$ in the system, the maximal values of
    the $15$ basis functions $\phi_k$ in \eqref{ex3-phi-basis} are shown in
    the columns with label ``Ch.$1$'', ``Ch.$2$'', $\dots$, and ``Ch.$6$'',
    respectively. The same set of basis functions $\phi_k$, $1\le k \le 15$,
    is used for each of the $6$ channels. As discussed in
    Remark~\ref{rmk-choice-of-basis}, index $k$ counts different basis
    functions $\phi_k$, while index $j$, $1 \le j \le 6 \cdot 15$, counts
    basis functions $\varphi_j$ for all the $6$ channels. For each $j$
    belonging to channel $\mathcal{C}_i$, i.e., $15(i-1)< j \le 15i$, we have
    the correspondence $\varphi_j = \phi_k$, if $j=15(i-1) + k$. See
    \eqref{varphi-j-phi-lambda}. For each channel $\mathcal{C}_i$, the column
    with label~``Ch.$i$'' shows the maximal values of the $15$ basis functions
    $\phi_k$ (in different rows) evaluated for the trajectory data. The
    maximal values are computed among all the states in the $10$ trajectories
    at which $\mathcal{C}_i$ has been activated. The rescaling constants $c_j$
    are determined empirically, such that after rescaling the basis functions
    are roughly of the same order of magnitude. Since the basis functions have
    similar maximal values in different channels, the same rescaling constants
    are used for all $6$ channels.
    \label{ex3-task2-range-cj-table}}
    \centering
    \scalebox{0.9}{
    \begin{tabular}{cc|c|c|c|c|c|c|c}
      \hline
      & & \multicolumn{6}{c|}{$\max \varphi_j$} & \\
      \cline{3-8}
      $k$ & $\phi_k$ & Ch.$1$ & Ch.$2$ & Ch.$3$ & Ch.$4$ & Ch.$5$ & Ch.$6$ & $c_j$\\
      \hline
      $1$ & $1$ & $1$ & $1$ & $1$ & $1$ & $1$ & $1$ & $1$ \\
      $2$ & $x^{(1)}$ & $9$ & $9$ & $9$ & $9$ & $9$ & $8$ & $1$ \\
      $3$ & $x^{(2)}$ & $17$ & $18$ & $18$ & $18$ & $17$ & $17$ & $1$ \\
      $4$ & $x^{(3)}$ & $1857$ & $1865$ & $1868$ & $1868$ & $1855$ & $1737$ & $10$ \\
      $5$ & $x^{(4)}$ & $363$ & $364$ & $365$ & $365$ & $362$ & $362$ & $3$ \\
      $6$ & $(x^{(1)})^2$ & $81 $ & $81$ & $81$ & $81$ & $81$ & $64$& $1$ \\
      $7$ & $x^{(1)}x^{(2)}$ & $102$ & $119$ & $119$ & $119$ & $112$& $98$ & $1$ \\
      $8$ &     $x^{(1)}x^{(3)}$ & $15786$ & $15795$ & $15804$ & $15804$ & $15723$& $12992$ & $100$ \\
      $9$ & $x^{(1)}x^{(4)}$ & $2254$ & $2247$ & $2254$ & $2254$ & $2254$& $1890$ & $20$ \\
      $10$ & $(x^{(2)})^2$ & $289$ & $324$ & $324$ & $324$ & $289$& $289$ & $2$\\
      $11$&     $x^{(2)}x^{(3)}$ & $20434$ & $26608$ & $26640$ & $26640$ & $24825$ & $20434$ & $200$ \\
      $12$ &     $x^{(2)}x^{(4)}$ & $2997$ & $3320$ & $3320$ & $3320$ & $2988$ & $3150$ & $30$\\
      $13$& $(x^{(3)})^2$ & $3.4 \cdot 10^{6}$ & $3.5\cdot 10^{6}$ & $3.5 \cdot 10^{6}$ & $3.5\cdot 10^{6}$ & $3.4\cdot 10^{6}$& $3.0 \cdot 10^{6}$ & $30000$\\
      $14$&     $x^{(3)}x^{(4)}$ & $514485$ & $515264$ & $516483$ & $516483$ & $514272$& $509288$ & $5000$ \\
      $15$&     $(x^{(4)})^2$ & $131769$ & $132496$ & $133225$ & $133225$ & $131044$& $131044$ & $1000$\\
      \hline
    \end{tabular}}
\end{table}
  
\begin{table}
    \caption{The second learning task in Example 3. The parameters in the
    propensity functions \eqref{ex3-ai} of the $6$ channels in Table~\ref{ex3-info-data-table} are estimated with $\epsilon=0.1$. In this example, different $\lambda$ have been chosen for different reaction channels. For each channel $\mathcal{C}_i$, $1 \le i \le 6$, the rescaled version of the sparse minimization problem~\eqref{opt-problem-ex1-task2} is solved by rescaling the basis functions using the constants $c_j$ in
    Table~\ref{ex3-task2-range-cj-table}. The same set of basis functions in \eqref{ex3-phi-basis} and the same set of rescaling constants are used in estimating the parameters for all the channels. In each column, the estimated parameters $\bm{\omega}^{(i)} = \big(\omega_{15(i-1)+1},\, \omega_{15(i-1)+2},\, \dots, \omega_{15(i-1)+15})^\top$,
    which are involved in \eqref{ex3-ai} in front of the basis functions $\phi_k$ are shown. The parameters that have relatively significant absolute values within the same column are underlined.
    \label{ex3-task2-omega-table}}
    \centering
    \scalebox{0.85}{
    \begin{tabular}{cc|c|c|c|c|c|c}
      \hline
      & & Ch.$1$ & Ch.$2$ & Ch.$3$ & Ch.$4$ & Ch.$5$ & Ch.$6$ \\
      \cline{3-8}
    $k$  & $\phi_k$ & $\lambda = 0.01$ & $\lambda =10$ & $\lambda =0.1$ & $\lambda =0.005$ & $\lambda =0.005$ & $\lambda =0.01$ \\
      \hline
      $1$& $1$ & $-2.6\cdot 10^{-2}$  & $0$ & $0$ & $-8.4\cdot 10^{-1}$ & $-9.0\cdot 10^{-2}$ & $-1.1 \cdot 10^{-2}$ \\
      $2$&  $x^{(1)}$ & \underline{$0.28$}  & $0$ & $0$ & \underline{$92.1$} & \underline{$1.81$} & $0$ \\
     $3$& $x^{(2)}$ & $1.8 \cdot 10^{-2}$ & $0$ & $0$ & $-3.3 \cdot 10^{-3}$ 
      & $1.4 \cdot 10^{-2}$ & \underline{$0.11$} \\
     $4$& $x^{(3)}$ & $-4.3\cdot 10^{-4}$ & $0$ & \underline{$0.30$} & $8.8 \cdot 10^{-3}$ 
      & $1.0 \cdot 10^{-3}$ & $2.6\cdot 10^{-5}$\\
    $5$&  $x^{(4)}$ & $-5.0 \cdot 10^{-4}$ & $0$ & $-4.1\cdot 10^{-4}$ & $-5.2 \cdot 10^{-3}$ 
      & $-2.2 \cdot 10^{-3}$& $-3.2 \cdot 10^{-4}$\\
    $6$&  $(x^{(1)})^2$ & $0$ & $0$ & $0$ & $-2.2 \cdot 10^{-4}$ 
      & $-3.7 \cdot 10^{-2}$ & $-2.0 \cdot 10^{-2}$\\
    $7$&  $x^{(1)}x^{(2)}$ & $1.0\cdot 10^{-2}$ & $0$ & $0$ & $3.3\cdot 10^{-1}$ 
      & $1.2 \cdot 10^{-3}$ & $9.4 \cdot 10^{-3}$\\
    $8$&  $x^{(1)}x^{(3)}$ & $-1.3\cdot 10^{-4}$ & $4.5\cdot 10^{-5}$ & $-1.6\cdot 10^{-3}$ & $5.2 \cdot 10^{-3}$ 
      & $3.6 \cdot 10^{-4}$& $4.1 \cdot 10^{-6}$\\
    $9$&  $x^{(1)}x^{(4)}$ & $2.3\cdot 10^{-4}$ & $0$ & $-6.7\cdot 10^{-4}$ & $1.2 \cdot 10^{-2}$
      & $7.4 \cdot 10^{-4}$& $-2.3 \cdot 10^{-4}$\\
     $10$& $(x^{(2)})^2$ & $-3.7\cdot 10^{-3}$ & $0$ & $0$ & $-8.8 \cdot 10^{-3}$
      & $-3.7 \cdot 10^{-3}$& $-7.4 \cdot 10^{-3}$\\
     $11$& $x^{(2)}x^{(3)}$ & $-2.1\cdot 10^{-5}$ & \underline{$9.5\cdot 10^{-4}$} & $6.1\cdot 10^{-4}$ &  $8.3 \cdot 10^{-5}$ 
      &  $7.4 \cdot 10^{-5}$&  $8.5 \cdot 10^{-5}$\\
     $12$& $x^{(2)}x^{(4)}$ & $2.2\cdot 10^{-4}$ & $9.5\cdot 10^{-5}$ & $6.5 \cdot 10^{-4}$ &  $1.4 \cdot 10^{-3}$ 
      & $-5.4 \cdot 10^{-4}$&  $-2.4 \cdot 10^{-5}$\\
    $13$&  $(x^{(3)})^2$ & $6.9\cdot 10^{-7}$ & $-1.7\cdot 10^{-7}$ & $1.2\cdot 10^{-7}$ &  $-1.0 \cdot 10^{-5}$ 
      & $-1.8 \cdot 10^{-6}$& $-1.9 \cdot 10^{-7}$\\
     $14$& $x^{(3)}x^{(4)}$ & $-1.5\cdot 10^{-6}$ & $6.4\cdot 10^{-7}$ & $-1.3\cdot 10^{-6}$ &  $-3.0 \cdot 10^{-5}$ 
      & $4.6 \cdot 10^{-6}$& $7.5 \cdot 10^{-7}$\\
     $15$& $(x^{(4)})^2$ & $1.5\cdot 10^{-6}$ & $-4.3\cdot 10^{-7}$ & $4.0\cdot 10^{-7}$ &  $8.9 \cdot 10^{-6}$ 
      & $4.1 \cdot 10^{-7}$& $5.5 \cdot 10^{-7}$\\
      \hline
    \end{tabular}}
\end{table}

\section{Asymptotic analysis of the two learning tasks}
\label{sec-asymptotics}
In this section, we consider the two learning tasks introduced in
Section~\ref{sec-inverse} when $T\rightarrow +\infty$. 
Although we are mainly interested in the second learning task and the
corresponding minimization problem \eqref{opt-problem-1}, 
in Subsection~\ref{subsec-infinite-many-data} we start with the first learning task, because it is
highly relevant to the second learning task. The analysis in Subsection~\ref{subsec-infinite-many-data}
will be useful when we study the second learning task in Subsection~\ref{subsec-asymptotic-sparse-learning}.
The proofs of the various results will be given in Appendix~\ref{app-3}.

Assume the true propensity functions of the underlying chemical reaction system
are $a^*_i, a^*$ in \eqref{a-sum-ai-true}, and recall that the system's state $X(t)$ satisfies the dynamical equation \eqref{eqn-of-xt}, where
$\mathcal{P}_i$, $1 \le i \le K$, are independent unit Poisson processes.
For most of the results in this section, we will make the following assumptions about the system. 
  Readers are referred to~\cite{meyn1993markov} for the study of the ergodicity of stochastic systems.
\begin{assump}
  The state space $\mathbb{X}$ is a finite set.
  \label{assump-finite-space}
\end{assump}
\begin{assump}
  $X(t)$ is ergodic on $\mathbb{X}$. It has a unique invariant distribution
  $\pi$, such that $\pi(x)>0, ~ \forall x \in \mathbb{X}$. 
  \label{assump-ergodicity}
\end{assump}
\begin{remark}
  Assumption~\ref{assump-finite-space} simplifies the analysis in this section. In particular, it implies that any function on $\mathbb{X}$, e.g., the basis function $\varphi_j$, is bounded.  
For many systems in chemical reaction applications, the state spaces, which
  although can be large, are indeed finite. This is especially the case
  when there are conservation relations in the reactions of the system.
  At the same time, we also expect the analysis presented below can be extended to systems whose
  state space is an infinite set, after taking into account additional technical
  issues. 
  \label{rmk-for-assump-finite-set-and-ergodicity}
\end{remark}

Our asymptotic analysis of the limit $T\rightarrow +\infty$ combines both
techniques from the large sample
theory~\cite{ferguson1996course,lehmann2004elements} in statistics and the
limit theorems for stochastic processes~\cite{ethier1986markov}. 
In particular, we rely on the important fact that the log-likelihood functions in \eqref{opt-problem-y-t-f} and
\eqref{likelihood-t-eps}, as well as their derivatives, can be expressed as
integrations with respect to the counting processes
\begin{equation}
    R_i(t) = \mathcal{P}_i\Big(\int_0^t a_i^*\big(X(s)\big)\,ds\Big)\,, 
  \label{r0-ri}
\end{equation}
and the corresponding compensated Poisson processes (martingales)
\begin{equation}
    \widetilde{R}_i(t) = R_i(t) - \int_0^t
    a^*_i\big(X(s)\big)\,ds \,,
\label{compensate-r0-ri}
\end{equation}
 where $1 \le i \le K$ and $t \ge 0$.
As an example, it is apparent that the process $R_i$ is related to $M_i$ in
\eqref{sum-of-mi}, i.e., the total activation number of the channel $\mathcal{C}_i$ within the
time $[0, T]$, since 
\begingroup
\setlength\abovedisplayskip{6pt}
\setlength\belowdisplayskip{6pt}
\begin{align}
      M_i = R_i(T)\,,\quad 1 \le i \le K\,,\quad \mbox{and}\qquad      M = \sum_{i=1}^K M_i = \sum_{i=1}^K R_i(T)\,.
    \label{m-mi-t}
\end{align}
\endgroup
We refer the readers to Appendix~\ref{app-2} for two limit results concerning integrations
with respect to the processes $R_i$ and $\widetilde{R}_i$ when $T\rightarrow +\infty$.

\subsection{Learning task 1: analysis of the log-likelihood maximizer}
\label{subsec-infinite-many-data}
In this subsection, we consider the first learning task in Subsection~\ref{subsec-learn-rate}.
Recall that $\bm{\omega}^* = (\omega^*_{1}, \omega^*_2, \dots,
\omega^*_N)^\top$ is the true parameter vector such that \eqref{ar-expression} holds and that
\begingroup
\setlength\abovedisplayskip{8pt}
\setlength\belowdisplayskip{8pt}
\begin{align}
  a^*_{i}(x) = a_i(x\,;\,\bm{\omega}^*)\,, \quad \forall x \in \mathbb{X}, \quad 1 \le i \le K\,.
   \label{ai-expression-repeat}
\end{align}
\endgroup
For fixed $T>0$, $\bm{\omega}^{(T)}$ denotes the solution of the minimization
problem \eqref{opt-problem-y-t}. We will study the asymptotic convergence of $\bm{\omega}^{(T)}$ to $\bm{\omega}^{*}$, as $T\rightarrow +\infty$.
It should be pointed out that the consistency of maximum likelihood estimation has been well studied in the statistics
community~\cite{wald1949,ferguson1996course,van2000asymptotic}. We refer the
readers to~\cite{mle-ctsp-feigin1976, keiding1975} for the asymptotic study of maximum likelihood estimation for continuous-time stochastic processes.

Let us first express the log-likelihood function $\ln\mathcal{L}^{(T)}$ in \eqref{opt-problem-y-t-f} and its derivatives
using the processes in \eqref{r0-ri} and \eqref{compensate-r0-ri}. For the
log-likelihood function, since the trajectory of the system is piecewise constant, we have
\begingroup\small
\begin{align} \label{exp-log-likelihood-r}
    &~~~ -\ln\mathcal{L}^{(T)}(\bm{\omega}) \nonumber \\
    &= -\sum_{i=1}^K\int_0^T \Big[\ln a_i\big(X(s)\,;\,\bm{\omega}\big)\Big]\,dR_i + \int_0^T
  a\big(X(s)\,;\,\bm{\omega}\big)\,ds \\
&= -\sum_{i=1}^K\int_0^T \Big[\ln
  a_i\big(X(s)\,;\,\bm{\omega}\big)\Big]\,d\widetilde{R}_i + \sum_{i=1}^K\int_0^T
  \Big[a_i\big(X(s)\,;\,\bm{\omega}\big)
  - a_i^*\big(X(s)\big) \ln a_i\big(X(s)\,;\,\bm{\omega}\big) \Big] \,ds\,, \nonumber
\end{align}
\endgroup
while for its first order derivatives in \eqref{opt-problem-euler-lagrange}, we obtain
\begingroup\small
\begin{align}
  \begin{split}
     \mathcal{M}^{(T)}_{j}(\bm{\omega}) &= - \int_0^T
     \frac{\varphi_{j}\big(X(s)\big)}{
      a_i\big(X(s)\,;\,\bm{\omega}\big)} \,dR_i(s) + \int_0^T
    \varphi_{j}\big(X(s)\big)\,ds \\
    &= - \int_0^T \frac{\varphi_{j}\big(X(s)\big)}{a_i\big(X(s)\,;\,\bm{\omega}\big)}
 \,d\widetilde{R}_i(s)  + \int_0^T
    \varphi_{j}\big(X(s)\big) \bigg[1 - 
    \frac{a_i^*\big(X(s)\big)}{
      a_i\big(X(s)\,;\,\bm{\omega}\big)}\bigg]\,ds\,, 
\end{split}
\label{exp-1st-derivative-r}
\end{align}
\endgroup
where $1 \le j \le N$ and $i$ is the index of channel such that $j \in \mathcal{I}_i$. Similarly, the second order
derivatives in \eqref{hessian} can be expressed as
\begingroup\small
\begin{align}
  \begin{split}
    &	\frac{\partial^2 \big(-\ln\mathcal{L}^{(T)}\big)}{\partial
	\omega_{j}\partial \omega_{j'}}(\bm{\omega})\\
	&=
	\int_0^T 
\frac{\varphi_{j}\big(X(s)\big)\,\varphi_{j'}\big(X(s)\big)}{a_i^2\big(X(s)\,;\,\bm{\omega}\big)} \,dR_i(s) \\
    &=
	\int_0^T 
\frac{\varphi_{j}\big(X(s)\big)\,\varphi_{j'}\big(X(s)\big)}{a_i^2\big(X(s)\,;\,\bm{\omega}\big)}
    \,d\widetilde{R}_i(s) + 
    	\int_0^T
	\frac{\varphi_{j}\big(X(s)\big)\,\varphi_{j'}\big(X(s)\big)}{a_i^2\big(X(s)\,;\,\bm{\omega}\big)}\,
	a_i^*\big(X(s)\big)\,ds\,,
  \end{split}
  \label{exp-2nd-derivative-r}
\end{align}
\endgroup
for two indices $1 \le j, j' \le N$ when there is a common channel index $i$, $1 \le i \le
K$, such that $j, j' \in \mathcal{I}_i$, and otherwise  
\begingroup\small
\setlength\abovedisplayskip{8pt}
\setlength\belowdisplayskip{8pt}
\begin{align*}
  \frac{\partial^2 \big(-\ln\mathcal{L}^{(T)}\big)}{\partial
  \omega_{j}\partial \omega_{j'}}(\bm{\omega}) = 0 \,,
\end{align*}
\endgroup
when $j \in \mathcal{I}_i$ and $j' \in \mathcal{I}_{i'}$ where 
$1 \le i \neq i' \le K$ are two different channel indices.

In particular,  \eqref{exp-log-likelihood-r} and \eqref{exp-1st-derivative-r} become simpler when $\bm{\omega} = \bm{\omega}^{*}$, and we have 
\begingroup\small
\begin{align}
  \begin{split}
     -\ln \mathcal{L}^{(T)}\big(\bm{\omega}^{*}\big) &= -\sum_{i=1}^K\int_0^T
     \ln a_i^*\big(X(s)\big)\,d\widetilde{R}_i + \sum_{i=1}^K\int_0^T a_i^*\big(X(s)\big) \Big[1 - \ln a_i^*\big(X(s)\big) \Big] \,ds \,,\\
     \mathcal{M}^{(T)}_{j}(\bm{\omega}^{*}) &= - 
  \int_0^T \frac{\varphi_{j}\big(X(s)\big)}{a_i^*\big(X(s)\big)}
    \,d\widetilde{R}_i(s)\,, \qquad  \forall~j \in \mathcal{I}_i\,.
 \end{split}
 \label{exp-derivatives-special}
\end{align}
\endgroup
Let us first recall the law of large numbers (LLN) for the unit
Poisson processes $\mathcal{P}_i$, $1 \le i \le K$, which states that~\cite{Anderson2011}
\begingroup\small
  \begin{align}
    \lim_{t\rightarrow +\infty} \sup_{ u\le u_0} \Big|\frac{\mathcal{P}_i(ut)}{t} - u\Big|
    = 0\,,\quad a.s., \quad \forall\,u_0 > 0\,. 
    \label{lln-poisson}
  \end{align}
\endgroup
It allows us to study the simple case when the reaction channel $\mathcal{C}_i$
contains a single reaction.
\begin{proposition}
Given $1\le i \le K$, suppose $N_i = 1$ and $\mathcal{I}_i=\{j\}$, for
  some $1 \le j \le N$. Assume that 
  \begingroup\small
\setlength\abovedisplayskip{8pt}
\setlength\belowdisplayskip{8pt}
  \begin{align}
    \lim_{T\rightarrow +\infty} \int_0^T \varphi_{j}\big(X(s)\big)\,ds = +\infty, \quad a.s.
    \label{ni1-divergence-assump}
  \end{align}
 \endgroup 
  Then $\lim\limits_{T\rightarrow +\infty} \omega^{(T)}_{j} =
  \omega^{*}_{j}$, almost surely. 
  \label{prop-for-ni-is-1}
\end{proposition}

Note that Assumption~\ref{assump-finite-space} and
Assumption~\ref{assump-ergodicity} are actually not necessary in Proposition~\ref{prop-for-ni-is-1}. In
what follows, we study the case when $N_i > 1$, i.e., when more than one
reactions belong to the same reaction channel $\mathcal{C}_i$. We need to further make the following two assumptions.
\begin{assump}
There is a unique vector $\bm{\omega}^{*} \in \mathbb{R}^N$, such that \eqref{ai-expression-repeat} is satisfied.
  \label{assump-task1-uniqueness}
\end{assump}
\begin{assump}
  The basis functions $\varphi_j$, $1\le j \le N$, are nonnegative on $\mathbb{X}$.
  \label{assump-task1-varphi-positive}
\end{assump}

As a consequence of Assumption~\ref{assump-task1-uniqueness}, we have the
following lemma which concerns the uniqueness of $\bm{\omega}^{(T)}$, when $T$ is sufficiently large. 
    \begin{lemma}
      Suppose that Assumptions~\ref{assump-finite-space}, \ref{assump-ergodicity}, \ref{assump-task1-uniqueness}, \ref{assump-task1-varphi-positive} hold. With probability one, the minimization problem
      \eqref{opt-problem-y-t}--\eqref{opt-problem-y-t-f} has a unique solution
      $\bm{\omega}^{(T)}$, when $T$ is sufficiently large. 
      \label{lemma-uniqueness-omega-t}
    \end{lemma}

    To proceed, we will need the Kullback--Leibler divergence between two probability distributions~\cite{mackay-info-theory-2002}. 
    It is known that the Kullback--Leibler divergence is nonnegative and
    it equals zero if and only if the two distributions are identical.
In particular, for the probability distributions whose density functions are
    $\psi$ and $p$ in \eqref{phi-p-density}, the Kullback--Leibler divergences
    can be computed as
    \begingroup
    \small
\begin{align}
  \begin{split}
    D_{KL}\Big(\psi\big(\cdot\,;\,x,\bm{\omega}'\big)\,\Big|\,
	\psi\big(\cdot\,;\,x,\bm{\omega}\big)\Big) 
	&= \int_0^{+\infty} \ln
	\frac{\psi\big(t\,;\,x,\bm{\omega}'\big)}{\psi\big(t\,;\,x,\bm{\omega}\big)}
	\psi\big(t\,;\,x,\bm{\omega}'\big)\,dt \\
	&=
	-\ln\frac{a(x\,;\,\bm{\omega})}{a(x\,;\,\bm{\omega}')} + 
	\frac{a(x\,;\,\bm{\omega})}{a(x\,;\,\bm{\omega}')} - 1\,, \\
D_{KL}\Big(p\big(\cdot\,;\,x,\bm{\omega}'\big)\,\Big|\,
	p\big(\cdot\,;\,x,\bm{\omega}\big)\Big) 
	&= \sum_{i=1}^K \ln
	\frac{p\big(i\,;\,x,\bm{\omega}'\big)}{p\big(i\,;\,x,\bm{\omega}\big)}
	p\big(i\,;\,x,\bm{\omega}'\big)  \\
	&= \ln\frac{a(x\,;\,\bm{\omega})}{a(x\,;\,\bm{\omega}')} -
	\sum_{i=1}^K
	\frac{a_i(x\,;\,\bm{\omega}')}{a(x\,;\,\bm{\omega}')} 
	\ln\frac{a_i(x\,;\,\bm{\omega})}{a_i(x\,;\,\bm{\omega}')}\,,
      \end{split}
      \label{d-kl-phi-p}
\end{align}
\endgroup
respectively, where $x \in \mathbb{X}$ and  $\bm{\omega}$, $\bm{\omega}'$ are two parameter vectors in \eqref{omega-vector}.

The convergence of $\bm{\omega}^{(T)}$ towards $\bm{\omega}^{*}$ as $T\rightarrow +\infty$ is established in the following result.
\begin{proposition}
  Suppose that Assumptions~\ref{assump-finite-space}, \ref{assump-ergodicity}, \ref{assump-task1-uniqueness}, \ref{assump-task1-varphi-positive} hold.
  \begin{enumerate}[wide]
    \item
      For any vector $\bm{\omega}$ in \eqref{omega-vector}, we have  
      \begingroup
      \small
      \begin{align*}
	\begin{split}
	&~~~	\lim_{T\rightarrow +\infty} 
\frac{\ln \mathcal{L}^{(T)}(\bm{\omega}) -
	\ln\mathcal{L}^{(T)}(\bm{\omega}^{*})}{T}  \\
	    &=\sum_{x\in\mathbb{X}} \bigg[
	    a\big(x\,;\,\bm{\omega}^{*}\big) - a(x\,;\,\bm{\omega}) + 
	    \sum_{i=1}^K
	    \bigg(a_i\big(x\,;\,\bm{\omega}^{*}\big)
	\ln\frac{a_i(x\,;\,\bm{\omega})}{a_i\big(x\,;\,\bm{\omega}^{*}\big)}
	\bigg)\bigg]\, \pi(x)\,\\
	    &=-\sum_{x\in\mathbb{X}}
	\bigg[D_{KL}\Big(\psi\big(\cdot\,;\,x,\bm{\omega}^{*}\big)\,\Big|\,
	\psi\big(\cdot\,;\,x,\bm{\omega}\big)\Big) +
	D_{KL}\Big(p\big(\cdot\,;\,x,\bm{\omega}^{*}\big)\,\Big|\,
	p\big(\cdot\,;\,x,\bm{\omega}\big)\Big)\bigg]
	a\big(x\,;\,\bm{\omega}^{*}\big)\,\pi(x) \\
	  &\le 0 \,.
	\end{split}
      \end{align*}
      \endgroup
    \item
      Let $\bm{\omega}^{(T)}=(\omega^{(T)}_1, 
      \omega^{(T)}_2,\dots, \omega^{(T)}_N)^\top$ be the unique
      minimizer of the problem \eqref{opt-problem-y-t}, such that
      $\omega^{(T)}_j \ge 0$ for each $1 \le j \le N$.
      With probability one, it holds that $\lim\limits_{T\rightarrow +\infty} \bm{\omega}^{(T)} = \bm{\omega}^{*}$.
  \end{enumerate}
    \label{prop-omega-limit-general}
\end{proposition}

We now study the asymptotic normality of the sequence $\bm{\omega}^{(T)}$ as $T\rightarrow +\infty$. We have the following result.
\begin{proposition}
Suppose that Assumptions~\ref{assump-finite-space}, \ref{assump-ergodicity}, \ref{assump-task1-uniqueness}, \ref{assump-task1-varphi-positive} hold.
  Let $\mathcal{F}$ be the $N\times N$ matrix whose entries are 
  \begingroup
  \small
  \begin{align}
    \mathcal{F}_{j,j'} = 
      \begin{cases}
	\,\sum\limits_{x \in \mathbb{X}} \frac{\varphi_{j}(x) \varphi_{j'}(x)}
	{a_i(x\,;\,\bm{\omega}^{*})}\,\pi(x)\,, &  
	\mbox{if}~~j, j' \in \mathcal{I}_i \,,~\mbox{for some}~~1 \le i \le
	K\,,\\[12pt]
	\,0\,, &  \mbox{otherwise} \,, 
      \end{cases}
      \label{covariance-c-varphi}
  \end{align}
  \endgroup
  for $1 \le j,\,j' \le N$.  Then, as $T \rightarrow +\infty$, $\sqrt{T}
  \big(\bm{\omega}^{(T)} - \bm{\omega}^{*}\big)$ converges in
  distribution to  $\mathcal{Z}\sim \mathcal{N}(\bm{0},
  \mathcal{F}^{-1})$, i.e., $\mathcal{Z}$ is a Gaussian random variable whose
  mean equals zero and whose covariance matrix is $\mathcal{F}^{-1}$. 
  \label{prop-asymptotic-normality}
\end{proposition}
\subsection{Learning task 2: asymptotic analysis of the sparse optimization problem}
\label{subsec-asymptotic-sparse-learning}
Based on the analysis in Subsection~\ref{subsec-infinite-many-data}, in this subsection we study the
minimizer $\bm{\omega}^{(T,\epsilon,\lambda)}$ of the sparse minimization
problem \eqref{opt-problem-1} as $T\rightarrow +\infty$, where both
$\epsilon=\epsilon(T)$
and $\lambda=\lambda(T)$ depend on $T$.

Recall that $\psi^*, p^*$ are the probability densities (distributions) in \eqref{phi-p-density-true}.
With the convention $G_0(z) = \lim\limits_{\epsilon \rightarrow 0+}
G_\epsilon(z) = \max(z,0)$, for $z \in \mathbb{R}$, we will denote 
\begingroup
\small
\setlength\abovedisplayskip{6pt}
\setlength\belowdisplayskip{6pt}
  \begin{align*}
    a_i^{(0)}\big(x\,;\,\bm{\omega}\big) =
     \max\Big(\sum_{j\in \mathcal{I}_i}
    \omega_{j} \varphi_{j}(x),0\Big)\,, \qquad 
a^{(0)}\big(x\,;\,\bm{\omega}\big) =
     \sum_{i=1}^K\max\Big(\sum_{j\in \mathcal{I}_i}
    \omega_{j} \varphi_{j}(x),0\Big)\,,
\end{align*}
\endgroup
and, correspondingly,  
\begingroup
\small
\begin{align*}
  \begin{split}
    \psi^{(0)}(t\,;\,x, \bm{\omega}) &= a^{(0)}(x\,;\,\bm{\omega}) \exp\big(- a^{(0)}(x\,;\,\bm{\omega}) t\big)\,, \quad t \ge 0\,,\\
    p^{(0)}(i\,;x,\bm{\omega}) &=
    \frac{a^{(0)}_i(x\,;\,\bm{\omega})}{a^{(0)}(x\,;\,\bm{\omega})}\,, \quad 1 \le i \le K\,.
  \end{split}
\end{align*}
\endgroup

Instead of Assumption~\ref{assump-task1-uniqueness}, here we assume that the
set of basis functions is chosen such that the underlying (true) system can be uniquely parameterized.
\begin{assump}
There is a unique vector $\bm{\omega}^{*} \in \mathbb{R}^N$, such that
\begingroup
  \small
\setlength\abovedisplayskip{6pt}
\setlength\belowdisplayskip{6pt}
\begin{align}
  a_i^*(x) = a_i^{(0)}(x\,;\,\bm{\omega}^{*})\,, \quad \forall~x\in
  \mathbb{X}\,,\quad 1 \le i \le K\,.
  \label{true-ai-can-be-represented}
\end{align}
\endgroup
  \label{assump-task2-omega-true}
\end{assump}
We also need the following assumption in order to guarantee the boundedness of
$\bm{\omega}^{(T,\epsilon, \lambda)}$.  
\begin{assump}
  For each $1 \le i \le K$, assume that the index set is $\mathcal{I}_i = \big\{j_1, j_2, \dots, j_{N_i}\big\}$.
  $\bm{\eta}^{(k)} = (\eta^{(k)}_1, \eta^{(k)}_2,\dots, \eta^{(k)}_{N_i})^\top
  \in \mathbb{R}^{N_i}$, $k\ge 1$, is a sequence of vectors satisfying
  $\lim\limits_{k\rightarrow +\infty} \|\bm{\eta}^{(k)}\|_2 = +\infty$.
  Then $\exists x \in \mathbb{X}$ such that
  $a^*_i(x) > 0$ and 
  \begingroup
  \small
\setlength\abovedisplayskip{6pt}
\setlength\belowdisplayskip{6pt}
  \begin{align}
    \lim_{k\rightarrow +\infty}\Big|\sum_{l=1}^{N_i} \eta^{(k)}_l
    \varphi_{j_l}(x)\Big| = +\infty\,. 
    \label{go-to-inf-at-some-state}
    \end{align}
  \endgroup
  \label{assump-boundedness}
\end{assump}
\begin{remark}
  In fact, under Assumption~\ref{assump-task2-omega-true} (uniqueness of $\bm{\omega}^*$),  one can argue by
  contradiction and show that there exists $x \in \mathbb{X}$ such that \eqref{go-to-inf-at-some-state} holds.
  Therefore, Assumption~\ref{assump-boundedness} simply further asserts that $a^*_i(x)$ is positive.
  In particular, Assumption~\ref{assump-boundedness} is not needed if $a^*_i(x)>0$ is true for all $x\in \mathbb{X}$ and $1 \le i \le K$.
  \label{rmk-explain-assump-boundedness}
\end{remark}

Similar to \eqref{exp-1st-derivative-r} and \eqref{exp-2nd-derivative-r}, it
will be helpful to express the derivatives of the log-likelihood function in
\eqref{likelihood-t-eps} using the processes $R_i$, $\widetilde{R}_i$ in \eqref{r0-ri} and \eqref{compensate-r0-ri}.
For the first order derivative \eqref{log-likelihood-derivative-lambda}, we have  
\begingroup
\small
\begin{align}
  \begin{split}
    & ~~~ \mathcal{M}^{(T,\epsilon)}_{j}(\bm{\omega}) \\
     &= - \int_0^T
      (\ln G_\epsilon)'\Big(\sum\limits_{j'\in \mathcal{I}_i}
    \omega_{j'}\,\varphi_{j'}(X(s))\Big)\,\varphi_{j}\big(X(s)\big)\,dR_i(s) + \int_0^T
G_\epsilon'\Big(\sum\limits_{j'\in \mathcal{I}_i}
    \omega_{j'}\,\varphi_{j'}(X(s))\Big)
    \varphi_{j}\big(X(s)\big)\,ds \\
    &= 
  - \int_0^T
      (\ln G_\epsilon)'\Big(\sum\limits_{j'\in \mathcal{I}_i}
    \omega_{j'}\,\varphi_{j'}(X(s))\Big)\,\varphi_{j}\big(X(s)\big)\,d\widetilde{R}_i(s) \\
     & ~~~ + \int_0^T
G_\epsilon'\Big(\sum\limits_{j'\in \mathcal{I}_i} \omega_{j'}\,\varphi_{j'}(X(s))\Big)
    \varphi_{j}\big(X(s)\big)\,\bigg[1-\frac{
    a^*_i(X(s))}{G_\epsilon(\sum\limits_{j'\in \mathcal{I}_i}
    \omega_{j'}\,\varphi_{j'}(X(s)))}\bigg] ds\,,
  \end{split}
 \label{exp-1st-derivative-r-eps}
\end{align}
\endgroup
for $j \in \mathcal{I}_i$. For the second order derivatives, we have 
\begingroup
\small
\begin{align}
    &\frac{\partial^2 \big(-\ln \mathcal{L}^{(T,\epsilon)}\big)}{\partial \omega_{j}\partial \omega_{j'}}(\bm{\omega}) \label{exp-2nd-derivative-r-eps}
    \\
	=& -\int_0^T (\ln G_\epsilon)''\Big(\sum\limits_{k\in \mathcal{I}_i}
    \omega_{k}\,\varphi_{k}(X(s))\Big)
    \varphi_{j}\big(X(s)\big)\,\varphi_{j'}\big(X(s)\big) \,dR_i(s) \notag \\
    & + \int_0^T 
	G_\epsilon''\Big(\sum\limits_{k\in \mathcal{I}_i} \omega_{k}\,\varphi_{k}(X(s))\Big)
\varphi_{j}\big(X(s)\big)\,\varphi_{j'}\big(X(s)\big) \,ds\notag \\
    =& -\int_0^T 
	(\ln G_\epsilon)''\Big(\sum\limits_{k\in \mathcal{I}_i}
    \omega_{k}\,\varphi_{k}(X(s))\Big)
    \varphi_{j}\big(X(s)\big)\,\varphi_{j'}\big(X(s)\big) \,d\widetilde{R}_i(s) \notag\\
    &+ \int_0^T 
\varphi_{j}\big(X(s)\big)\,\varphi_{j'}\big(X(s)\big) 
    \bigg\{
	G_\epsilon''
    \Big(\sum\limits_{k\in \mathcal{I}_i} \omega_{k}\varphi_{k}(X(s))\Big)
	-\Big[(\ln G_\epsilon)''
    \Big(\sum\limits_{k\in \mathcal{I}_i} \omega_{k}\varphi_{k}(X(s))\Big)\Big]
    a^{*}_i(X(s))\bigg\}\,\,ds \notag
\end{align}
\endgroup
when there is an index $i$, $1 \le i \le K$, such that $j, j' \in \mathcal{I}_i$, and otherwise  
\begingroup
\small
\begin{align*}
  \frac{\partial^2 \big(-\ln \mathcal{L}^{(T,\epsilon)}\big)}{\partial
  \omega_{j}\partial \omega_{j'}}(\bm{\omega}) = 0 \,,
\end{align*}
\endgroup
when $j \in \mathcal{I}_i, \, j' \in \mathcal{I}_{i'}$, for two different
indices $1 \le i \neq i' \le K$.

The following technical lemma addresses the boundedness of the minimizers of the minimization problem \eqref{opt-problem-1}.
\begin{lemma}
  Suppose that Assumptions~\ref{assump-finite-space}, \ref{assump-ergodicity}, and \ref{assump-task2-omega-true} hold. 
  The parameter $\epsilon=\epsilon(T)$ satisfies $\lim\limits_{T\rightarrow +\infty} \epsilon(T) = 0$.
  Let $\bm{\omega}^{(T,\epsilon,\lambda)}$ be the minimizer of the
  minimization problem \eqref{opt-problem-1} and $\mathcal{L}^{(T,\epsilon)}$
  be the likelihood function in \eqref{likelihood-t-eps}. 	  Then, for each index $i$,  $1 \le i \le K$, and $x \in \mathbb{X}$, such that $a^*_i(x) > 0$, we have
  \begingroup
  \small
	  \begin{align}
0 < \liminf\limits_{T\rightarrow +\infty}
      \Big(\sum_{j' \in \mathcal{I}_i} \omega^{(T,\epsilon, \lambda)}_{j'}
	    \varphi_{j'}(x)\Big) 
	    \le 
	\limsup\limits_{T\rightarrow +\infty}
      \Big(\sum_{j' \in \mathcal{I}_i} \omega^{(T,\epsilon, \lambda)}_{j'}
	    \varphi_{j'}(x)\Big) < +\infty 
	    \,, \quad a.s.
	    \label{omega-lower-bound}
	  \end{align}
	  \endgroup
	  Assuming furthermore Assumption~\ref{assump-boundedness} holds, then the
	  sequence $\bm{\omega}^{(T,\epsilon, \lambda)}$ is bounded for $T>0$.
	  \label{bounded-log-likelihood-imply-bounded-inf-sup}
\end{lemma}

Now we are ready to state the asymptotic results for the sequence
$(\bm{\omega}^{(T, \epsilon, \lambda)})_{T > 0}$, as $T\rightarrow +\infty$.
Readers are referred to Appendix~\ref{app-3} for their proofs.
\begin{theorem}
  Suppose that Assumptions~\ref{assump-finite-space}, \ref{assump-ergodicity}, \ref{assump-task2-omega-true}, and \ref{assump-boundedness} hold. 
  The parameters $\lambda = \lambda(T)$, $\epsilon=\epsilon(T)$ in the
  minimization problem \eqref{opt-problem-1} satisfy 
  \begin{align}
    \lim_{T\rightarrow +\infty} \lambda(T) = 0\,, \quad \lim_{T\rightarrow
    +\infty} \epsilon(T) = 0\,.
    \label{lambda-t-limit-is-zero}
  \end{align}
Then we have $\lim\limits_{T\rightarrow +\infty} \bm{\omega}^{(T, \epsilon, \lambda)} = \bm{\omega}^{*}$, a.s. 
    \label{thm-omega-limit-general-lambda}
\end{theorem}
\begin{theorem}
  Suppose that Assumptions~\ref{assump-finite-space}, \ref{assump-ergodicity}, \ref{assump-task2-omega-true}, and \ref{assump-boundedness} hold. 
  Let $\mathcal{F}$ be the $N\times N$ matrix whose entries are given in
  \eqref{covariance-c-varphi} and let $\bm{\omega}^{(T,\epsilon,\lambda)}$ be the
  minimizer of the problem \eqref{opt-problem-1}. 
  Further assume that the following conditions are met.
  \begin{enumerate}[wide]
\item
  The parameters $\lambda = \lambda(T)$, $\epsilon=\epsilon(T)$ in \eqref{opt-problem-1} satisfy 
  \begin{align}
    \lim_{T\rightarrow +\infty}  \sqrt{T}\lambda(T) = 0\,, \quad \epsilon(T) =
    \mathcal{O}(T^{-\alpha})\,, \quad \mbox{as}~T\rightarrow +\infty\,,
    \label{sqrt-t-times-lambda-t-limit-is-zero}
  \end{align}
      for some $\alpha > 0$.
    \item
      There exists $c>0$, such that for all $x\in \mathbb{X}$ and $1 \le i \le
      K$ satisfying $a^*_i(x) = 0$, we have either $\varphi_j(x) = 0$ for all
      $j \in \mathcal{I}_i$, or $\sum_{j \in
      \mathcal{I}_i}\omega_{j}^{*}\varphi_j(x)\le -c < 0$.
  \end{enumerate}
  Then, as $T \rightarrow +\infty$, $\sqrt{T}
  \big(\bm{\omega}^{(T, \epsilon, \lambda)} - \bm{\omega}^{*}\big)$ converges in
  distribution to a Gaussian random variable with mean zero and covariance
  matrix $\mathcal{F}^{-1}$. 
  \label{thm-asymptotic-normality-lambda}
\end{theorem}
\appendix
\section{Pseudocode of FISTA with backtracking} \label{app-0}
We summarize the main algo\-rithmic steps of FISTA with backtracking~\cite{fista2009} for the optimization problem  
\begin{align}
  \min_{\bm{x} \in \mathbb{R}^N} F(\bm{x}) = \min_{\bm{x} \in \mathbb{R}^N}
  \Big(f(\bm{x}) + \lambda \sum_{j=1}^{N} \frac{|x_j|}{c_j}\Big)
  \label{app-0-opt-problem}
\end{align}
in Algorithm~\ref{fista-algo}, where $\lambda>0$ and $c_j > 0$. The optimization problems
\eqref{opt-problem-1-sub} and \eqref{opt-problem-1-rescaled} are in the form of \eqref{app-0-opt-problem}, with
$f$ being the (negative) logarithmic likelihood function. We refer the readers
to the original paper~\cite{fista2009}, where FISTA is developed for
optimization problems which are more general than \eqref{app-0-opt-problem}.
\begin{algorithm}[H]
  \caption{FISTA with backtracking for $\min_{\bm{x}}F(\bm{x}) =
  \min_{\bm{x}}\big(f(\bm{x}) + \lambda \sum_{j} \frac{|x_j|}{c_j}\big)$\label{fista-algo}}
  \begin{algorithmic}[1]
    \Function {$\mathcal{T}_\alpha$}{$x$} \Comment{shrinkage operator}
       \State
	  \textbf{return} $\max(|x|-\alpha, 0) \cdot \sgn(x)$
       \EndFunction
       \Statex
       \Function {$Q_L$}{$\bm{x}, \bm{y}$} \Comment{quadratic approximation of $F(\bm{x})$ at $\bm{y}$}
       \State
\textbf{return} $f(\bm{y}) + \langle \bm{x} - \bm{y}, \nabla f(\bm{y})\rangle + \frac{L}{2}\|\bm{x} -
       \bm{y}\|^2 + \lambda \sum_{j=1}^N \frac{|x_j|}{c_j}$
       \EndFunction
\Statex
    \Function{$p_L$}{$\bm{y}$} \Comment{$p_L(\bm{y}) = \argmin\limits_{\bm{x}\in \mathbb{R}^N}\big\{Q_L(\bm{x}, \bm{y})\big\}$}
     \For{$j\gets 1$ to $N$}
       \State $z_j =
       \mathcal{T}_{\lambda/(Lc_j)}\big(y_j -\frac{1}{L}\frac{\partial f}{\partial x_j}(\bm{y})\big)$
       \EndFor
       \State $\bm{z} = (z_1, z_2, \dots, z_N)^\top$
       \State \textbf{return} $\bm{z}$
    \EndFunction
\Statex
    \Procedure{FISTA}{}
    \State Choose $L_0 > 0, \eta > 1$, and $\bm{x}_{0}\in \mathbb{R}^N$. Set $\bm{y}_1 = \bm{x}_0$, $t_1 = 1$, $k=0$.
    \While{not converged}  
    \State $k\gets k + 1$.
    \State Find the smallest nonnegative integer $i_k$, such that with
    $\bar{L} = \eta^{i_k}L_{k-1}$,
	$$F(p_{\bar{L}}(\bm{y}_k)) \le Q_{\bar{L}}(p_{\bar{L}}(\bm{y}_k), \bm{y}_k)\,.$$
	\State Set $L_k = \eta^{i_k}L_{k-1}$, and compute
	\begin{align*}
	  \bm{x}_k = p_{L_k}(\bm{y_k}),\,~  t_{k+1} = \frac{1}{2}\Big(1 +
	  \sqrt{1 + 4t_k^2}\Big),\,~ 
	  \bm{y}_{k+1} =  \bm{x}_k + \Big(\frac{t_k -
	  1}{t_{k+1}}\Big)(\bm{x}_k - \bm{x}_{k-1})\,.
	\end{align*}
      \EndWhile
    \EndProcedure
  \end{algorithmic}
\end{algorithm}

\section{Properties of the function $G_\epsilon$}
\label{app-1}
We now summarize some asymptotic properties of the function $G_\epsilon$ in \eqref{g-eps}.
Given $\epsilon>0$, recall that
  $G_\epsilon(x) = \epsilon\ln \big(1 + e^{x/\epsilon}\big)$, for all $ x \in
  \mathbb{R}$\,, whose first and second derivatives are
  \begingroup
 \small
\setlength\abovedisplayskip{6pt}
\setlength\belowdisplayskip{6pt}
\begin{align}
  G_\epsilon'(x) = \frac{e^{x/\epsilon}}{1+e^{x/\epsilon}}\,,\quad
  G_\epsilon''(x) = \frac{1}{\epsilon}
  \frac{e^{x/\epsilon}}{(1+e^{x/\epsilon})^2}\,,
  \label{g-1st-2rd}
\end{align}
\endgroup
respectively. The following lemma can be easily proved and therefore its proof is omitted.

\begin{lemma}
  Given $\epsilon>0$, we have the following estimates.
  \begin{enumerate}[wide]
    \item $\max(x,0) < G_\epsilon(x) \le \max(x,0) + \epsilon \ln 2$, $\quad\forall\,x \in \mathbb{R}$.
    \item $1 - e^{-x/\epsilon}< G_\epsilon'(x) < 1$, if $x \ge 0$, and ~$0<G_\epsilon'(x) < e^{x/\epsilon}$, if $x < 0$.
\item $0 < G''_\epsilon(x) < \frac{1}{\epsilon} e^{-|x|/\epsilon}\,, \quad \forall~x \in \mathbb{R}\,.$
  \end{enumerate}
  \label{lemma-g-eps}
\end{lemma}

In particular, Lemma~\ref{lemma-g-eps} implies 
  $\lim\limits_{\epsilon \rightarrow 0+} G_\epsilon(x) = \max(x,0)$,
uniformly for $x \in \mathbb{R}$, and 
\begingroup
\small
\setlength\abovedisplayskip{6pt}
\setlength\belowdisplayskip{6pt}
\begin{align*}
  \lim_{\epsilon \rightarrow 0+}
  G_\epsilon'(x) = 
	\begin{cases}
	  1 \,, & x > 0\\
	  \frac{1}{2}\,, & x=0\\
	  0 \,, & x < 0
	\end{cases}
	\,,
	\qquad
  \lim_{\epsilon \rightarrow 0+}
  G_\epsilon''(x) = 
	\begin{cases}
	  0 \,, & x \neq 0\\
	  +\infty \,, & x = 0\,.
	\end{cases}
\end{align*}
\endgroup

We also need to study the function $\ln G_\epsilon(x) = \ln \big[\epsilon \ln(1+e^{x/\epsilon})\big]$,
 whose first and second derivatives are 
 \begingroup
 \small
\setlength\abovedisplayskip{6pt}
\setlength\belowdisplayskip{6pt}
\begin{align}
  \begin{split}
  (\ln G_\epsilon)'(x) &= \frac{e^{x/\epsilon}}{\epsilon (1 +
  e^{x/\epsilon})\ln(1 + e^{x/\epsilon})}\,,\\
  (\ln G_\epsilon)''(x) &= \frac{1}{\epsilon^2}
  \frac{e^{x/\epsilon}}{(1+e^{x/\epsilon})^2}\frac{1}{\ln(1+e^{x/\epsilon})}
  \bigg[1 - \frac{e^{x/\epsilon}}{\ln(1+e^{x/\epsilon})}\bigg]\,.
  \end{split}
  \label{log-g-eps-derivatives}
\end{align}
\endgroup
\begin{lemma}
  Given $\epsilon>0$, we have the following estimates.
  \begin{enumerate}[wide]
    \item
      For all $x>0$, it holds that
      \begingroup
      \small
      \begin{align}
	\begin{split}
	  & \ln x < \ln G_\epsilon(x) < \ln x + \frac{\epsilon}{x}
	  e^{-x/\epsilon},\quad \frac{1}{(1+e^{-x/\epsilon})(x + \epsilon\ln 2)} < (\ln G_\epsilon)'(x) < \frac{1}{x},\\
	  \mbox{and}\quad&	 -\frac{1}{x^2}
	  < (\ln G_\epsilon)''(x) < \frac{e^{-x/\epsilon}}{\epsilon x}
      -\frac{1}{(1+e^{-x/\epsilon})^2(x + \epsilon\ln 2)^2}\,.
    \end{split}
	\label{lemma-for-log-geps-eqn1}
      \end{align}
      \endgroup
    \item
      For $x = 0$, we have
      \begingroup
      \small
      \begin{align*}
	\ln G_\epsilon(0) = \ln \epsilon + \ln\ln 2\,, \quad  
	  (\ln G_\epsilon)'(0) = \frac{1}{(2\ln 2)\epsilon}\,, \quad
	   (\ln G_\epsilon)''(0) = 
	  \frac{1}{4\ln 2}\Big(1-\frac{1}{\ln 2}\Big) \frac{1}{\epsilon^2} \,.
      \end{align*}
      \endgroup
    \item
      For all $x<0$, we have
      \begingroup
      \small
      \begin{align*}
	  \ln G_\epsilon(x) < \ln \epsilon + \frac{x}{\epsilon}\,,\quad (\ln G_\epsilon)'(x) > \frac{1}{2\epsilon}\,, 
	  \quad
-\frac{2}{\epsilon^2} e^{x/\epsilon} < (\ln G_\epsilon)''(x) < 0\,.
      \end{align*}
      \endgroup
  \end{enumerate}
  \label{lemma-for-log-geps}
\end{lemma}
\begin{proof}
  We will only prove the inequalities concerning $(\ln G_\epsilon)''$. 
  \begin{enumerate}[wide]
    \item
  When $x>0$, using \eqref{log-g-eps-derivatives} and the fact
  $\epsilon\ln(1+e^{x/\epsilon})>x$, we have 
	   $(\ln G_\epsilon)''(x) > -\frac{1}{x^2}$.
  For the upper bound, using Lemma~\ref{lemma-g-eps}, we have
      \begingroup
      \small
  \begin{align*}
    & \frac{1}{\epsilon^2}
    \frac{e^{x/\epsilon}}{(1+e^{x/\epsilon})^2}\frac{1}{\ln(1+e^{x/\epsilon})}
    =
\frac{1}{\epsilon^2}
    \frac{e^{-x/\epsilon}}{(1+e^{-x/\epsilon})^2}\frac{1}{\ln(1+e^{x/\epsilon})}
    < \frac{e^{-x/\epsilon}}{\epsilon x}\,,\\
    &
    -\frac{e^{2x/\epsilon}}{(1+e^{x/\epsilon})^2}\frac{1}{\big[\epsilon\ln(1+e^{x/\epsilon})\big]^2}
    \le 
    -\frac{1}{(1+e^{-x/\epsilon})^2}\frac{1}{(x + \epsilon \ln 2)^2}\,,
  \end{align*}
      \endgroup
  and therefore \eqref{lemma-for-log-geps-eqn1} is obtained.
    \item
  When $x<0$, using the fact that $-\frac{u^2}{2} < \ln(1+u) - u < 0$, for all
  $ u>0$, we have 
      $\ln(1+e^{x/\epsilon}) > e^{x/\epsilon} - \frac{1}{2}e^{2x/\epsilon} > \frac{1}{2} e^{x/\epsilon}.$
  Therefore,
      \begingroup
      \small
  \begin{align*}
  (\ln G_\epsilon)''(x) &= 
    \frac{e^{x/\epsilon}}{\epsilon^2 (1+e^{x/\epsilon})^2}\frac{1}{\big(\ln(1+e^{x/\epsilon})\big)^2}
  \Big(\ln(1+e^{x/\epsilon}) - e^{x/\epsilon}\Big) \\
    &>
    -\frac{e^{3x/\epsilon}}{2\epsilon^2(1+e^{x/\epsilon})^2}\frac{1}{\frac{1}{4}e^{2x/\epsilon}}
    > -\frac{2}{\epsilon^2} e^{x/\epsilon}\,.
  \end{align*}
      \endgroup
  \end{enumerate}
\end{proof}

Summarizing the estimates in Lemma~\ref{lemma-for-log-geps}, we can conclude that
\begingroup
\small
\begin{align*}
  \lim_{\epsilon \rightarrow 0+} \ln G_\epsilon(x) = &
  \begin{cases}
    \ln x\,, & x > 0 \\
    -\infty\,, & x \le 0
  \end{cases}
  \,,
  \quad 
  \lim_{\epsilon \rightarrow 0+} (\ln G_\epsilon)'(x) = 
  \begin{cases}
    \frac{1}{x}\,, & x > 0 \\
    +\infty\,, & x \le 0
  \end{cases}
  \,,
   \\
  \text{and}\quad
  \lim_{\epsilon \rightarrow 0+} (\ln G_\epsilon)''(x) = &
  \begin{cases}
    -\frac{1}{x^2}\,, & x > 0 \\
    -\infty\,, & x = 0 \\
    0\,, & x < 0
  \end{cases}
  \,.
\end{align*}
\endgroup

\section{Two limit lemmas on integrations with respect to counting processes}
\label{app-2}
In this section, we summarize two useful results pertaining to integrations with respect to the
processes $R_i$, $\widetilde{R}_i$ in \eqref{r0-ri} and
\eqref{compensate-r0-ri}, respectively. These results play an important
role in the asymptotic analysis in Section~\ref{sec-asymptotics}.
The first result is a type of law of large numbers (LLN) for Poisson processes.
\begin{lemma}
  Suppose that Assumptions~\ref{assump-finite-space}-\ref{assump-ergodicity} hold. 
  Functions $f^{(T)} : \mathbb{X} \rightarrow \mathbb{R}$ 
  satisfy $\lim\limits_{T\rightarrow +\infty} f^{(T)}(x) = f(x)$, $\forall~x\in \mathbb{X}$. For each $1 \le i \le K$, we have
  \begingroup
    \small
\setlength\abovedisplayskip{6pt}
\setlength\belowdisplayskip{6pt}
  \begin{align}
    &\lim_{T\rightarrow +\infty} \frac{1}{T} \int_0^T
    f^{(T)}\big(X(s)\big)\,dR_i(s) = \sum_{x \in \mathbb{X}}
    f(x)\,a_i^*(x) \pi(x) \,, \quad a.s.
    \label{lemma-int-ri-claim-1} \\
     \mbox{and} \quad&
    \lim_{T\rightarrow +\infty} \frac{1}{T} \int_0^T
    f^{(T)}\big(X(s)\big)\,d\widetilde{R}_i(s) = 0\,, \quad a.s.
    \label{lemma-int-ri-claim-2}
  \end{align}
  \endgroup
  \label{lemma-integral-ri-limit}
\end{lemma}
\begin{proof}
  Since $\mathbb{X}$ is a finite set (Assumption~\ref{assump-finite-space}), the convergence of $f^{(T)}$ to $f$ is
  in fact uniform on $\mathbb{X}$.
  Using the LLN of Poisson processes in \eqref{lln-poisson} and the uniform
  convergence of $f^{(T)}$, we have 
  \begingroup
  \small
  \begin{align*}
    &\Big|\lim_{T\rightarrow +\infty} \frac{1}{T} \int_0^T
    f^{(T)}\big(X(s)\big)\,dR_i(s) 
    - \lim_{T\rightarrow +\infty} \frac{1}{T} \int_0^T f\big(X(s)\big)\,dR_i(s)
    \Big|\\
     \le & \lim_{T\rightarrow +\infty} \bigg[\frac{R_i(T)}{T} \sup_{x \in
    \mathbb{X}} \big|f^{(T)} - f\big|\bigg] = 0 \,.
  \end{align*}
  \endgroup
  Therefore, it is sufficient to prove \eqref{lemma-int-ri-claim-1} for the
  case $f^{(T)}\equiv f$. 
  Note that we have
  \begingroup
  \small
  \begin{align}
    \frac{1}{T} \int_0^T f(X(s))\, dR_i(s)  = \sum_{x \in \mathbb{X}} 
    \bigg[
f(x)
    \frac{1}{T} \int_0^T \mathbf{1}_{x}(X(s))\,dR_i(s)\bigg]\,,
  \label{lemma-int-ri-eqn1}
  \end{align}
  \endgroup
  where $\mathbf{1}_{x}$ denotes the indicator function at state $x$. 
  For each $x \in \mathbb{X}$, $\int_0^T
  \mathbf{1}_x\big(X(s)\big)\,dR_i(s)$ can be interpreted as the total number of times that the $i$th channel $\mathcal{C}_i$ becomes active within time $[0, T]$
  when the state of the system is $x$. 
Similarly, $\int_0^T \mathbf{1}_x\big(X(s)\big)\,ds$ is the total time that the system
  spends at state $x$ within time $[0, T]$. Since the waiting times
  at state $x$ before the channel $\mathcal{C}_i$ becomes activated are independent and follow exponential distributions with mean
  value $\big(a_i^*(x)\big)^{-1}$, the LLN of exponential distributions implies that
  \begingroup
  \small
  \begin{align}
    \lim_{T\rightarrow +\infty}  \frac{\int_0^T
    \mathbf{1}_x\big(X(s)\big)\,ds}{\int_0^T \mathbf{1}_x\big(X(s)\big)\,dR_i(s)} =
    \frac{1}{a_i^*(x)}\,, \quad a.s.
  \label{lemma-int-ri-eqn2}
  \end{align}
  \endgroup
  Since the system is ergodic (Assumption~\ref{assump-ergodicity}), Birkhoff's ergodic theorem implies 
  \begingroup
  \small
  \begin{align}
    \lim_{T\rightarrow +\infty}  \frac{1}{T}\int_0^T \mathbf{1}_x\big(X(s)\big)\,ds = \pi(x)\,, \quad a.s.
  \label{lemma-int-ri-eqn3}
  \end{align}
  \endgroup
  Combining \eqref{lemma-int-ri-eqn1}--\eqref{lemma-int-ri-eqn3}, we obtain
$\lim\limits_{T\rightarrow +\infty} \frac{1}{T} \int_0^T f\big(X(s)\big)\, dR_i(s)  =
    \sum\limits_{x \in \mathbb{X}} f(x)
    a_i^*(x)\,\pi(x)$\,, a.s.
The conclusion \eqref{lemma-int-ri-claim-2} follows as a consequence, using the definition of $\widetilde{R}_i$ in
\eqref{compensate-r0-ri} and the ergodicity of the system. 
\end{proof}
The second result is a corollary of the martingale central limit theorem~\cite[Theorem~7.1.4]{ethier1986markov}.
\begin{lemma}
  Suppose that Assumptions~\ref{assump-finite-space}-\ref{assump-ergodicity} hold.  For each $1 \le j \le N$, 
  functions $f_{j}, f^{(T)}_{j} \colon \mathbb{X} \rightarrow \mathbb{R}$ satisfy
   $\lim\limits_{T\rightarrow +\infty} f^{(T)}_{j}(x) = f_{j}(x)$, $\forall~x \in \mathbb{X}$. Let 
  $\mathcal{W}^{(T)}(u) \in \mathbb{R}^N$ denote the $N$-dimensional process whose components are 
    $\mathcal{W}^{(T)}_{j}(u) =  \frac{1}{\sqrt{T}} \int_0^{Tu} 
    f^{(T)}_{j}\big(X(s)\big)\, d\widetilde{R}_i(s)$\,, where $u \ge 0$, $1 \le j \le N$, and the index $i$ satisfies $j \in \mathcal{I}_i$, $1 \le i \le K$. Moreover,
  $\mathcal{F}$ is the $N\times N$ matrix whose entries are given by 
  \begingroup
  \small
  \begin{align}
    \mathcal{F}_{j,j'} = 
      \begin{cases}
	~\sum\limits_{x \in \mathbb{X}} f_{j}(x) f_{j'}(x)
	a_i^*(x)\, \pi(x)\,, &  \mbox{if}~~j, j' \in \mathcal{I}_i
	\,,~\mbox{for some index}~~1 \le i \le K\\
	~0\,, &   \mbox{otherwise}\,, \\
      \end{cases}
      \label{covariance-c-u1}
  \end{align}
  \endgroup
  for $1 \le j, j' \le N$. We define the matrix-valued (linear) process $\mathcal{A}(u)=u\,\mathcal{F}$, $u \ge 0$.

As~$T\rightarrow \infty$, $\mathcal{W}^{(T)}$ converges in distribution to $\mathcal{W}$, where
$\mathcal{W}$ is an $N$-dimensional process with independent Gaussian increments whose quadratic variation process is $\mathcal{A}$. In particular, 
$\mathcal{W}^{(T)}(1)$ converges in distribution to a Gaussian random variable
  whose mean is zero and whose covariance matrix is $\mathcal{F}$ in \eqref{covariance-c-u1}.
  \label{lemma-integral-ri-ctl}
\end{lemma}
\begin{proof}
  For each $T>0$, we define the matrix-valued process $\mathcal{A}^{(T)}(u)$, $u\ge 0$, whose entries are given by 
  \begingroup
  \small
\begin{align}
  \mathcal{A}_{j,j'}^{(T)}(u) = 
  \begin{cases} 
  ~\frac{1}{T} \int_0^{Tu} 
    f^{(T)}_{j}\big(X(s)\big)
    f^{(T)}_{j'}\big(X(s)\big)\,a_i^*(X(s))\, ds\,, &
    \mbox{if}~~j, j' \in \mathcal{I}_i \,,~\mbox{for some}~1 \le i \le K\\[5pt]
    ~0\,, & \mbox{otherwise}\,,
  \end{cases}
\label{a-i-j-jprime}
\end{align}
\endgroup
  for $1 \le j,\,j' \le N$. 
Let us verify the conditions required by the martingale central limit theorem~\cite[Theorem~7.1.4]{ethier1986markov}. 

Firstly, from \eqref{r0-ri} and \eqref{compensate-r0-ri}, applying Ito's formula, we know the process 
\begingroup
  \small
  $$\mathcal{W}_{j}^{(T)}(u) \mathcal{W}_{j'}^{(T)}(u) -
  \mathcal{A}_{j,j'}^{(T)}(u)\,,\quad u \ge 0,$$
  \endgroup
is a martingale. Using the expression \eqref{a-i-j-jprime} and the ergodicity of the system, we have
\begingroup
  \small
\begin{align*}
  \lim_{T\rightarrow +\infty} 
  \mathcal{A}_{j,j'}^{(T)}(u) = u\, \mathcal{F}_{j,j'} =
  \mathcal{A}_{j,j'}(u)\,,\quad a.s.
\end{align*}
\endgroup
  Furthermore, since $f^{(T)}_{j}$ converge to $f_j$ and
  $\mathbb{X}$ is a finite set (Assumption~\ref{assump-finite-space}), it is clear that $f^{(T)}_j$ are uniformly bounded. 
  This implies that 
  \begingroup
\setlength\abovedisplayskip{6pt}
\setlength\belowdisplayskip{6pt}
  \small
\begin{align*}
  & \lim_{T\rightarrow +\infty} \mathbf{E}\bigg[\sup_{u \le u_0} \Big|\mathcal{W}_{j}^{(T)}(u)
  - \mathcal{W}_{j}^{(T)}(u-)\Big|^2\bigg]
  \le \lim_{T\rightarrow
  +\infty}\frac{1}{T} \mathbf{E}\bigg[\sup_{u \le Tu_0}
  \Big|f_{j}^{(T)}(X(u))\Big|^2 \bigg] =0\,,\quad\forall~u_0 \ge 0.
\end{align*}
\endgroup

  Secondly, because the processes $\mathcal{A}_{j,j'}^{(T)}(u)$ in
  \eqref{a-i-j-jprime} have continuous paths, the limit
  \begingroup
  \small
\setlength\abovedisplayskip{6pt}
\setlength\belowdisplayskip{6pt}
\begin{align*}
\lim_{T\rightarrow +\infty} \mathbf{E}\bigg[\sup_{u \le u_0} 
  \Big|\mathcal{A}_{j,j'}^{(T)}(u) - 
\mathcal{A}_{j,j'}^{(T)}(u-)\Big|\bigg] = 0\,,\qquad 1 \le j, j' \le N
\end{align*}
\endgroup
holds trivially. Therefore,  we can apply the martingale central limit theorem~\cite[Theorem~7.1.4]{ethier1986markov} and 
the conclusion follows readily. 
\end{proof}

\section{Proofs of results in Section~\ref{sec-asymptotics}}
\label{app-3}
In this section, we prove the results presented in Section~\ref{sec-asymptotics}.

We start with the results in Subsection~\ref{subsec-infinite-many-data}.
\begin{proof} [Proof of Proposition~\ref{prop-for-ni-is-1}]
  As already pointed out in Subsection~\ref{subsec-learn-rate},  the Euler--Lagrange equation
  \eqref{opt-problem-euler-lagrange} can be explicitly solved when $N_i=1$ and the solution
  is given in \eqref{opt-omega-ni1}.  Using the representations in
  \eqref{r0-ri} and \eqref{m-mi-t}, we can rewrite \eqref{opt-omega-ni1} as 
  \begingroup
  \small
\setlength\abovedisplayskip{6pt}
\setlength\belowdisplayskip{6pt}
\begin{align*}
  \omega_{j}^{(T)} = \frac{M_i}{\sum\limits_{l=0}^{M} t_l \, \varphi_{j}(y_l)}
  = \frac{\mathcal{P}_i\Big(\omega_{j}^{*}\int_0^T \varphi_{j}\big(X(s)\big)\,ds\Big)}{
    \int_0^T \varphi_{j}\big(X(s)\big)\,ds}\,.
\end{align*}
 \endgroup 
Applying \eqref{lln-poisson} together with \eqref{ni1-divergence-assump},
  we conclude that $\lim\limits_{T\rightarrow +\infty} \omega^{(T)}_{j} =
  \omega^{*}_{j}$, almost surely. 
\end{proof}
    \begin{proof}[Proof of Lemma~\ref{lemma-uniqueness-omega-t}]
From \eqref{opt-problem-y-t-f} and \eqref{log-likelihood-ith-omega}, it is not difficult to see that, with probability one, there is at least one minimizer for large enough $T$. 
We show the uniqueness by contradiction. Suppose that, with positive probability, the solution of \eqref{opt-problem-y-t}--\eqref{opt-problem-y-t-f}
      is not unique for an increasing subsequence $T_k$, where
      $\lim\limits_{k\rightarrow +\infty}T_k=+\infty$. According to
      Proposition~\ref{prop-uniqueness-conditions}, we can find an index $i$, $1 \le i
\le K$, such that the column vectors $\Phi_{i,l}$, $1\le l\le N_i$, of the matrix $\Phi_i$ in \eqref{mat-phi-i} are linearly
dependent for $T_k$, where $k=1,2,\cdots$. Let us order the states in $\mathbb{X}$ such that
      $\mathbb{X} = \{x_1, x_2, \dots, x_m\}$, where $m =|\mathbb{X}|$ is positive.
     The ergodicity of the system (Assumption~\ref{assump-ergodicity}) implies
      that with probability one the states $x_1, x_2, \dots, x_m$ will
    be visited by the system within some large finite time.
      Since there is a positive probability that the column vectors $\Phi_{i,l}$
      are linearly dependent for all $T_k$ where $\lim\limits_{k\rightarrow +\infty} T_k=+\infty$, 
      we can find a nonzero vector $\bm{\eta} \in (\eta_1, \eta_2, \dots, \eta_{N_i})^\top \in \mathbb{R}^{N_i}$, such that 
      $\sum\limits_{k=1}^{N_i} \eta_k \varphi_{j_k}(x_l) = 0$\,, $\forall\,1 \le l \le m$, where $\mathcal{I}_i = \big\{j_1, j_2, \dots, j_{N_i}\big\}$.
 This contradicts Assumption~\ref{assump-task1-uniqueness}.
    \end{proof}
\begin{proof}[Proof of Proposition~\ref{prop-omega-limit-general}]
      \begin{enumerate}[wide]
  \item
    Under Assumption~\ref{assump-ergodicity}, using expressions \eqref{exp-log-likelihood-r}, \eqref{exp-derivatives-special},
	  and applying Lemma~\ref{lemma-integral-ri-limit} in
	  Appendix~\ref{app-2}, we can compute 
	  \begingroup
	  \small
    \begin{align} \label{prop-omega-limit-general-proof-eqn1}
	&\lim_{T\rightarrow +\infty} \frac{\ln\mathcal{L}^{(T)}(\bm{\omega}) -
	\ln\mathcal{L}^{(T)}\big(\bm{\omega}^{*}\big)}{T} \nonumber \\
	=& \lim_{T\rightarrow +\infty}\bigg[\frac{1}{T} \sum_{i=1}^K\int_0^T
	\ln
	\frac{a_i\big(X(s)\,;\,\bm{\omega}\big)}{a_i\big(X(s)\,;\,\bm{\omega}^{*}\big)}\,d\widetilde{R}_i + 
\frac{1}{T} \sum_{i=1}^K\int_0^T \bigg(a_i\big(X(s)\,;\,\bm{\omega}^{*}\big) -
	a_i\big(X(s)\,;\,\bm{\omega}\big)\bigg)\,ds \nonumber \\
	& + \sum_{i=1}^K\int_0^T a_i\big(X(s)\,;\,\bm{\omega}^{*}\big)
	\ln
	\frac{a_i\big(X(s)\,;\,\bm{\omega}\big)}{a_i\big(X(s)\,;\,\bm{\omega}^{*}\big)}
	 \,ds\bigg] \nonumber \\
	=& \sum_{x\in\mathbb{X}} \sum_{i=1}^K\bigg[
	    a_i\big(x\,;\,\bm{\omega}^{*}\big) - a_i(x\,;\,\bm{\omega}) +
	    a_i\big(x\,;\,\bm{\omega}^{*}\big)
	\ln\frac{a_i(x\,;\,\bm{\omega})}{a_i\big(x\,;\,\bm{\omega}^{*}\big)}
	\bigg]\, \pi(x)\, \\
	    =& \sum_{x\in\mathbb{X}} \bigg[
	    a(x\,;\,\bm{\omega}^{*}) - a(x\,;\,\bm{\omega}) + 
	    \sum_{i=1}^K
	    \bigg(a_i\big(x\,;\,\bm{\omega}^{*}\big)
	\ln
	\frac{a_i(x\,;\,\bm{\omega})}{a_i\big(x\,;\,\bm{\omega}^{*}\big)}
	\bigg)\bigg]\, \pi(x)\, \nonumber \\
	    =& -\sum_{x\in\mathbb{X}}
	\bigg[D_{KL}\Big(\psi\big(\cdot\,;\,x,\bm{\omega}^{*}\big)\,\Big|\,
	\psi\big(\cdot\,;\,x,\bm{\omega}\big)\Big) +
	D_{KL}\Big(p\big(\cdot\,;\,x,\bm{\omega}^{*}\big)\,\Big|\,
	p\big(\cdot\,;\,x,\bm{\omega}\big)\Big)\bigg]
	a\big(x\,;\,\bm{\omega}^{*}\big)\,\pi(x) \,, \nonumber
    \end{align}
	 \endgroup 
    where we have used \eqref{d-kl-phi-p} in the last equality.
    Therefore, the first conclusion is obtained.
  \item
  Firstly, let us show that the sequence $\big(\omega^{(T)}_{j}\big)_{T>0}$ is almost surely bounded for each $1 \le j \le N$. 
  From the Euler--Lagrange equation \eqref{opt-problem-euler-lagrange}, we can
  obtain the relation
	  \begingroup
	  \small
\setlength\abovedisplayskip{6pt}
\setlength\belowdisplayskip{6pt}
    \begin{align*}
      \sum_{j\in \mathcal{I}_i} \omega^{(T)}_{j}\bigg( \sum_{l=0}^{M}
      t_l\,\varphi_{j}(y_l)\bigg) = M_i\,, \qquad \forall~1 \le i \le K\,,
    \end{align*}
	\endgroup 
    which implies 
	  \begingroup 
	  \small
\setlength\abovedisplayskip{6pt}
\setlength\belowdisplayskip{6pt}
    \begin{align}
      0 \le \omega^{(T)}_{j} \le& \frac{M_i}{\sum\limits_{l=0}^{M} t_l\, \varphi_{j}(y_l)}
      = \frac{\frac{1}{T} \int_0^T 1\,dR_i(s)}{\frac{1}{T}\int_0^T
      \varphi_{j}\big(X(s)\big)\,ds}\,,
      \label{omega-upper-bound-ineq}
    \end{align}
	 \endgroup 
	  where $i$, $1 \le i \le K$, is the index such that $j \in \mathcal{I}_i$.
    Note that both the numerator and the denominator on the right-hand side of
    \eqref{omega-upper-bound-ineq} converge, as consequences of
    Lemma~\ref{lemma-integral-ri-limit} in Appendix~\ref{app-2} and the ergodicity of the system
	  (Assumption~\ref{assump-ergodicity}), respectively. Taking the limit $T \rightarrow +\infty$ in \eqref{omega-upper-bound-ineq} and using
\eqref{ai-omega}, we have 
\begingroup
\small
\setlength\abovedisplayskip{6pt}
\setlength\belowdisplayskip{6pt}
    \begin{align*}
      \limsup_{T\rightarrow +\infty}\, \omega^{(T)}_{j} \le
      \lim_{T\rightarrow +\infty}\frac{\frac{1}{T}
\int_0^T 1\,dR_i(s)}{\frac{1}{T} \int_0^T \varphi_{j}\big(X(s)\big)\,ds}
= \frac{\sum\limits_{x\in\mathbb{X}}\bigg[\sum\limits_{j'\in
\mathcal{I}_i}\omega^{*}_{j'}\,\varphi_{j'}(x)\bigg]\,\pi(x)}{\sum\limits_{x\in\mathbb{X}}
      \varphi_{j}(x)\,\pi(x)}\,, \quad a.s.,
    \end{align*}
    \endgroup
    which implies that the sequence $\big(\omega_{j}^{(T)}\big)_{T>0}$ is almost surely bounded.

    Secondly, from \eqref{exp-1st-derivative-r} we know that the minimizer $\bm{\omega}^{(T)}$ satisfies the identity
    \begingroup
    \small
    \begin{align*}
      \int_0^T \frac{\varphi_{j}\big(X(s)\big)}{
	a_i\big(X(s)\,;\,\bm{\omega}^{(T)}\big)}\,dR_i(s) = \int_0^T
      \varphi_{j}\big(X(s)\big)\,ds \,, \quad j \in \mathcal{I}_i\,.
    \end{align*}
    \endgroup
    In particular, for each state $x \in \mathbb{X}$, it implies
    \begingroup
    \small
    \begin{align*}
      \frac{\varphi_{j}(x)}{a_i\big(x\,;\,\bm{\omega}^{(T)}\big)}
      \int_0^T \mathbf{1}_x\big(X(s)\big) \,dR_i(s) = 
      \int_0^T \frac{\varphi_{j}\big(X(s)\big)\,\mathbf{1}_x\big(X(s)\big)}{
	a_i\big(X(s)\,;\,\bm{\omega}^{(T)}\big)}
      \,dR_i(s) \le \int_0^T
      \varphi_{j}\big(X(s)\big)\,ds\,,
    \end{align*}
    \endgroup
  where $\mathbf{1}_{x}$ denotes the indicator function at $x$. 
    Therefore, applying Lemma~\ref{lemma-integral-ri-limit} in
    Appendix~\ref{app-2} and using the ergodicity of the system,
    we have
    \begingroup 
    \small
    \begin{equation} \label{lower-bound-for-ai}
      \liminf_{T\rightarrow +\infty} a_i\big(x\,;\,\bm{\omega}^{(T)}\big) \ge
      \varphi_{j}(x)\lim_{T\rightarrow +\infty}\frac{
	\int_0^T \mathbf{1}_x\big(X(s)\big) dR_i(s)}{\int_0^T \varphi_{j}(X(s))\,ds} 
       = \frac{\varphi_{j}(x) \,\pi(x)}{\sum\limits_{x' \in \mathbb{X}} \varphi_{j}(x')\,\pi(x')}\,a_i\big(x\,;\,\bm{\omega}^{*}\big)\,.
    \end{equation}
    \endgroup
    Note that whenever there is an increment for the counting process $R_i(s)$
    when $X(s)=x$, we know $a_i(x\,;\,\bm{\omega}^{*})>0$ and we can find an index
    $j\in \mathcal{I}_i$ such that the lower bound in \eqref{lower-bound-for-ai} is positive.
    
    Finally, let $\bar{\bm{\omega}}$ be a limit point of $\bm{\omega}^{(T)}$ as
    $T\rightarrow +\infty$.  Using a similar derivation as in
    \eqref{prop-omega-limit-general-proof-eqn1} and taking the lower bound \eqref{lower-bound-for-ai} into
    account, we obtain 
    \begingroup
    \small
	  \begin{align}
	    &  \liminf_{T\rightarrow +\infty} \frac{\ln\mathcal{L}^{(T)}(\bm{\omega}^{(T)}) -
	\ln\mathcal{L}^{(T)}(\bm{\omega}^{*})}{T} \nonumber \\
	    \le& -\sum_{x\in\mathbb{X}}
	    \bigg[D_{KL}\Big(\psi\big(\cdot\,;\,x,\bm{\omega}^{*}\big)\,\Big|\,
	\psi\big(\cdot\,;\,x,\bar{\bm{\omega}}\big)\Big) +
	    D_{KL}\Big(p\big(\cdot\,;\,x,\bm{\omega}^{*}\big)\,\Big|\,
	p\big(\cdot\,;\,x,\bar{\bm{\omega}}\big)\Big)\bigg]a\big(x\,;\,\bm{\omega}^{*}\big)\,
	\pi(x) \label{prop-omega-t-and-bar-limit-general-in-proof} \\
	    \le& ~ 0 \nonumber\,.
	  \end{align}
	  \endgroup 
	  On the other hand, since $\bm{\omega}^{(T)}$ is the minimizer of \eqref{opt-problem-y-t}, we also have 
	  \begingroup
	  \small
	  \begin{align*}
		        \liminf_{T\rightarrow +\infty} \frac{\ln\mathcal{L}^{(T)}(\bm{\omega}^{(T)}) -
	\ln\mathcal{L}^{(T)}(\bm{\omega}^{*})}{T}  \ge 0\,.
	  \end{align*}
	  \endgroup
	  Therefore, the Kullback--Leibler divergences in
	  \eqref{prop-omega-t-and-bar-limit-general-in-proof} must be equal to zero
	  at each state $x$. The expressions \eqref{d-kl-phi-p} then imply
	    $a_i\big(x\,;\,\bar{\bm{\omega}}\big) =
	    a_i\big(x\,;\,\bm{\omega}^{*}\big)$, $\forall\, 1 \le i \le K$ and
	    $\forall\,x \in \mathbb{X}$.
	  Using \eqref{ai-omega} and Assumption \ref{assump-task1-uniqueness}, we conclude 
	  $\bar{\bm{\omega}} = \bm{\omega}^{*}$ and therefore
$\lim\limits_{T\rightarrow +\infty} \bm{\omega}^{(T)} = \bm{\omega}^{*}$.
    \end{enumerate}
\end{proof}
\begin{proof}[Proof of Proposition~\ref{prop-asymptotic-normality}]
  First of all, under Assumption~\ref{assump-task1-uniqueness}, it is
  straightforward to verify that the matrix $\mathcal{F}$ is positive definite and
  therefore invertible.
  Given $1 \le j \le N$, expanding the function $\mathcal{M}^{(T)}_{j}(\bm{\omega})$
  in \eqref{opt-problem-euler-lagrange}, we have 
  \begingroup
  \small
\setlength\abovedisplayskip{6pt}
\setlength\belowdisplayskip{6pt}
\begin{align}
  \begin{split}
  &\mathcal{M}^{(T)}_{j}\big(\bm{\omega}^{(T)}\big) -  
  \mathcal{M}^{(T)}_{j}\big(\bm{\omega}^{*}\big)  
  = \sum_{j'=1}^{N} \bigg[\int_0^1
  \frac{\partial\mathcal{M}^{(T)}_{j}}{\partial
  \omega_{j'}}\Big(\theta\bm{\omega}^{(T)} + (1-\theta) 
\bm{\omega}^{*}\Big) d\theta\bigg] \Big(\omega_{j'}^{(T)} -
  \omega_{j'}^{*}\Big)\,.
  \end{split}
  \label{taylor-expansion-m-j}
\end{align}
\endgroup
Since $\mathcal{M}^{(T)}_{j}\big(\bm{\omega}^{(T)}\big) = 0$, dividing both sides of the equality above by $\sqrt{T}$, using
  \eqref{exp-2nd-derivative-r} and \eqref{exp-derivatives-special}, we have 
  \begingroup
  \small
\begin{align}
\frac{1}{\sqrt{T}} \int_0^{T} 
\frac{\varphi_{j}\big(X(s)\big)}{a_i\big(X(s)\,;\,\bm{\omega}^{*}\big)}
d\widetilde{R}_i(s)
  = \sum_{j'=1}^{N} 
 \mathcal{B}_{j,j'}^{(T)}\, \Big[\sqrt{T} \big(\omega_{j'}^{(T)} -
 \omega_{j'}^{*}\big)\Big]\,,
 \label{w-talor-eqn}
\end{align}
\endgroup
  where $i$, $1 \le i \le K$, is the index such that $j \in \mathcal{I}_i$, and we have introduced 
  \begingroup
  \small
\begin{align*}
  \begin{split}
     \mathcal{B}_{j,j'}^{(T)} &=
\frac{1}{T}\int_0^1
  \frac{\partial\mathcal{M}^{(T)}_{j}}{\partial
  \omega_{j'}}\Big(\theta \bm{\omega}^{(T)} + (1-\theta)
    \bm{\omega}^{*}\Big)\, d\theta \\
    &=
\frac{1}{T}
\int_0^T 
\bigg[\int_0^1
\frac{\varphi_{j}(X(s))\,\varphi_{j'}(X(s))}{
  \big[a_i\big(X(s)\,;\,\theta\bm{\omega}^{(T)} + (1-\theta) 
    \bm{\omega}^{*}\big)\big]^2}\, d\theta\bigg]\, dR_i(s)\,,
\end{split}
\end{align*}
\endgroup
  if $j,j' \in \mathcal{I}_i$, for some index $i$, $1 \le i \le K$, and $\mathcal{B}_{j,j'}^{(T)}=0$, otherwise.

  Let $\mathcal{B}^{(T)}$ denote the $N\times N$ matrix whose entries
  are $B^{(T)}_{j,j'}$, and let $\mathcal{W}^{(T)}$ denote the $N$-dimensional vector whose component 
    $\mathcal{W}^{(T)}_{j}$ equals the left-hand side of \eqref{w-talor-eqn}.
    With these notations, \eqref{w-talor-eqn} can be written as 
    \begingroup
    \small
    \begin{align}
      \mathcal{W}^{(T)} = \mathcal{B}^{(T)} \Big[\sqrt{T}
      \big(\bm{\omega}^{(T)} - \bm{\omega}^{*}\big)\Big]\,.
      \label{w-talor-eqn-vector}
    \end{align}
    \endgroup
  Applying Lemma~\ref{lemma-integral-ri-ctl} in Appendix~\ref{app-2}, we know
  that, as $T\rightarrow +\infty$, the vector $\mathcal{W}^{(T)}$
   converges in distribution to a Gaussian random variable whose mean equals zero and whose covariance matrix is given by
  $\mathcal{F}$. At the same time, since $\lim\limits_{T\rightarrow +\infty}
  \bm{\omega}^{(T)} = \bm{\omega}^{*}$ almost
  surely according to Proposition~\ref{prop-omega-limit-general},
  Lemma~\ref{lemma-integral-ri-limit} in Appendix~\ref{app-2} implies
  $\lim\limits_{T\rightarrow +\infty} \mathcal{B}^{(T)}  = \mathcal{F} \,,a.s.$
Therefore, applying Slutsky's Theorem~\cite{ferguson1996course}, we can
conclude 
\begin{align*}
  \sqrt{T} \big(\bm{\omega}^{(T)} - \bm{\omega}^{*}\big)
 = (\mathcal{B}^{(T)})^{-1} \mathcal{W}^{(T)}
  \Longrightarrow \mathcal{Z} \in
  \mathcal{N}\big(0, \mathcal{F}^{-1}\big)\,, \quad \mbox{as } T\rightarrow +\infty\,.
\end{align*}
\end{proof}

We continue to prove the results in Subsection~\ref{subsec-asymptotic-sparse-learning}.
\begin{proof}[Proof of Lemma~\ref{bounded-log-likelihood-imply-bounded-inf-sup}]
Lemma~\ref{lemma-g-eps} in Appendix~\ref{app-1} implies that 
\begingroup
  \small
$$G_\epsilon\Big(\sum_{j\in \mathcal{I}_{i'}} \omega_{j}^{*}
  \varphi_{j}(x)\Big) \ge G_0 \Big(\sum_{j\in \mathcal{I}_{i'}}
  \omega_{j}^{*} \varphi_{j}(x)\Big)= a^{*}_{i'}(x), \quad \forall~x \in
  \mathbb{X},~1 \le i'\le K\,.$$
  \endgroup
  Since $\bm{\omega}^{(T,\epsilon,\lambda)}$ is the minimizer, we can derive
  \begingroup
  \small
  \begin{align*}
    & ~~~   -\frac{1}{T}\ln
    \mathcal{L}^{(T,\epsilon)}\big(\bm{\omega}^{(T,\epsilon,\lambda)}\big) +
    \lambda \|\bm{\omega}^{(T,\epsilon,\lambda)}\|_1\\\
    &\le -\frac{1}{T}\ln
    \mathcal{L}^{(T,\epsilon)}\big(\bm{\omega}^{*}\big) +
    \lambda \|\bm{\omega}^{*}\|_1\\\
    &= \sum_{i'=1}^K
    \bigg[
    - \frac{1}{T}
    \int_0^T \ln G_\epsilon\bigg(\sum_{j\in \mathcal{I}_{i'}}
    \omega_{j}^{*}\,\varphi_{j}\big(X(s)\big)\bigg) dR_{i'}(s)
    + \frac{1}{T} \int_0^T G_\epsilon\Big(\sum_{j\in \mathcal{I}_{i'}}
    \omega_{j}^{*} \varphi_{j}\big(X(s)\big)\Big)\,ds\bigg] + \lambda \|\bm{\omega}^{*}\|_1\\
    &\le \sum_{i'=1}^K
    \bigg[
    - \frac{1}{T}
    \int_0^T \ln a_{i'}^{*}(X(s))\,dR_{i'}(s)
    + \frac{1}{T} \int_0^T G_\epsilon\Big(\sum_{j\in \mathcal{I}_{i'}}
    \omega_{j}^{*} \varphi_{j}\big(X(s)\big)\Big)\,ds\bigg] + \lambda
  \|\bm{\omega}^{*}\|_1\,.
  \end{align*}
  \endgroup
  Taking the limit $T\rightarrow +\infty$, using the fact
  $\lim\limits_{\epsilon \rightarrow 0} G_\epsilon = G_0$, as well as Lemma~\ref{lemma-integral-ri-limit} in Appendix~\ref{app-2}, we obtain
  \begingroup
  \small
  \begin{align}
    \limsup_{T\rightarrow +\infty} -\frac{1}{T}\ln \mathcal{L}^{(T,\epsilon)}\big(\bm{\omega}^{(T,\epsilon,\lambda)}\big) <
    +\infty\,, \quad a.s.
    \label{likelihood-finite-limisup}
  \end{align}
  \endgroup
  Now we show \eqref{omega-lower-bound} by contradiction. Suppose it does not hold, applying Lemma~\ref{lemma-g-eps} in Appendix~\ref{app-1}, we can find an index $i$, $1 \le
      i \le K$, and a state $x \in \mathbb{X}$ with $a_i^{*}(x)>0$, such that by extracting a
      subsequence, which will be again denoted by $\bm{\omega}^{(T,\epsilon,\lambda)}$, we have either
      \begingroup
      \small
   \begin{align}
     \begin{split}
       \lim_{T\rightarrow +\infty} G_\epsilon\Big(\sum_{j' \in \mathcal{I}_i} \omega^{(T,\epsilon, \lambda)}_{j'}
      \varphi_{j'}(x)\Big) = 0\,,\quad ~\mbox{or}~ 
      \lim_{T\rightarrow +\infty} G_\epsilon\Big(\sum_{j' \in \mathcal{I}_i} \omega^{(T,\epsilon, \lambda)}_{j'}
      \varphi_{j'}(x)\Big) = +\infty\,.
     \end{split}
      \label{i-x-limit-zero-or-infty}
   \end{align}
\endgroup
      Using \eqref{likelihood-t-eps}, we can estimate 
      \begingroup
      \small
\begin{align}
    &-\frac{1}{T}\ln\mathcal{L}^{(T, \epsilon)}(\bm{\omega}^{(T,\epsilon,
    \lambda)}) \label{log-likelihood-min-lb} \\
    &=  
    \sum_{i'=1}^K
    \bigg[
    - \frac{1}{T}
    \int_0^T \ln G_\epsilon\bigg(\sum_{j\in \mathcal{I}_{i'}}
    \omega_{j}^{(T,\epsilon, \lambda)}\,\varphi_{j}\big(X(s)\big)\bigg)
    dR_{i'}(s)
    + \frac{1}{T}
    \int_0^T 
     G_\epsilon\Big(\sum_{j\in \mathcal{I}_{i'}} \omega_{j}^{(T,\epsilon, \lambda)}
  \varphi_{j}\big(X(s)\big)\Big)\,ds\bigg] \notag\\
    &= 
    \sum_{i'=1}^K
    \sum_{x' \in \mathbb{X}} 
    \bigg[
      -
\ln G_\epsilon\bigg(\sum_{j\in \mathcal{I}_{i'}}
    \omega_{j}^{(T,\epsilon, \lambda)}\,\varphi_{j}(x')\bigg)
    \frac{1}{T}\int_0^T\mathbf{1}_{x'}(X(s))
    \,dR_{i'}(s) \notag\\
    &\quad + G_\epsilon\bigg(\sum_{j\in \mathcal{I}_{i'}}
\omega_{j}^{(T,\epsilon, \lambda)}\,\varphi_{j}(x')\bigg)\,\frac{1}{T}\int_0^T \mathbf{1}_{x'}(X(s))\,ds\bigg]\notag \\
    &\ge 
    \sum_{x' \neq x} 
    \sum_{i'=1}^K
    \bigg[
      -
      \bigg(
    \frac{1}{T}\int_0^T\mathbf{1}_{x'}(X(s))
    \,dR_{i'}(s)\bigg)
    \ln \frac{
      \frac{1}{T}\int_0^T\mathbf{1}_{x'}(X(s)) \,dR_{i'}(s)
    }{
    \frac{1}{T}\int_0^T \mathbf{1}_{x'}(X(s))\,ds 
    }
    + \frac{1}{T}\int_0^T\mathbf{1}_{x'}(X(s))\, dR_{i'}(s)
\bigg]\notag \\
&+ \sum_{1 \le i'\le K,i'\neq i}
    \bigg[
      -
      \bigg(
    \frac{1}{T}\int_0^T\mathbf{1}_{x}(X(s))
    \,dR_{i'}(s)\bigg)
    \ln \frac{
      \frac{1}{T}\int_0^T\mathbf{1}_{x}(X(s)) \,dR_{i'}(s)
    }{
    \frac{1}{T}\int_0^T \mathbf{1}_{x}(X(s))\,ds 
    }
    + \frac{1}{T}\int_0^T\mathbf{1}_{x}(X(s))\, dR_{i'}(s)
\bigg]\notag\\
& + 
   \bigg[
G_\epsilon\bigg(\sum_{j\in \mathcal{I}_{i}}
\omega_{j}^{(T,\epsilon, \lambda)}\,\varphi_{j}(x)\bigg)\,\frac{1}{T}\int_0^T
\mathbf{1}_{x}(X(s))\,ds
      -
\ln G_\epsilon\bigg(\sum_{j\in \mathcal{I}_{i}}
    \omega_{j}^{(T,\epsilon, \lambda)}\,\varphi_{j}(x)\bigg)
    \frac{1}{T}\int_0^T\mathbf{1}_{x}(X(s))
    \,dR_{i}(s)
\bigg] \notag \\
&=: J_1 + J_2 + J_3\,. \notag 
    \end{align}
    \endgroup
    where we have used Lemma~\ref{fun-c1logx-plus-c2x} below, as well as the
    convention $0\ln 0 = 0$. 
    Since $a^*_i(x) > 0$, applying
    Lemma~\ref{lemma-integral-ri-limit} in Appendix~\ref{app-2}, we have 
    \begingroup
    \small
    \begin{align*}
      \lim_{T\rightarrow +\infty} \frac{1}{T} \int_0^T\mathbf{1}_{x}(X(s))
      \,dR_{i} = a^*_i(x)\,\pi(x) > 0\,, \quad a.s.
    \end{align*}
    \endgroup
    Therefore, Lemma~\ref{fun-c1logx-plus-c2x} below implies that
    $\lim\limits_{T\rightarrow +\infty} J_3 = +\infty$\,, almost surely in the both cases in \eqref{i-x-limit-zero-or-infty}.
    For the same reason, applying Lemma~\ref{lemma-integral-ri-limit} in
    Appendix~\ref{app-2}, we know that 
    \begingroup
    \small
   \begin{align*}
     \begin{split}
       \lim_{T\rightarrow +\infty} J_1 &= 
     \sum_{x' \neq x} \pi(x') \sum_{i'=1}^K
     \Big(-a^*_{i'}(x') 
     \ln a^*_{i'}(x')
     + a^*_{i'}(x') \Big) > -\infty\,,  \\
       \lim_{T\rightarrow +\infty} J_2 &= 
     \pi(x) \sum_{1 \le i'\le K, i'\neq i}
     \Big(-a^*_{i'}(x) 
     \ln a^*_{i'}(x)
     + a^*_{i'}(x) \Big) > -\infty\,.
       \end{split}
   \end{align*}
   \endgroup
   Taking the limit $T\rightarrow +\infty$ in \eqref{log-likelihood-min-lb}, we obtain
     $\limsup\limits_{T\rightarrow +\infty} -\frac{1}{T}\ln \mathcal{L}^{(T,
     \epsilon)}(\bm{\omega}^{(T,\epsilon, \lambda)}) = +\infty$, which contradicts \eqref{likelihood-finite-limisup}. Therefore, \eqref{omega-lower-bound} has been proved.
   The boundedness of the sequence $\bm{\omega}^{(T,\epsilon,
   \lambda)}$ follows directly from \eqref{omega-lower-bound} and  Assumption~\ref{assump-boundedness}.
\end{proof}
 The following elementary facts have been used in the proof above.
\begin{lemma}
  Consider the function $f(x) = -c_1 \ln x + c_2\,x$, where $c_1 \ge 0, c_2 > 0$ are two constants. We have
  \begin{enumerate}[wide]
    \item
      $f(x)$ is convex on $(0,+\infty)$. 
    \item 
      $f(x) \ge -c_1 \ln \frac{c_1}{c_2} + c_1, ~\forall x \in (0,+\infty)$,
      and $\lim\limits_{x \rightarrow +\infty} f(x) = +\infty$. 
    \item
      When $c_1>0$, then $\lim\limits_{x \rightarrow 0+} f(x) = +\infty$.
  \end{enumerate}
  \label{fun-c1logx-plus-c2x}
\end{lemma}

Finally, we briefly present the proofs of Theorem~\ref{thm-omega-limit-general-lambda} and Theorem~\ref{thm-asymptotic-normality-lambda}, since the argument is similar to the one in Proposition~\ref{prop-omega-limit-general} and Proposition~\ref{prop-asymptotic-normality}. 
\begin{proof}[Proof of Theorem~\ref{thm-omega-limit-general-lambda}]
 Lemma~\ref{bounded-log-likelihood-imply-bounded-inf-sup} implies that the
  sequence $\bm{\omega}^{(T, \epsilon, \lambda)}$, $T > 0$, is bounded.
Let $\bar{\bm{\omega}}$ be a limit point of $\bm{\omega}^{(T, \epsilon, \lambda)}$ as $T\rightarrow +\infty$.
   Similar to \eqref{likelihood-t-eps}, let us define 
   \begingroup
   \small
\begin{equation*}
  -\ln\mathcal{L}^{(T), *} 
    =  -\sum_{l=0}^{M-1} \ln a^*_{i_l}(y_l) + \sum_{l=0}^{M} t_l\,a^*(y_l)\,.
\end{equation*}
\endgroup
  For any $\bm{\omega} \in \mathbb{R}^N$, using a similar derivation as \eqref{prop-omega-limit-general-proof-eqn1}, we can obtain 
  \begingroup
  \small
	  \begin{align}
	    \begin{split}
	    & ~~~   \lim_{T\rightarrow +\infty} \frac{
 \ln\mathcal{L}^{(T),*}- \ln \mathcal{L}^{(T,\epsilon)}(\bm{\omega})}{T}   \\
	      	    &=\sum_{x\in\mathbb{X}}
	\bigg[D_{KL}\Big(\psi^*(\cdot\,;\,x)\,\Big|\,
	  \psi^{(0)}\big(\cdot\,;\,x,\bm{\omega}\big)\Big) +
	D_{KL}\Big(p^*(\cdot\,;\,x)\,\Big|\,
      p^{(0)}\big(\cdot\,;\,x,\bm{\omega}\big)\Big)\bigg]
	a^*(x)\,\pi(x) \,,
	    \end{split}
	    \label{limit-log-l-t-eps-differ-omega}
	  \end{align}
	  \endgroup
	  as well as
	  \begingroup
	  \small
	  \begin{align}
	    \begin{split}
	    &~~~    \limsup_{T\rightarrow +\infty} 
	      \frac{ \ln\mathcal{L}^{(T),*}
- \ln \mathcal{L}^{(T,\epsilon)}(\bm{\omega}^{(T, \epsilon,\lambda)})}{T} \\
	      	    &\ge \sum_{x\in\mathbb{X}}
	\bigg[D_{KL}\Big(\psi^*(\cdot\,;\,x)\,\Big|\,
	  \psi^{(0)}\big(\cdot\,;\,x,\bar{\bm{\omega}}\big)\Big) +
	      D_{KL}\Big(p^*(\cdot\,;\,x)\,\Big|\,
	    p^{(0)}\big(\cdot\,;\,x,\bar{\bm{\omega}}\big)\Big)\bigg]
	a^*(x)\,\pi(x) \,.
	    \end{split}
	    \label{limit-log-l-t-eps-differ-bar}
	  \end{align}
	  \endgroup
	  Since $\bm{\omega}^{(T, \epsilon,\lambda)}$ is the minimizer of the problem \eqref{opt-problem-1}, we have
	  \begingroup
	  \small
	  \begin{align*}
	    &   \frac{1}{T}
	    \Big(-\ln\mathcal{L}^{(T,\epsilon)}(\bm{\omega}^{(T,
	    \epsilon,\lambda)}) +
	    \ln\mathcal{L}^{(T),*}\Big) +\lambda(T) \|\bm{\omega}^{(T, \epsilon,\lambda)}\|_1 
	    \le  \frac{1}{T}
	    \Big(-\ln\mathcal{L}^{(T,\epsilon)}(\bm{\omega}) +
	    \ln\mathcal{L}^{(T),*}\Big) +\lambda(T) \|\bm{\omega}\|_1\,.
	  \end{align*}
	  \endgroup
	  Taking the limit $T\rightarrow +\infty$ in the inequality above,
	  using \eqref{lambda-t-limit-is-zero},
	  \eqref{limit-log-l-t-eps-differ-omega} and
	  \eqref{limit-log-l-t-eps-differ-bar}, we obtain 
	  \begingroup
	  \small
	  \begin{align} \label{omega-bar-has-min-kl-divergence}
	      	    &~~~\sum_{x\in\mathbb{X}}
	\bigg[D_{KL}\Big(\psi^*(\cdot\,;\,x)\,\Big|\,
	  \psi^{(0)}\big(\cdot\,;\,x,\bar{\bm{\omega}}\big)\Big) +
	      D_{KL}\Big(p^*(\cdot\,;\,x)\,\Big|\,
	    p^{(0)}\big(\cdot\,;\,x,\bar{\bm{\omega}}\big)\Big)\bigg]
	a^*(x)\,\pi(x) \\
	    &\le
\sum_{x\in\mathbb{X}}
	\bigg[D_{KL}\Big(\psi^*(\cdot\,;\,x)\,\Big|\,
	  \psi^{(0)}\big(\cdot\,;\,x,\bm{\omega}\big)\Big) +
	D_{KL}\Big(p^*(\cdot\,;\,x)\,\Big|\,
      p^{(0)}\big(\cdot\,;\,x,\bm{\omega}\big)\Big)\bigg]
	a^*(x)\,\pi(x) \,, \quad \forall~\bm{\omega}\in\mathbb{R}^N\,.  \nonumber
	  \end{align}
	  \endgroup
	  In particular, choosing $\bm{\omega} = \bm{\omega}^{*}$, we get 
	 \begingroup 
	  \small
	  \begin{align*}
\sum_{x\in\mathbb{X}}
	    \bigg[D_{KL}\Big(\psi^{(0)}(\cdot\,;\,x, \bm{\omega}^{*})\,\Big|\,
	    \psi^{(0)}\big(\cdot\,;\,x,\bar{\bm{\omega}}\big)\Big) +
	    D_{KL}\Big(p^{(0)}(\cdot\,;\,x, \bm{\omega}^{*})\,\Big|\,
	    p^{(0)}\big(\cdot\,;\,x,\bar{\bm{\omega}}\big)\Big)\bigg]
	a^*(x)\,\pi(x)  = 0 \,,
	  \end{align*}
	  \endgroup
	  which implies
	    $a_i^{(0)}\big(x\,;\,\bar{\bm{\omega}}\big) =
	    a_i^{(0)}\big(x\,;\,\bm{\omega}^{*}\big)$\,,$\forall~ 1 \le i \le
	    K$ and $\forall~x \in \mathbb{X}$.
	  From the uniqueness of $\bm{\omega}^{*}$
	  (Assumption~\ref{assump-task2-omega-true}), we know $\bar{\bm{\omega}} = \bm{\omega}^{*}$ and therefore the convergence
$\lim\limits_{T\rightarrow +\infty} \bm{\omega}^{(T, \epsilon, \lambda)} = \bm{\omega}^{*}$ is obtained.
\end{proof}
\begin{proof}[Proof of Theorem~\ref{thm-asymptotic-normality-lambda}]
  First of all, the assumption \eqref{sqrt-t-times-lambda-t-limit-is-zero} implies
  $\lim\limits_{T\rightarrow +\infty} \lambda(T) = 0$. Therefore,
  Theorem~\ref{thm-omega-limit-general-lambda} assures the almost sure
  convergence of the sequence $\bm{\omega}^{(T, \epsilon, \lambda)}$ to $\bm{\omega}^{*}$.
  
  The same identity \eqref{taylor-expansion-m-j} still holds for 
    $\mathcal{M}^{(T,\epsilon)}_{j}$ and $\bm{\omega}^{(T,
    \epsilon, \lambda)}$ in the current setting. 
Similar to \eqref{w-talor-eqn-vector}, using the relation \eqref{euler-lagrange-lambda-vector},
   in the current case we can obtain
   \begingroup
   \small
    \begin{align*}
      \sqrt{T}\lambda(T) \bm{v}^{(T)} + \mathcal{W}^{(T)} = \mathcal{B}^{(T)} \Big[\sqrt{T}
      \big(\bm{\omega}^{(T, \epsilon,\lambda)} - \bm{\omega}^{*}\big)\Big]\,,
    \end{align*}
    \endgroup
    where the vector $\bm{v}^{(T)} \in -\partial|\bm{\omega}|(\bm{\omega}^{(T, \epsilon,\lambda)})$ is bounded, $\mathcal{W}^{(T)} \in \mathbb{R}^N$ is given by 
    \begingroup
    \small
   \begin{align}
     \begin{split}
     \mathcal{W}_j^{(T)} &=
        \frac{1}{\sqrt{T}}\int_0^T
       (\ln G_\epsilon)'\Big(
    \sum\limits_{j'\in \mathcal{I}_i}
       \omega^{*}_{j'}\,\varphi_{j'}(X(s))\Big)
     \varphi_{j}\big(X(s)\big)\,d\widetilde{R}_i(s) \\
       &~~~ - \frac{1}{\sqrt{T}}\int_0^T
       G_\epsilon'\Big(\sum\limits_{j'\in \mathcal{I}_i}
       \omega^{*}_{j'}\,\varphi_{j'}(X(s))\Big)
       \varphi_{j}\big(X(s)\big)\,\Big[1-\Big(\frac{G_0
       }{G_\epsilon}\Big)\big(\sum\limits_{j'\in \mathcal{I}_i}
       \omega^{*}_{j'}\,\varphi_{j'}(X(s))\big)\Big] ds\,, 
     \end{split}
  \label{w-lambda}
     \end{align}
     \endgroup
     for $1 \le j \le N$, and $\mathcal{B}^{(T)} \in \mathbb{R}^{N\times N}$ is given by 
     \begingroup
     \small
     \begin{align}
       \begin{split}
     \mathcal{B}_{j,j'}^{(T)} &= 
	-\frac{1}{T}\int_0^T 
	\Big[
	\int_0^1 
	(\ln G_\epsilon)''\Big(\sum\limits_{k\in \mathcal{I}_i}
       \big(\theta \omega_{k}^{(T,\epsilon, \lambda)}+ (1-\theta)
       \omega_{k}^{*}\big)\,\varphi_{k}(X(s))\Big)
       d\theta\Big]
\varphi_{j}\big(X(s)\big)\,\varphi_{j'}\big(X(s)\big) \,dR_i(s) \\
       & ~~~ + \frac{1}{T}\int_0^T 
	\Big[\int_0^1 G_\epsilon''\Big(\sum\limits_{k\in \mathcal{I}_i} 
	       \big(\theta \omega_{k}^{(T,\epsilon, \lambda)}+ (1-\theta)
	       \omega_{k}^{*}\big) \,\varphi_{k}(X(s))\Big)
	       d\theta\Big] \varphi_{j}\big(X(s)\big)\,\varphi_{j'}\big(X(s)\big) \,ds \\
\end{split}
  \label{b-lambda}
\end{align}
\endgroup
  if there is an index $i$, $1 \le i \le K$, such that $j,j' \in
  \mathcal{I}_i$, and otherwise $\mathcal{B}_{j,j'}^{(T)}=0$. See \eqref{exp-1st-derivative-r-eps} and \eqref{exp-2nd-derivative-r-eps}.

  Due to the second assumption in Theorem~\ref{thm-asymptotic-normality-lambda}, we can find another constant
  $c'>0$ such that $|\sum\limits_{j \in \mathcal{I}_i}\omega_{j}^{*}\varphi_j(x)| \ge c$,
  for all $x \in \mathbb{X}$ and $1 \le i \le K$, unless $\varphi_j(x)=0$ for
  all $j \in \mathcal{I}_i$. Applying Lemma~\ref{lemma-g-eps} in Appendix~\ref{app-1}, we have 
  \begin{align*}
    G_\epsilon'(x) \Big(1-\frac{G_0(x)}{G_\epsilon(x)}\Big) 
    < \begin{cases}
       \frac{\epsilon e^{-c'/\epsilon}}{c'}\,, & x \ge c' \,, \\
      e^{-c'/\epsilon}\,,& x \le -c'\,,
    \end{cases}
    \quad 
    \mbox{and}
    \quad G_\epsilon''(x) < \frac{1}{\epsilon} e^{-c'/\epsilon}, \quad \forall~|x|\ge c'\,.
  \end{align*}
  Since $\epsilon = \mathcal{O}(T^{-\alpha})$ and the functions $\varphi_j$
  are bounded, the second terms in the expressions of both
  $\mathcal{W}_{j}^{(T)}$ in \eqref{w-lambda} and $\mathcal{B}_{j,j'}^{(T)}$ in
  \eqref{b-lambda} converge to zero as $T\rightarrow +\infty$.
At the same time, Lemma~\ref{lemma-for-log-geps} in Appendix~\ref{app-1} implies
  that $\lim\limits_{\epsilon \rightarrow 0+} (\ln G_\epsilon)'(x) =\frac{1}{x}$ and 
$\lim\limits_{\epsilon \rightarrow 0+} (\ln G_\epsilon)''(x)=-\frac{1}{x^2}$, uniformly on $x \ge
c'>0$. Applying Lemma~\ref{lemma-integral-ri-ctl} in Appendix~\ref{app-2}, we know that the vector $\mathcal{W}^{(T)}$
   converges in distribution to a Gaussian random variable with zero mean and
   covariance matrix given by $\mathcal{F}$. 
  Since $\lim\limits_{T\rightarrow +\infty} \bm{\omega}^{(T,\epsilon, \lambda)} =
  \bm{\omega}^{*}$ almost surely,
  Lemma~\ref{lemma-integral-ri-limit} in Appendix~\ref{app-2} implies 
  $\lim\limits_{T\rightarrow \infty} \mathcal{B}^{(T)}  = 
  \mathcal{F}$, a.s.  Applying Slutsky's Theorem~\cite{ferguson1996course} and using the assumption
\eqref{sqrt-t-times-lambda-t-limit-is-zero}, we can conclude
  \begingroup
\small
\begin{align*}
  \sqrt{T} \big(\bm{\omega}^{(T, \epsilon, \lambda)} - \bm{\omega}^{*}\big)
  = (\mathcal{B}^{(T)})^{-1} \big(\mathcal{W}^{(T)} + \sqrt{T}\lambda(T)
  \bm{v}^{(T)}\big)
  \Longrightarrow \mathcal{Z} \in
  \mathcal{N}\big(0, \mathcal{F}^{-1}\big)\,, 
  ~\mbox{as}~~ T\rightarrow \infty\,.
\end{align*}
\endgroup
\end{proof}
\begin{remark}
  The second assumption in Theorem~\ref{thm-asymptotic-normality-lambda} is used to handle the
  second terms of $\mathcal{W}_j^{(T)}$ in~\eqref{w-lambda} and $\mathcal{B}_{j,j'}^{(T)}$ in \eqref{b-lambda}. It is not needed if
  $a^*_i(x)>0$ for all $x\in \mathbb{X}$ and $1 \le i \le K$.
  \label{rmk-explain-assump-in-thm}
\end{remark}

\section*{Acknowledgements}
This work is funded by the Deutsche Forschungsgemeinschaft (DFG, German Research Foundation) under Germany's Excellence Strategy --- The Berlin Mathematics Research Center MATH+ (EXC-2046/1, project ID: 390685689). The authors also acknowledge financial support from the Einstein Center of Mathematics (ECMath) through project CH21.

\bibliographystyle{abbrv}
\bibliography{references}

\end{document}